\documentclass[preprint,12pt]{elsarticle}
 
\usepackage{graphicx}
\usepackage{subfigure}
\usepackage{amsmath,amssymb,amsthm}
\usepackage{color}
\usepackage{framed}
\usepackage{url}

\usepackage{algorithm}
\usepackage{algorithmic}

\newcommand{\NRR}{\textup{NRR}}

\newcommand{\PK}{\textup{P}_{\textup{K}}}

\newcommand{\ut}{\underline{\boldsymbol{t}}}
\newcommand{\ux}{\underline{\boldsymbol{x}}}
\newcommand{\us}{\underline{\boldsymbol{s}}}
\newcommand{\tus}{\underline{\tilde{\boldsymbol{s}}}}
\newcommand{\uq}{\underline{\boldsymbol{q}}}
\newcommand{\vY}{\boldsymbol{Y}}

\newcommand{\NN}{\mathbb{N}}
\newcommand{\ZZ}{\mathbb{Z}}

\newcommand{\RR}{\mathbb{R}}
\newcommand{\CC}{\mathbb{C}}

\newcommand{\ud}{\textup{d}}

\newcommand{\argmax}{\operatornamewithlimits{argmax}}

\newtheorem{theorem}{Theorem}[section]
\newtheorem{lemma}[theorem]{Lemma}

\newtheorem{corollary}[theorem]{Corollary}

\newenvironment{definition}[1][Definition]{\vspace{.4cm}\begin{trivlist}
\item[\hskip \labelsep {\bfseries #1}]}{\end{trivlist}\vspace{.4cm}}
\newenvironment{example}[1][Example]{\begin{trivlist}
\item[\hskip \labelsep {\bfseries #1}]}{\end{trivlist}}
\newenvironment{remark}[1][Remark]{\vspace{.4cm}\begin{trivlist}
\item[\hskip \labelsep {\bfseries #1}]}{\end{trivlist}\vspace{.4cm}}

\begin{document}

\begin{frontmatter}

\title{Real-time dynamics acquisition from irregular samples -- with application to anesthesia evaluation} 

\author{Charles K. Chui}
\address{Department of Statistics, Stanford University, Stanford, CA 94305, USA (\url{ckchui@stanford.edu})}

\author{Yu-Ting Lin}
\address{Department of Anesthesiology, Shin Kong Wu Ho-Su Memorial Hospital, Taipei, Taiwan;\, Graduate Institute of Biomedical Electronics and Bioinformatics, National Taiwan University, Taipei, Taiwan
 (\url{linyuting@hotmail.com.tw})}

\author{Hau-tieng Wu}
\address{Department of Mathematics, Stanford University, Stanford, CA 94305, USA (\url{hauwu@stanford.edu})}


\begin{abstract}
Although digital representations of information sources are ubiquitous in our daily life, almost all digital information sources are regular samples (or obtained by uniform sampling) of some continuous function representations. However, there are many important events for which only irregular data samples are available, including trading data of the financial market and various clinical data, such as the respiration signals hidden in ECG measurements. For such digital information sources, the only available effective smooth function interpolation scheme for digital-to-analog (D/A) conversion algorithms are mainly for off-line applications. Hence, in order to adapt the powerful continuous-function mathematical approaches for real-time applications, it is necessary to introduce an effective D/A conversion scheme as well as to modify the desired continuous-function mathematical method for on-line implementation. The powerful signal processing tool to be discussed in this paper is the synchrosqueezing transform (SST), which requires computation of the continuous wavelet transform (CWT), as well as its derivative, of the analog signal of interest. An important application of this transform is to extract information, such as the underlying dynamics, hidden in the signal representation. The first objective of this paper is to introduce a unified approach to remove the two main obstacles for adapting the SST approach to irregular data samples in order to allow online computation. Firstly, for D/A conversion, a real-time algorithm, based on spline functions of arbitrarily desired order, is proposed to interpolate the irregular data samples, while preserving all polynomials of the same spline order, with assured maximum order of approximation. Secondly, for real-time dynamic information extraction from an oscillatory signal via SST, a family of vanishing-moment and minimum-supported spline-wavelets (to be called VM wavelets) are introduced for on-line computation of the CWT and its derivative. The second objective of this paper is to apply the proposed real-time algorithm and VM wavelets to clinical applications, particularly to the study of the ``anesthetic depth'' of a patient during surgery, with emphasis on analyzing two dynamic quantities: the ``instantaneous frequencies'' and the ``non-rhythmic to rhythmic ratios'' of the patient's respiration, based on a one-lead electrocardiogram (ECG) signal. Indeed, the ``R-peaks'' of the ECG signal, which constitute a waveform landmark for clinical evaluation, are non-uniform samples of the respiratory signal. It is envisioned that the proposed algorithm and VM wavelets should enable real-time monitoring of ``anesthetic depth'', during surgery, from the respiration signal via ECG measurement.
\end{abstract}

\end{frontmatter}

\section{Introduction}

Digital representation of time series, signals, images, videos, and other information sources is ubiquitous in our daily life. In almost all application areas, digital representation is achieved by sampling or discretization of a continuous function. For efficiency, uniform sampling is used to give regular data samples. This means that sampling is carried out at uniformly (or equally) spaced points of the time or spatial domain. On the other hand, there are many important events for which only irregular data samples are available. These include trading data of the financial market, package data traffic \cite{Gunnarsson_Gustafsson_Gunnarsson:2004,Jacobsson_Andrew_Tang_Low_Hjalmarsson:2009}, measurement of star luminosity \cite{Vio_Strohmer_Wamsteker:2000} (due to weather condition or day/night difference), and clinical data, such as respiration signals (hidden in ECG measurements) \cite{moody_mark:1986}. Other irregular data samples are introduced by non-uniform sampling to meet certain requirements, such as speeding up MRI (magnetic resonance image) acquisition [38] and increasing performance of radar systems to avoid jamming \cite{Chen_Yaidyanathan:2008} and to minimize aliasing \cite{Ait-Sahalia_Mykland:2003,Bilinskis_Mikelson:1992}.

        In any case, for such application areas as mentioned above, we need effective methods to process non-uniform data samples in order to obtain the information of interest, and more importantly, to be able to apply the powerful tools that apply only to analog data representations. In other words, we need an effective D/A (digital to analog) conversion tool. Of course there are various D/A methods for non-uniform data samples introduced for different purposes that are available in the literature; for example, finding a proper building block under certain reconstruction requirements \cite{deBoor:1978,Benedetto:1992,Chui:1992,Eldar:2003}, applying non-parametric regression to irregular data fitting \cite{Fan_Gijbels:1996}, and iterative reconstruction algorithms or projection onto suitable convex sets \cite{Marvasti:2001,Gurin_Polyak_Raik:1967}, just to name a few. However, all of these and other approaches, including recovery of the continuous-time function from its Fourier transform, estimated directly from the irregular sample \cite{Lomb:1976,Moody:1993}, are not suitable for D/A conversion in a real-time or online system. Examples include real-time warning of a network under attack from non-uniformly sampled packet data traffic, and for the clinical application to be studied in this paper, real-time warning if a patient is entering a critical condition. 

A powerful signal processing tool that applies only to continuous-time signals is the synchrosqueezed continuous wavelet transform (SST), which requires knowledge of the continuous wavelet transform (CWT) of the signal \cite{Daubechies_Lu_Wu:2011}.  An important application of this transform is to extract information, such as the underlying dynamics, hidden in the signal representation \cite{Daubechies_Lu_Wu:2011,Chen_Cheng_Wu:2013}. The main benefit of choosing CWT is its polynomial cancelation property, which allows us to handle the trend in real time. However, there are two obstacles for real-time application of the SST. Firstly, an effective algorithm for D/A conversion of non-uniform samples is needed. Although a partial off-line solution to meet this challenge is proposed in \cite{Thakur_Wu:2011, Wu:2011Thesis} based on the synchrosqueezed short-time Fourier transform (STFT) of non-uniform data samples, there are various short-comings of this direct approach. In particular, due to the lack of sufficient vanishing moments, the possible trend that commonly exists in the signal cannot be eliminated. On the other hand, in order to achieve precise spectral information, the analysis wavelet for the CWT is assumed to have compactly supported Fourier transform \cite{Daubechies_Lu_Wu:2011}, and consequently the SST cannot be realized in real-time.  Of course an obvious suggestion is to replace the analysis wavelet by a compactly supported wavelet. However, the obstacle remains in that for real-time realization, the CWT has to be taken on a bounded time-interval, and that a high order of vanishing moments of the analysis wavelet is required for reasonably good performance of the SST, which in turn requires larger support of the wavelet.

        One of the main objectives of this paper is to introduce a unified approach to remove, or at least minimize, the two obstacles mentioned above. Firstly, for D/A conversion, a real-time algorithm, based on spline functions of arbitrarily desired order, is proposed to interpolate the non-uniform data, while preserving all polynomials of the same order, with assured maximum order of approximation. Secondly, for real-time dynamic information extraction from an oscillatory signal via SST, we introduce a family of analysis wavelets with minimum support in the time-domain, again in terms of B-splines of arbitrarily desired order, for both the interior and the boundary (i.e. at the end-points) of any bounded interval, with arbitrarily desired order of vanishing moments. We coin these analysis wavelets as VM wavelets (for maximum vanishing moments and minimum support). We give explicit formulas of the VM wavelets on arbitrarily desired knot sequences, and prove that for the VM wavelets $\psi_{m,n}$, in terms of the $m^{th}$ Cardinal B-splines (i.e. on integer knots) and with $n^{th}$ order vanishing moments, a suitable scaling of the centered wavelets $\psi_{m,n}$ is asymptotically the same as the $n^{th}$ derivative of the Gaussian function, as $m+n$ tends to infinity. With a VM wavelet as the mother wavelet of the CWT, computation of its derivative can be eliminated by replacing the CWT with a companion VM wavelet.

To demonstrate the usefulness of the proposed real-time algorithm and VM wavelets is applied to clinical applications, in the study of the anesthetic depth of a patient during surgery. More precisely, we will focus on analyzing two dynamic quantities: the instantaneous frequencies and the non-rhythmic to rhythmic ratios of the patient's respiration, by analyzing solely the one-lead electrocardiogram (ECG) signal. Indeed, the R-peaks of the ECG signal, that constitute a waveform landmark, are non-uniform samples of the respiratory signal. The proposed algorithm and VM wavelets enable us to study the possibility of real-time monitoring the anesthetic depth during surgery, from the respiration signal via ECG measurement.

        This paper is organized as follows. In Section \ref{section:blending}, the theoretical background of B-splines on arbitrary knot sequences, quasi-interpolation operation, complete local spline interpolation, and the blending operator are reviewed; a real-time algorithm for spline interpolation of non-uniform data samples is formulated; and the order of maximum order of approximation by the real-time interpolation algorithm is derived. The notion of VM wavelets is introduced in Section \ref{section:VMwavelet}, where both interior and boundary wavelets with minimum support are derived, with explicit formulas, even on arbitrary knot sequences. Furthermore, the performance of the VM wavelets is analyzed in terms of the side-lobe/main lobe ratios of their powers spectra; and their asymptotic behavior, as compared with derivatives of the Gaussian function, is also derived in Section \ref{section:VMwavelet}. In Section \ref{section:tvPS}, the synchrosqueezing transform (SST) is reviewed and the notion of time-varying power spectrum (tvPS) is discussed, with the goal of extracting dynamic information from an oscillating signal. Section \ref{section:testbed2} is devoted to the study of clinical applications of the proposed algorithm. In particular, the {\it blending ECG derived respiration (EDR) algorithm} is proposed to extract the respiratory dynamics from the single ECG signal in real-time and to study the relationship between anesthetic depth and respiratory dynamical features extracted from analyzing the morphology of the single lead ECG signal. 
\newline\newline
\textbf{Notation:} In this paper, the Fourier transform of a function $h$ in $L^2(\RR)$ is defined as $\hat{h}(\xi):=\int_{-\infty}^\infty h(x)e^{-ix\xi}\ud x$.

\section{Optimal real-time spline interpolation}\label{section:blending}

Let $\{x_n\}$, $n=0,1,\ldots$, be a time series contaminated with noise, where 
\begin{align}
x_n=g(t_n)+\epsilon_n,
\end{align}
is the non-uniform measurement of an underlying signal $g(t)$ at the time positions $t=t_n$. An important problem in signal processing is to convert $\{x_n\}$ on-line to an analog signal $\tilde{g}(t)$, where $t\geq t_0$, under certain requirements and having certain desirable properties. Here, $\{\epsilon_n\}$ is a sequence of additive noise, including instrumental noise, electoral contact noise, baseline drift/motion noise, and statistical random noise. Since there has been a lot of study on detection and removal of deterministic noise, it is safe to assume that $\{\epsilon_n\}$ is a random noise process. 

For (statistical) white noise $\{\epsilon_n\}$ with variance $\sigma^2>0$, the most efficient real-time algorithm to optimal recover $\{g(t_n)\}$ from $\{x_n\}$ is Kalman filtering. In a nutshell, the Kalman noise-removal filter is a prediction-correction algorithm that corrects the prediction $\hat{x}_{n,n-1}$ to yield the optimal estimation $\hat{x}_n=\hat{x}_{n,n}$ of $x_n$, in the sense that $\{\hat{x}_k\}_{k=0}^n$ is the weighted least-squares fit of the measurement $\{x_k\}_{k=0}^n$, with $1/\sigma^2$ as the weight. The interested reader is referred to \cite{Chui_Chen:2009}, where white noise is also extended to color noise.

In this paper, we are only concerned with real-time conversion of a ``clean'' digital signal to spline function representation. In other words, we will assume that $\hat{x}_n=x_n=g(t_n)$, for all $n=0,1,\ldots$. Let us first pose the problem and introduce some notations. The digital samples $g(t_n)$ are taken at
\begin{align}\label{nonuniform:definition:t_knot}
\ut:\, a=t_0<t_1<t_2<\ldots,
\end{align}
where $\{t_k\}$ may be irregular (or non-uniform), meaning that we allow $t_{k+1}-t_k\neq t_{j+1}-t_j$ for $k\neq j$. Let $m\geq 3$ be any desired integer and $\Pi_{m-1}$ denote the space of all polynomials of degree less than $m$. 

\begin{definition}{[Spline space]}
For any sequence 
\begin{align}\label{nonuniform:definition:s_knot}
\us:\,s_{-m+1}=s_{-m+2}=\ldots=a=s_0<s_1<s_2<\ldots,
\end{align}
we use the notation $S_{\us,m}:=S_{\us,m}[a,\infty)$ to denote the spline space of order $m$ with knot sequence $\us$; that is, $f\in S_{\us,m}$ if and only if $f\in C^{m-2}[0,\infty)\quad\mbox{and}\quad f|_{[s_k,s_{k+1}]}\in \Pi_{m-1}$ for all $k=0,1,\ldots$. 
\end{definition}

To formulate a locally supported basis of $S_{\us,m}$, we first introduce the notion of {\it truncated powers} 
\[
x_+^{m-1}:=(\max\{0,x\})^{m-1}.
\]
Then it is clear that the truncated power functions $(s_k-t)^{m-1}_+$, $k=0,1,\ldots$, are in $S_{\us,m}$. Since $(x_k-t)_+^{m-1}$, $k\geq 1$, have global support, we apply the $m$-th order divided differences to change them to locally supported functions. In general, the {\it divided differences} are defined by
\[
[\,u,\ldots,u\,]f:=\frac{f^{(l)}(u)}{l!}
\]
if there are $l+1$ entries in $[u,\ldots,u]$, and 
\[
[\,u_0,\ldots,u_n\,]f:=\frac{[\,u_1,\ldots,u_n\,]f-[\,u_0,\ldots,u_{n-1}\,]f}{u_n-u_0}
\]
if $u_0\leq u_1\leq\ldots\leq u_n$ and $u_n>u_0$. 

\begin{definition}{[B-spline \cite{deBoor:1978}]}
Fix $m>0$ and consider the knot sequence $\us$ of the spline space $S_{\us,m}$. The {\it normalized B-splines} are defined by applying 
the divided difference operation to the truncated power $(s_k-t)^{m-1}_+$, namely,
\begin{align}\label{definition:Nmktuniform}
N_{m,k}(t)=N_{\us,m,k}(t)=(s_{m+k}-s_k)[\,s_k,\ldots,s_{m+k}\,](\cdot - t)_+^{m-1},
\end{align}
for $k=-m+1,\ldots,0,1,2,\ldots$. 
\end{definition}
Numerically, we would apply the recursive formula to implement $N_{m,k}$ (see, for example, \cite[page 143 (6.6.12a)]{Chui:1997}).
It is also shown in \cite{deBoor:1978} that $\{N_{\us,m,k}\}$, $k=-m+1,\ldots,0,1,2,\ldots$ is a stable locally supported basis of the space $S_{\us,m}$ in that every function $f\in S_{\us,m}\cap L^\infty[a,\infty)$ has a unique spline series representation with coefficient sequence $\{c_k\}\in\ell^\infty$. Moreover, if $\delta_m:=\inf_k(t_{m+k}-t_k)>0$, then
\begin{align}
c(m,\delta_m)\|\{c_k\}\|_{\ell^\infty}\leq \left\| \sum_{k=-m+1}^\infty c_kN_{\us,m,k}(t)\right\|_{L^\infty[a,\infty)}\leq \|\{c_k\}\|_{\ell^\infty},
\end{align}
where $c(m,\delta_m)>0$ is a constant depending only on $m$ and the ``minimum knot spacing'' $\delta_m$. 

However, if we choose the sequence $\ut$ of time positions in (\ref{nonuniform:definition:t_knot}) as the knot sequences $\us$ in (\ref{nonuniform:definition:s_knot}) by attaching $t_{-m+1}=\ldots=t_0=a$ to $\ut$ as in (\ref{nonuniform:definition:s_knot}), there are possibly two complications. Firstly, for the spline space $S_{\ut,m}$ of odd order $m\geq 5$, computation of the spline interpolant 
\[
f(t)=\sum_{k}c_kN_{\ut,m,k}(t)
\]
of the ``clean'' non-uniform samples $\hat{x}_k=g(t_k)$, where $k=0,1,\ldots,n$, could be unstable, since the coefficient matrix
\[
A_{n+1}:=[N_{\ut,m,k}(t_j)]_{j=0,\ldots,n;\,k=-m+1,\ldots,n-m+1}
\]
may not be diagonal dominant. Secondly, for even order $m\geq 4$, although the matrix $A_{n+1}$ is usually diagonal dominant, the computational cost to obtain the coefficients $c_{-m+1},c_{-m+2},\ldots,c_{n-m+1}$ is high for non-uniform $\{t_j\}$ and large values of $n$. 

On the other hand, if the spline representer $f(t)$ is not required to interpolate the target data function $g(t)$ at $t=t_j$, $j=0,1,\ldots,n$ (that is, if $f(t_j)\neq g(t_j)$ is allowed), then the ``quasi-interpolation'' scheme introduced by de Boor and Fix \cite{deBoor_Fix:1973} meets the requirement of ``polynomial preservation'', in that for $g(t)=p(t)\in \Pi_{m-1}$ in a desired neighborhoods of some time sample $t^*$, then $f(t)=p(t)$ in this neighborhood. Polynomial preservation is important to facilitate the continuous wavelet transform (CWT), and hence the synchrosqueezing transform (SST) for oscillatory components separation and analysis, such as analysis of the respiratory signal from the ECG signal as will be discussed in Section \ref{section:tvPS} in this paper. In particular, if the analysis wavelet is also a compactly supported spline of order $m$, then the $m$-th order vanishing moment annihilates the $m$-th order Taylor polynomial expansion of $g(t)$, that facilitates in removing the trend of the signal $g(t)$. However, the quasi-interpolation scheme in \cite{deBoor_Fix:1973,deBoor:1978} requires derivative data values of $g(t)$ that are not available in our study. In \cite{Lyche_Schumake:1975} and \cite{Schumaker:1981}, derivatives of $g(t)$ are replaced by divided differences of $\{g(t_i)\}$, for our need. Other quasi-interpolation scheme including a comprehensive study in \cite{deVilliers:2012} have also been derived. But to the best of our knowledge, only the basis functions introduced in our work \cite{Chen_Chui_Lai:1988} have the real-time approximation feature. As mentioned above, quasi-interpolation introduced the error $g(t_i)-f(t_i)$ for $i=0,1,\ldots,n$. In \cite{Chui_Diamond:1991}, the notion of local interpolation is introduced to correct the errors, leading to the ``blending operator'', to be discussed in the following sub-section.

\subsection{Blending operator}
Let us first study the real-time quasi-interpolant introduced in \cite{Chen_Chui_Lai:1988}. We start from preparing some notations. For each $k=0,1,\ldots$, consider the Vandermonde determinant
\[
D(t_k,\ldots,t_{k+m+1}):=\text{det}\left[
\begin{array}{cccc}
1 & 1 & \ldots & 1\\
t_k & t_{k+1} & \ldots & t_{k+m-1}\\
\vdots & \vdots & \vdots &\vdots \\
t^{m-1}_k & t^{m-1}_{k+1} & \ldots & t^{m-1}_{k+m-1}
\end{array}
\right]
\]
and the determinant $D(t_k,\ldots, t_{k+j-1},\xi_j,t_{k+j+1},\ldots,t_k+m-1)$ obtained by replacing the $(j+1)$-st column in the definition of $D(t_k,\ldots,t_{k+m+1})$ by the column vector
\[
\xi_j:=[\,\xi^0(j,m),\ldots,\xi^{m-1}(j,m)\,]^T,
\]
where $\xi^0(j,m)=1$ and 
\[
\xi^i(j,m)=\frac{\sigma^i(t_{j+1},\ldots,t_{j+m-1})}{\left(\begin{array}{c}m-1\\ i\end{array}\right)}
\]
for $i=1,\ldots,m-1$ with $\sigma^i(r_1,\ldots,r_{m-1})$ being the classical symmetric functions defined by $\sigma^0(r_1,\ldots,r_{m-1})=1$ and for $i=1,\ldots,m-1$,
\[
\sigma^i(r_1,\ldots,r_{m-1})=\sum_{1\leq l_1<\ldots<l_i\leq m-1} r_{l_1}\ldots r_{l_i}.
\]

\begin{definition}{[Quasi-interpolation operator]}
Using the determinants introduced above, we apply the spline coefficients
\begin{align*}
a_{k,l}:=\frac{D(t_k,\ldots, t_{k+j-1},\xi_j,t_{k+j+1},\ldots,t_{k+m-1})}{D(t_k,\ldots,t_{k+m-1})}
\end{align*}
to formulate the compactly supported spline function
\begin{align*}
M_{\ut,m,k}(t):=\sum_{l=m-1}^{2m-2}a_{k,l-m+1}N_{\ut,m,k+l-m+1}(t)
\end{align*}
with $\text{supp}M_{\ut,m,k}=[t_{k-m+1},t_{k+m}]$. These basis functions provide a real-time implementation of the {\it quasi-interpolation operator}
\begin{align}\label{nonuniform:Q_m:definition}
(\mathsf{Q}_mg)(t)=\sum_{k}g(t_k)M_{\ut,m,k}(t).
\end{align}
\end{definition}

We summarize some of the properties of the quasi-interpolation operator $\mathsf{Q}_m$ in (\ref{nonuniform:Q_m:definition}).
\begin{lemma}
For any $m\geq 1$, $\mathsf{Q}_m$ possesses the polynomial preservation property
\[
(\mathsf{Q}_mp)(t)=p(t)
\]
for all $t\geq a$ and for all $p\in\Pi_{m-1}$, provided that the summation in (\ref{nonuniform:Q_m:definition}) is taken for all non-negative integers $k=0,1,\ldots$. Furthermore, in view of the support of $M_{\ut,m,k}$, it follows that
\[
\sum_{k=v-m+1}^{v+m-2}p(t_k)M_{\ut,m,k}(t)=p(t),\quad t\in[t_u,t_v]
\]
for all $p\in \Pi_{m-1}$. 
\end{lemma}
This local polynomial preservation property allows the CWT, and hence the SST, to annihilate the $(m-1)$-th degree Taylor polynomial approximation of the signal at $t_j$, where $u<j<v$.

We now turn to the discussion of the local interpolation operator, to be denoted by $\mathsf{R}_m$, which satisfies the interpolation property. To define $\mathsf{R}_m$, we may insert knots to $\ut$ by considering a new knot sequence $\us\supset\ut$ as in (\ref{nonuniform:definition:s_knot}). More precisely, we consider even and odd orders separately, as follows:

\vspace{0.4cm}
$\boldsymbol{(i)}$  For even $m\geq 4$, we set
\[
s_{mk/2}=t_k,\quad k=0,1,2,\ldots.
\]
That is, we insert $(m/2-1)$ knots in between two consecutive knots in $\ut$. For convenience, we may choose the new knots equally spaced in-between every pair of two consecutive knots. 

\begin{definition}{[Completely local spline basis function]}
Fix even $m\geq 4$. Let $N_{\us,m,j}$ be the $m$-th order B-spline with knot sequence $\us$. Then the {\it completely local spline basis function} can be defined by
\begin{align*}
L_{\us,m,j}(t):=\frac{N_{\us,m,m(j-1)/2}(t)}{N_{\us,m,m(j-1)/2}(t_j)}.
\end{align*}
\end{definition}

Since $t_j=s_{mj/2}$ is the ``centered'' knot and
\[
\text{supp}N_{\us,m,m(j-1)/2}=[\,s_{m(j-1)/2},s_{m(j+1)/2}\,]=[\,t_{j-1},t_{j+1}\,],
\]
we have
\begin{align}\label{nonuniform:property:Lmsj}
\left\{
\begin{array}{l}
\displaystyle L_{\us,m,j}(t_j)=1\\ \\
\displaystyle\text{supp}L_{\us,m,j}=[\,t_{j-1},t_{j+1}\,].
\end{array}
\right.
\end{align}
In view of (\ref{nonuniform:property:Lmsj}), it is clear that $L_{\us,m,j}(t_k)=\delta_{j-k}$, where the Kronecker delta notation is used. 
\vspace{0.4cm}

$\boldsymbol{(ii)}$ For odd $m\geq 3$, we may insert $\frac{m+1}{2}-1$ equally spaced new knots in between $[t_{2k},t_{2k+1}]$, and $\frac{m+1}{2}-2$ equally spaced new knows in between $[t_{2k+1},t_{2k+2}]$, for $k=0,1,\ldots$. Then, by considering even and odd indices ($j=2k$ and $j=2k+1$) separately, a similar construction as the even order setting yields completely local spline basis functions $L_{\us,m,j}(t)$ that has the property (\ref{nonuniform:property:Lmsj}) as the even order $m$.
\newline\newline
The above preparation provides a real-time implementation of the {\it local interpolation operator}, which satisfies the interpolation property due to (\ref{nonuniform:property:Lmsj}).

\begin{definition}{[Local interpolation operator]}
Fix $m\geq3$. For a given function $g\in C(\RR)$, the local interpolation operator $\mathsf{R}_m$ is defined by
\begin{align}\label{nonuniform:R_m:definition}
(\mathsf{R}_mg)(t):=\sum_k g(t_k) L_{\us,m,k}(t).
\end{align}
\end{definition}

We are now ready to apply (\ref{nonuniform:Q_m:definition}) to obtain the blending operator, denoted by $\mathsf{R}_m\oplus \mathsf{Q}_m$. 

\begin{definition}{[Blending operator]}
Fix $m\geq 3$ and $g\in C(\RR)$. The blending operator is defined as $\mathsf{P}_m:=\mathsf{R}_m\oplus \mathsf{Q}_m$, where
\begin{align*}
\mathsf{R}_m\oplus \mathsf{Q}_m:=\mathsf{Q}_m+\mathsf{R}_m(\mathsf{I}-\mathsf{Q}_m)=\mathsf{Q}_m+\mathsf{R}_m-\mathsf{R}_m\mathsf{Q}_m,
\end{align*}
and $\mathsf{I}$ is the identity operator. In particular, we have 
\begin{align}\label{nonuniform:definition:Pm}
(\mathsf{P}_mg)(t):=\sum_{k}g(t_k)M_{\ut,m,k}(t)+\sum_k\big[g(t_k)-\sum_j g(t_j)M_{\ut,m,j}(t_k)\big]L_{\us,m,k}(t).
\end{align}
\end{definition}

We remark that in the definition of $\mathsf{P}_m$, the two operators $\mathsf{R}_m$ and $\mathsf{Q}_m$ are not commutative. Let us summarize the two key properties of the blending operator in the following theorem.
\begin{theorem}
The blending operator $\mathsf{P}_m$ possesses both the polynomial preservation property of $\mathsf{Q}_m$ and the interpolatory property of $\mathsf{R}_m$.
\end{theorem} 
\begin{proof}
Indeed, for any $p\in \Pi_{m-1}$, since $\mathsf{Q}_mp=p$, we have
\[
(\mathsf{R}_m\oplus \mathsf{Q}_m)p=\mathsf{Q}_mp+\mathsf{R}_mp-\mathsf{R}_m(\mathsf{Q}_mp)=p+\mathsf{R}_mp-\mathsf{R}_mp=p.
\]
Also, $f:=(\mathsf{R}_m\oplus \mathsf{Q}_m)g\in S_{\ut,m}$ interpolates the discrete data $\{g(t_j)\}$ in that
\[
f(t_j)=(\mathsf{Q}_mg)(t_j)+(\mathsf{R}_mg)(t_j)-\mathsf{R}_m(\mathsf{Q}_mg)(t_j)=(\mathsf{Q}_mg)(t_j)+g(t_j)-(\mathsf{Q}_mg)(t_j)=g(t_j).
\]
Observe that the definition of $(\mathsf{P}_mg)(t)$ in (\ref{nonuniform:definition:Pm}) depends only on the discrete samples $g(t_k)$, $k=0,1,\ldots$ of $g(t)$ and that
\begin{align*}
f(t):=(\mathsf{P}_mg)(t)\in S_{\us,m}[0,\infty)
\end{align*}
since $\ut\subset \us$. As already discussed earlier, we have, firstly, the interpolatory property
\[
f(t_j)=g(t_j),\quad j=0,1,2,\ldots
\]
and secondly, the polynomial preservation property
\begin{align}\label{nonuniform:Pm_property2}
(\mathsf{P}_mp)(t)=p(t),\quad t\geq a,
\end{align}
for all $p\in \Pi_{m-1}$. In fact, if only $t\in[t_u,t_v]$ in (\ref{nonuniform:Pm_property2}) is desired, then the summation over $k$ in (\ref{nonuniform:definition:Pm}) is taken from $k=u-m+1$ to $k=v+m-2$, due to the finite support property of $M_{\ut,m,k}(t)$ and $L_{\us,m,k}(t)$. 
\end{proof}

To study how well the blending operator approximates functions that are not polynomials, we consider only knot sequences $\ut$ that satisfies the condition
\begin{align}\label{blending:proof:condition}
\frac{\sup_{0\leq i\leq k}|t_{k+i+m}-t_{k+i}|}{\inf_{0\leq i\leq k}|t_{k+i+1}-t_{k+i}|}\leq C_{\ut,k}
\end{align}
for some positive constant $C_{\ut,k}$, where $m>0$ is the order of the B-splines $N_{\ut,m,k}$ and $k\geq 0$ is the knot index. Under this condition, it has been shown in \cite[Lemma 2,1]{Chen_Chui_Lai:1988} that the coefficients $a_{k,j}$ in the definition of the quasi-interpolation operator (\ref{nonuniform:Q_m:definition}) (with $g(t_j)$ replaced by $a_{k,j}$) satisfy  
\[
|a_{k,j}|\leq \frac{1}{(m-1)!}C_{\ut,k}.
\]
Since the B-splines $N_{\ut,m,j}$ are non-negative and constitute a partition of unity, we have
\begin{align}\label{blending:proof:MtmkBound}
|M_{\ut,m,k}(t)|\leq \frac{1}{(m-1)!}C_{\ut,k}^{m-1}\sum_{l=m-1}^{2m-2}N_{\ut, m,k+l-m+1}(t)\leq \frac{1}{(m-1)!}C_{\ut,k}^{m-1}
\end{align}
for all $t\geq a$. With this preparation, we are ready to formulate our result on error estimation.
\begin{theorem}
Fix $m>1$. Let $s$ be an arbitrary integer with $0\leq s\leq m-1$, and $M_{\ut,m,k}(t)$ be the quasi-interpolation spline basis functions defined in (\ref{nonuniform:Q_m:definition}) in terms of the $m$-th order B-splines on a knot sequence $\ut$ that satisfies (\ref{blending:proof:condition}). Then for a chosen $k\geq0$, there exists a constant $C_{m,s,k}>0$ such that for every $g\in C^{s+1}[a,\infty)$, the error of spline interpolation at the knots $t_j$ by the blending operator is given by
\[
\|g-\mathsf{P}_mg\|_{L^\infty[t_k,t_{k+1}]}\leq C_{m,s,k}\|g^{(s+1)}\|_{L^\infty[t_k,t_{k+1}]}\delta_{\ut,k}^{s+1},
\]
where 
\begin{align}\label{blending:proof:definition:Cms}
C_{m,s,k}:=\frac{1}{s!}\left(2+\frac{m(m-2)}{(m-1)!}C_{\ut,k}^{m-1}\right)
\end{align} 
and
\[
\delta_{\ut,k}:=\max_{k+1\leq i\leq k+m}(t_{i+1}-t_i).
\]
\end{theorem}
\begin{proof}
For any fixed $k\geq 0$, consider the Taylor polynomial expansion of $g$ at $t_k$
\[
g(t)=\sum_{j=0}^s\frac{g^{(j)}(t_k)}{j!}(t-t_k)^j+\int_{t_k}^{t_{k+1}}\frac{g^{(s+1)}(\tau)}{s!}h(t,\tau)\ud \tau,
\]
where $t\in [t_k,t_{k+1}]$ and $h(t,\tau):=(t-\tau)_+^s$. Since the blending operator $\mathsf{P}_m$ preserves all polynomials of degree $\leq m-1$ and $s\leq m$ is arbitrary, we have
\[
g(t)-\mathsf{P}_mg(t)=\int_{t_k}^{t_{k+1}}\frac{g^{(s+1)}(\tau)}{s!}\big[h(t,\tau)-(\mathsf{P}_mh(\cdot,\tau))(t)\big]\ud \tau.
\]
In the following, we suppress the variable of integration $\tau$ by writing $h(t)=h(t,\tau)$, so that
\begin{align*}
&h(t)-(\mathsf{P}_mh)(t)\\
&\quad=h(t)-\sum_{j=k-m+1}^{k+m-2}h(t_j)M_{\ut,m,j}(t)-\sum_{j=k-1}^{k+1}\left[h(t_j)-\sum_{l=j-m+1}^{j+m-2}h(t_l)M_{\ut,m,l}(t_j)\right]L_{\us,m,j}(t).
\end{align*}
Now, observe that since $h(t_j)=(t_j-\tau)_+^s$ and $\tau\in[t_k,t_{k+1}]$, we have $h(t_j)=0$ for all $j\leq k$. Hence,
\begin{align*}
&h(t)-(\mathsf{P}_mh)(t)\\
&\quad=h(t)-\sum_{j=k+1}^{k+m-2}h(t_j)M_{\ut,m,j}(t)-h(t_{k+1})L_{\us,m,k+1}(t)+\sum_{l=k+1}^{k+m-1}h(t_l)M_{\ut,m,l}(t_{k+1})L_{\us,m,k+1}(t).
\end{align*}
From (\ref{blending:proof:MtmkBound}) and the fact that $|L_{\us,m,j}(t)|\leq 1$ for all $t\in[t_k,t_{k+1}]$, we have
\begin{align*}
|h(t)-(\mathsf{P}_mh)(t)|\leq &\,2\delta_{\ut,k}^s+\frac{1}{(m-1)!}C_{\ut,k}^{m-1}(\delta_{\ut,k}^s+2\delta_{\ut,k}^s+\ldots+(m-2)\delta_{\ut,k}^s)\\
&\quad+\frac{1}{(m-1)!}C_{\ut,k}^{m-1}(\delta_{\ut,k}^s+2\delta_{\ut,k}^s+\ldots+(m-1)\delta_{\ut,k}^s)\\
=\,&\left(2+\frac{m(m-2)}{(m-1)!}C_{\ut,k}^{m-1}\right)\delta_{\ut,k}^s.
\end{align*}
Therefore, it follows that
\begin{align*}
\|g-\mathsf{P}_mg\|_{L^\infty[t_k,t_{k+1}]}&\leq \int_{t_k}^{t_{k+1}}\frac{g^{(s+1)}(\tau)}{s!}\left(2+\frac{m(m-2)}{(m-1)!}C_{\ut,k}^{m-1}\right)\delta_{\ut,k}^s\ud \tau\\
&\leq C_{m,s,k}\|g^{(s)}\|_{L^\infty[t_k,t_{k+1}]}\delta_{\ut,k}^{s+1}.
\end{align*}
\end{proof}
In conclusion, the blending operator, as a local spline interpolation operator, achieves the optimal interpolation error rate compared with the traditional spline interpolation operator. In addition, the error depends only on the local data profile, which allows the real-time implementation with the optimal error rate.

\subsection{Real-time implementation}
We conclude our discussion by giving a brief description of a real-time computational scheme to compute $f(t)$ for the in-coming data samples $g(t_0),g(t_1),\ldots$. For convenience, we only consider even order $m$. The formulation for the odd order is similar but slightly more complicated. 

\begin{algorithm}[h!]
  \begin{algorithmic}
\STATE First, pre-compute the B-spline values $N_{\us,m,l}(t_j)=:n_{l,j}$. Then for each $k$, since $M_{\ut,m,k}\in S_{\us,m}$, there exists a finite sequence $\{b_{k,l}\}$ in the formulation of  
\[
M_{\ut,m,k}(t)=\sum_{l}b_{k,l}N_{\us,m,l}(t).
\]
Also pre-compute $d_{k,j}=M_{\ut,m,k}(t_j)$.

Now, while the data sequence $\{g(t_k)\}$ is acquired, compute
\[
\tilde{g}_l=\sum_k b_{k,l}g(t_k)
\]
and simultaneously compute
\[
g^*_l=\frac{g(t_l)-\sum_{j} d_{l,j}g(t_j)}{n_{m(l-1)/2},l},
\]
and then up-sample $\{g^*_l\}$ by $m(l-1)/2$; that is, set
\[
g^{\#}_{m(k-1)/2}=g^*_k,\quad\mbox{and }g^{\#}_l=0\mbox{ otherwise}.
\]
Then, we have an on-line computational scheme for the quasi-interpolation spline interpolation:
\begin{align*}
f(t)=\sum_{l=-m+1}^n(\tilde{g}_l-g^{\#}_l)N_{\us,m,l}(t)
\end{align*}
for increasing number of samples from $g(t_n)$ to $g(t_{n+1})$, $\ldots$; and this can be implemented for real-time D/A conversion. 
\end{algorithmic}
\caption{Real-time implementation of the blending operator}
\label{alg:Blending}
\end{algorithm}
For uniform sampling $\{g(k\tau)\}$, $k=0,1,2,\ldots$, where $\tau>0$, a very simple algorithm is derived in \cite[p. 114-117]{Chui:1992}.

\begin{remark}
For larger values of the spline order $m$, observe that the support of the completely local spline basis function $L_{\us,m,j}(x)$ could be too small, perhaps resulting in some undesirable ``bumpy'' shape of the interplant $\mathsf{P}_mg$ of the data function $g$. Therefore, at the expense of higher computational cost and slower on-line performance, local interpolants with larger support, such as approximate modification of those constructed in \cite{Chui_DeVilliers:1996} could be constructed to replace the simple interpolants $L_{\us,m,j}(x)$.
\end{remark}

\section{Vanishing-moment (VM) wavelets with minimum support}\label{section:VMwavelet}

In this section, we introduce the notion of vanishing-moment (VM) wavelets on bounded intervals. In Subsection \ref{subsection:VM}, we first disclose the method of construction of those with minimum support (and maximum order of vanishing moments) in terms  of the B-splines $N_{\us,m,k}(t)$ of arbitrary order $m\geq 2$ on the knot sequence $\us$ in (3), with the following modification:
\begin{equation}\label{equation:underlines}
\us:s_{-m+1}=\ldots=s_0=0<s_1<\ldots<N=s_N=\ldots=s_{N+m-1}
\end{equation}
where $s_j=j$, for $j=0,\ldots,N$. More precisely, we will first introduce the wavelets $\psi_{\us,m,j}(x)$, $j=-m+1,\ldots,2N-m-1$, that, up to any constant multiple, are uniquely determined by the following governing conditions:
\begin{equation}\label{condition:wavelet:definition}
\psi_{\tus,m;n,j}(x):=\sum_{k=j}^{n+j}q_{n;j,k}N_{\frac{1}{2}\us,m,k}(x) 
\end{equation}
and
\begin{equation}\label{condition:VM}
\int_0^Nx^l \psi_{\tus,m;n,j}(x)\ud x=0,\mbox{ for }l=0,1,\ldots,n-1, 
\end{equation}
where the knot sequence $\frac{1}{2}\us$ is defined by 
\begin{equation}\label{equation:tilde_s}
\tus:=\frac{1}{2}\us :\,\tilde{s}_{-m+1}=\ldots=\tilde{s}_0=0<\tilde{s}_1<\ldots<\tilde{s}_{2N-1}<\tilde{s}_{2N}=2N=\ldots=\tilde{s}_{2N+m-1},
\end{equation}
with $\tilde{s}_j=\frac{1}{2}j$, for $j=0,\ldots,2N$, by ``halving'' the knot sequence $\us$, but with extension from $j=N$ to $j=2N$ to fill the given bounded interval $[0,N]$.

Observe that in view of the $n$-th order moment condition in (\ref{condition:VM}), the $n+1$ coefficients $q_{n;j,k}$, $k=j,\ldots,j+n$, in (\ref{condition:wavelet:definition}), we are assured that the non-trivial VM wavelets $\psi_{\tus,m;n,j}(x)$ are unique, up to a non-zero constant multiple. Indeed, since $N_{\tilde{\boldsymbol{s}},m,k}(x)$ reproduces polynomials \cite{Chui:1992} and there are $n$ equations,  when $n\leq m$, the solution to the $n+1$ coefficients are free up to $\RR$. In addition, the VM wavelets $\psi_{\tus,m;n,j}(x)$ have minimum support:
\begin{equation*}
\text{supp}\,\psi_{\tus,m;n,j}=\big[\,\tilde{s}_j,\tilde{s}_{j+m+n}\,\big],\quad j=-m+1,\ldots,2N-m-1, 
\end{equation*}
so that the support of the ``interior wavelets'', with $j=0,\ldots,2N-2m$, are given by
\begin{equation*}
\text{supp}\,\psi_{\tus,m,j}=\big[\,\frac{j}{2},\frac{j+m+n}{2} \,\big].
\end{equation*}

Since (almost) all wavelets in the literature embrace the multi-resolution analysis (MRA) structure and are generated by the operations of translation and dilation, we will first construct our VM (spline) wavelets on uniform knots in Subsection \ref{subsection:VM}. In addition, since (almost) all spline wavelets introduced in the wavelet literature restrict the order (i.e. number) $n$ of vanishing moments of the wavelets not to exceed the order $m$ of the spline function (i.e. $n\leq m$), we only consider $n=m$ in Subsection \ref{subsection:VM}. After all, VM wavelets for the MRA setting could have their place in further theoretical development and applications in the future. However, for our application of wavelets to the synchrosqueezed continuous wavelet transform in this paper, the MRA structure is not required. In Subsection \ref{subsection:VMextension}, we will extend the VM spline wavelets to spline functions on arbitrary knot sequences and with arbitrary order $n$ of vanishing moments, and allow $n>m$. For example, for $m=1$, it follows from our general result in Theorem \ref{uniform:VMwaveletProperty} that a compactly supported piecewise constant wavelet $\psi_{\us,1,n}(x)$, with finitely many jump discontinuities on an arbitrarily chosen knot sequence $\us$, can be constructed to have any desirable number $n$ of vanishing moments.

\subsection{VM wavelets with the same number of vanishing moments as the spline order}\label{subsection:VM}

This subsection is focused on MRA vanishing moment wavelets on a bounded interval. The general VM wavelets will be discussed in Subsection \ref{subsection:VMextension}.
Let us first consider the special case where the order $n$ of vanishing moments agrees with that of the VM spline wavelets. This is the initial investigation that leads to the general theory to be discussed in Subsection \ref{subsection:VMextension}. In addition, for our application to the SST, to be studied in Section \ref{section:tvPS}, these are the preferred wavelets on a bounded interval. For $n=m$, we simplify the notation by setting
\[
\psi_{\tus,m,j}(x):=\psi_{\tus,m;m,j}(x);\quad q_{j,k}:=q_{n;j,k}.
\]
Observe that the length of the support of the interior wavelets $\psi_{\tus,m,j}$ (i.e. for $j=0,\ldots,2N-2m$) is the same as the length of the support of the interior B-splines $N_{\us,m,k}(x)=N_{\us,m,0}(x-j)$, $j=0,\ldots,2N-m-1$, that generate the interior VM wavelets. There are $2(m-1)$ boundary 
VM wavelets, with $\psi_{\tus,m,j}$, $j=-m+1,\ldots,-1$ for the end-point $s_0=0$, and $\psi_{\tus,m,j}$, $j=2N-2m+1,\ldots,2N-m-1$ for the other end-point $s_N=\tilde{s}_{2N}=0$. It follows from the formulation (\ref{condition:wavelet:definition}) for these boundary wavelets that their supports are
\begin{equation*}
\text{supp}\,\psi_{\tus,m,j}=\big[\,0,\frac{2m+j}{2}\,\big],\quad j=-m+1,\ldots,-1,
\end{equation*}
and
\begin{equation*}
\text{supp}\,\psi_{\tus,m,j}=\big[\,\frac{j}{2},N\,\big],\quad j=2N-2m+1\ldots,2N-m-1.
\end{equation*}

For each $j=-m+1,\ldots,2N-m-1$, to compute the sequence 
\begin{equation*}
\uq_j:=\{q_{j,j},\ldots,q_{j,m+j}\},
\end{equation*}
we may simply apply (\ref{condition:VM}) to obtain the homogeneous system of $m$ linear equations with $m+1$ unknowns in $\uq_j$:
\[
\sum_{k=j}^{m+j}\left(\int_0^N x^lN_{\tilde{\boldsymbol{s}},m,k}(x)\ud x\right)q_{j,k}=0;
\]
or equivalently,
\begin{equation}\label{equation:evaluate_qij}
\sum_{k=j}^{m+j}\left(\int_0^N x^lN_{\us,m,k}(x)\ud x\right)q_{j,k}=0;
\end{equation}
which clearly has the same solution (or null) space. We next introduce the family $\{\underline{\boldsymbol{c}}^l\}$, $l=0,\ldots,m-1$, of sequences 
\[
\underline{\boldsymbol{c}}^l:=\{c^l_{-m+1},\ldots,c^l_{N-1}\},
\] 
defined by
\begin{equation*}
c^l_j:=(-1)^l\frac{l!}{(m-1)!}\left[\frac{\ud^{m-l-1}}{\ud t^{m-l-1}}\prod^{m-1}_{i=1}(t-s_{i+j})\right]\Big|_{t=0}.
\end{equation*}
Then it follows from Marsden's identity \cite{Marsden:1970} that 
\begin{equation*}
x^l=\sum_{u=-m+1}^{N-1}c^l_uN_{\us,m,u}(x),\quad x\in[0,N],
\end{equation*}
for $l=0,\ldots,m-1$ (see \cite[p.6 Eq(1.12)]{Chui:1988}). Therefore, the linear system (\ref{equation:evaluate_qij}) can be reformulated as
\begin{equation}\label{equation:dq:product}
\sum_{k=j}^{m+j}d^l_kq_{j,k}=0,\quad l=0,\ldots,m-1,
\end{equation}
where 
\begin{equation}\label{definition:dlk}
d^l_k:=\sum_{u=-m+1}^{N-1}c^l_u\langle N_{\us,m,k},N_{\us,m,u}\rangle.
\end{equation}
Since it can be shown that each of the $(m+1)\times m$ matrices
\begin{equation*}
A_j=\big[\,d^l_k\,\big]_{j\leq k\leq m+j,\, 0\leq l\leq m-1},
\end{equation*}
where $j=-m+1,\ldots,2N-m-1$, has full rank $m$, the null space of (\ref{equation:dq:product}) has dimension $1$; that is, the VM wavelets $\psi_{\us,m,j}(x)$, $j=-m+1,\ldots,2N-m-1$, in (\ref{condition:wavelet:definition}) are unique, up to a (non-zero) constant multiple.

The remaining computational scheme to be discussed is the computation of the inner products in (\ref{definition:dlk}). For convenience, we drop the subscript $\underline{s}$ in the B-spline; that is, $N_{m,k}=N_{\us,m,k}$. 
For interior B-spline, $N_{m,k}$ and $N_{m,u}$, that is, $0\leq k,u\leq N-m$, these B-splines are simply integer translations of the B-spline $N_{m,0}(x)$, so that
\begin{align*}
\langle N_{m,k},N_{m,u}\rangle&=\int_0^N N_{m,0}(x-k)N_{m,0}(x-u)\ud x\\
&=\int_0^N N_{m,0}(x)N_{m,0}(x-(u-k))\ud x\\
&=\int_0^N N_{m,0}(m-x)N_{m,0}(x-(u-k))\ud x
\end{align*}
since $N_{m,0}(m-x)=N_{m,0}(x)$ by the symmetric property of $N_{m,0}(x)$. Hence, by the convolution property of cardinal B-spline with integer knots, we have
\begin{align}\label{NmkNmu_innerproduct}
\langle N_{m,k},N_{m,u}\rangle&=\int_{-\infty}^{\infty} N_{m,0}((m-u+k)-x)N_{m,0}(x)\ud x=N_{2m,0}(m-u+k).
\end{align}
There are various methods for evaluating the cardinal B-spline $N_{2m,0}(x)$ at integers $x=v:=m-u+k$. The most popular one is ``weighted differencing'' of two lower order B-spline, namely,
\begin{align*}
N_{2m,0}(v)&=\frac{1}{2m-1}\left[\,vN_{2m-1,0}(v)+(2m-v)N_{2m-1,0}(v-1)\,\right]\\
N_{2m-1,0}(w)&=\frac{1}{2m-2}\left[\,wN_{2m-2,0}(w)+(2m-1-w)N_{2m-2,0}(w-1)\,\right]
\end{align*}
for $v,w\in\ZZ$, and so forth, till we arrive at 
\[
N_2(v)=\delta(v),\quad v\in\ZZ;
\]
(see, for example, \cite[p.86, Eq(4.2.4)]{Chui:1992}). Another, perhaps more efficient, algorithm is the algorithm by applying the generalized Pascal triangle as described in \cite[p.190-p.191]{Chui:1992}. For example, for $m=4$ and $m=5$, the non-zero values of $N_{2m,0}(v)$, $v\in\ZZ$, are listed in row \#7 and row \#9, respectively, in \cite[p.191, Table 6.3.1]{Chui:1992}, namely
\vspace{.2cm}
\begin{enumerate}
\item $7!N_8(v)=1,120,1191,2416,1191,120,1$, for $v=1,\ldots,7$ respectively.
\item $9!N_{10}(v)=1,502,14608,88234,156190,88234,14608,502,1$, for $v=1,\ldots,9$ respectively.
\end{enumerate}
\vspace{.2cm}
On the other hand, if one of the two B-splines $N_{m,k}$ and $N_{m,u}$ in the inner product $\langle N_{m,k},N_{m,u}\rangle$ is not an interior B-spline, then the convolution property is not valid. In this case, the algorithm developed in \cite{Chui_Lai:1987} (see also \cite[p.12]{Chui:1988}) for computing each polynomial piece of the B-spline $N_{m,k}(u)$ on the interval $[j,j+1]$, $j=k,\ldots,k+m-1$, in terms of Bernstein polynomials can be applied. We remark that an efficient scheme for computing the inner product on $[0,1]$ of two Bernstein polynomials of the same degree is available in the literature.

\begin{example}{[VM wavelets $\psi_{\us,4,j}(x)$ in (\ref{condition:wavelet:definition}) with $m=4$]}
\begin{enumerate}
\item For the interior VM wavelets; that is, for $j=0,\ldots,2N-8$, the coefficient sequence $\uq_j=\{q_{j,j},\ldots,q_{j,j+4}\}$ is given by any non-zero constant multiple of $\{1,-4,6,-4,1\}$.
\item For the $m-1=4-1=3$ boundary VM wavelets corresponding to the end-point $x=0$, we only list $\psi_{\tus,4,-2}(x)$ and $\psi_{\tus,4,-1}(x)$ since we do not need $\psi_{\tus,4,-3}(x)$ which does not vanish at $x=0$.
\begin{enumerate}
\item $\uq_{-1}=\{q_{-1,-1},q_{-1,0},q_{-1,1},q_{-1,2},q_{-1,3}\}$ is any non-zero constant multiple of 
\[
\Big\{\,\frac{7}{3},\frac{-319}{60},\frac{101}{15},\frac{-25}{6},1\,\Big\}.
\]
\item $\uq_{-2}=\{q_{-2,-2},q_{-2,-1},q_{-2,0},q_{-2,1},q_{-2,2}\}$ is any non-zero constant multiple of 
\[
\Big\{\,1,\frac{-116}{25},\frac{919}{100},\frac{-57}{5},6\,\Big\}.
\]
\end{enumerate}
\item For the boundary VM wavelet corresponding to the end-point $x=N$, we apply the symmetry property to write
\begin{align*}
\psi_{\tus,4,2N-9}(x)&=\psi_{\tus,4,-1}(N-x),\\
\psi_{\tus,4,2N-8}(x)&=\psi_{\tus,4,-2}(N-x),\\
\psi_{\tus,4,2N-7}(x)&=\psi_{\tus,4,-3}(N-x);
\end{align*}
but, again, since $\psi_{\tus,4,2N-7}(x)(N)=\psi_{\tus,4,-3}(0)$ does not vanish, we only need 
$\psi_{\tus,4,2N-9}(x)$ and $\psi_{\tus,4,2N-8}(x)$ in our application.
\end{enumerate}
\end{example}

Motivated by the above example of interior VM spline-wavelets $\psi_{\tus,4,j}(x)$ with $m=4$, let us consider the general setting
\begin{align}\label{definition:psim}
\psi_m(x):=\sum_{k=0}^m(-1)^k\left(\hspace{-4pt}
\begin{array}{c}
m\\ k
\end{array}\hspace{-4pt}\right)N_m(2x-k),
\end{align}
for an arbitrary integer $m\geq 1$, where
\begin{align}\label{definition:Nm}
N_m(x):=N_{m,0}(x)=\int_0^1N_{m-1}(x-t)\ud t,
\end{align}
defined recursively, with $N_1(x)=\chi_{[0,1)}(x)$. Observe that $\psi_1(x)=N_1(2x)-N_1(2x-1)$ is the Haar wavelet. 
\begin{theorem}\label{uniform:VMwaveletProperty}
For any integer $m\geq 2$ and $j=0,\ldots,2N-2m$, the VM wavelets defined in (\ref{condition:wavelet:definition}) and (\ref{condition:VM}) are (any non-zero constant multiple of)
\[
\psi_{m,j}(x)=\psi_m(x-j/2),\quad 0\leq j\leq 2(N-m),
\]
(called the interior VM wavelets), where $\psi_m(x)$ is defined in (\ref{definition:psim}). Furthermore, the Fourier transform of $\psi_m(x)$ is given by
\begin{align}\label{hatpsiomega}
\hat{\psi}_m(\omega)=\frac{1}{2}(i\omega/2)^m\hat{N}_{2m}(\omega/2);
\end{align}
or equivalently,
\begin{align}\label{psimx}
\psi_m(x)=N^{(m)}_{2m}(2x).
\end{align}
\end{theorem}

\begin{proof}
The formula (\ref{definition:Nm}) for the so-called {\it cardinal B-spline} $N_{m,0}$ is well-known (and in fact is commonly used to define $N_m(x)$). Now, since the family $\{q_{j,k}\}$ of finite sequences in the definition of the VM wavelets in (\ref{condition:wavelet:definition}) and (\ref{condition:VM}) is unique (up to a non-zero constant multiples), and since
\[
\psi_{\tus,m,j}(x)=\psi_{\tus,m,0}(x-j/2)
\]
for the interior VM wavelets (that is, for $j=0,\ldots,2(N-m)$), it is sufficient to prove that $\psi_m(x)$, as defined in (\ref{definition:psim}), meets the vanishing moment requirement (\ref{condition:VM}); that is,
\begin{align}\label{condition:VM2}
\int_{-\infty}^\infty x^l\psi_m(x)\ud x=0,\quad l=0,\ldots,m-1.
\end{align}
Observe that from the definition (\ref{definition:psim}), it is clear that $\psi_m(x)$ has compact support with $\text{supp}\,\psi_m=[0,m]$, so that its Fourier transform $\hat{\psi}_m(\omega)$, along with the derivatives $D^{l}\hat{\psi}_m(\omega)$, $l=1,2,\ldots$, are well-defined. Hence, since (\ref{condition:VM2}) is equivalent to 
\begin{align}\label{condition:Dl0}
(D^l\hat{\psi}_m)(0)=0,\quad l=0,\ldots,m-1,
\end{align}
it is sufficient to verify (\ref{condition:Dl0}). To compute the Fourier transform of $\psi_m(x)$, in (\ref{definition:psim}), we observe that $\hat{N}_m(\omega)=(\hat{N}_1(\omega))^m$, so that
\begin{align*}
\hat{\psi}_m(\omega)&=\sum_{k=0}^m(-1)^k\left(\begin{array}{c}m\\k\end{array}\right)\frac{1}{2}e^{-ik\omega/2}\hat{N}_m(\omega/2)\\
&=\frac{1}{2}\left(1-e^{-i\omega/2}\right)^m(\hat{N}_1(\omega/2))^m\\
&=\frac{1}{2}\left(1-e^{-i\omega/2}\right)^m\left(\frac{1-e^{-i\omega/2}}{i\omega/2}\right)^m\\
&=\frac{1}{2}(i\omega/2)^m\left(\frac{1-e^{-i\omega/2}}{i\omega/2}\right)^{2m}\\
&=\frac{1}{2}(i\omega/2)^m\hat{N}_{2m}(\omega/2).
\end{align*}
Therefore, if $\omega$ is extended to a complex variable, then $\hat{\psi}_m(\omega)$ is an entire function with $m$-fold zero at $\omega=0$. This completes the proof of (\ref{condition:VM2}), and hence of (\ref{condition:Dl0}). The above computation also yields (\ref{hatpsiomega}). The proof of (\ref{psimx}) is simply integration of the Fourier transform of $N_{2m}^{(m)}(2x)$ by parts $m$ times, followed by application of (\ref{hatpsiomega}).
\end{proof}

\subsection{Extension to arbitrary knot sequences and order of vanishing moments}\label{subsection:VMextension}
Observe that (\ref{psimx}) can also be derived by applying the formula 
\[
N_\ell'(x)=N_{\ell-1}(x)-N_{\ell-1}(x-1)
\]
to $N_{2m}(x)$ $m$ times, starting with $\ell=2m$ (see, for example, \cite[p.86, Theorem 4.3 (vii)]{Chui:1992}). If the $n$-th order vanishing moment in (\ref{condition:VM}) is desired, for any integer $n\geq 1$, then the wavelet $\psi_{m;n}(x)$ of the corresponding interior wavelets $\psi_{\us,m;n,j}(x)$ in (\ref{condition:wavelet:definition}) can be formulated as
\[
\psi_{\tus,m;n,j}(x)=\psi_{m;n}(x-j/2),
\]
where 
\begin{align}\label{definition:psimnx}
\psi_{m;n}(x):=\sum_{k=0}^n(-1)^k\left(\hspace{-4pt}\begin{array}{c}n\\k\end{array}\hspace{-4pt}\right)N_m(2x-k).
\end{align}
The same proof of Theorem \ref{uniform:VMwaveletProperty} can be used to yield
\begin{align}\label{psimnx}
\psi_{m;n}(x)=N^{(n)}_{m+n}(2x).
\end{align}

The importance of the time-domain representation (\ref{psimx}), and more generally (\ref{psimnx}), is that the VM wavelets can be extended to splines on an arbitrary knot sequence
\begin{equation}\label{definition:ux}
\ux:\ldots\leq x_j\leq x_{j+1}\leq\ldots,
\end{equation}
with $x_{j+m}>x_j$ for all $j\in\ZZ$, which includes the knot sequence $\us$ in (\ref{nonuniform:definition:s_knot}), and allows for both bounded intervals and stacked knots, even among the interior knots of the interval. Furthermore, the order $n$, of vanishing moments can be arbitrarily chosen, independent of the order $m$. In particular, in some applications, it would be preferable to choose $n>m$. In addition, the MRA structure considered in the previous subsection is no longer required in our deviation. 

\begin{theorem}[VM spline-wavelets on arbitrary knot sequences]\label{Theorm:VMwavelet_arbitraryoKnots}
Let $m,n\geq 1$ be arbitrary integers, and $\ux$ an arbitrary knot sequence as defined in (\ref{definition:ux}). Then the spline basis functions
\begin{align}\label{definition:psi_uxmnx}
\psi_{\ux,m;n,k}(x):=N^{(n)}_{\ux,m+n,k}(x),\quad k\in\ZZ,
\end{align}
to be called {\it VM wavelets on $\ux$}, where $N_{\ux,m+n,k}(x)$ is defined in (\ref{definition:Nmktuniform}), which satisfy the moment conditions
\begin{equation}
\left\{\begin{array}{l}
\displaystyle\int_{-\infty}^\infty x^l\psi_{\ux,m;n,k}(x)\ud x=0,\quad l=0,\ldots,n-1\\ \\
\displaystyle\int_{-\infty}^\infty x^n\psi_{\ux,m;n,k}(x)\ud x\neq 0.
\end{array}
\right.
\end{equation}
\end{theorem}
\begin{remark}
Since the knot sequence $\tus$ in (\ref{equation:tilde_s}) on the bounded interval $[0,2N]$ is a special case of $\ux$ in (\ref{definition:ux}), the formula (\ref{definition:psi_uxmnx}) can be applied to compute the boundary VM wavelets, such as $\psi_{\us,4,-2}(x)$ and $\psi_{\us,4,-1}(x)$, in the example in Subsection \ref{subsection:VM}.
\end{remark}
\begin{proof}
The proof of Theorem \ref{Theorm:VMwavelet_arbitraryoKnots} is a straight-forward application of integration by parts, by using the compact-support property of $N_{\ux,m+n,k}^{(l)}(x)$, $l=0,\ldots,n-1$, namely,
\begin{align*}
&\int_{-\infty}^\infty x^lN_{\ux,m+n,k}^{(n)}\ud x\\
=\,&(-1)^ll!\int_{-\infty}^\infty N_{\ux,m+n,k}^{(n-l)}\ud x\\
=\,&\,(-1)^ll!\big[\,N^{(n-l-1)}_{\ux,m+n,k}(x_{m+n+k})-N^{(n-l-1)}_{\ux,m+n,k}(x_{k})\,\big]=0
\end{align*}
since $N^{(n-l-1)}_{\ux,m+n,k}(x_{m+n+k})=N^{(n-l-1)}_{\ux,m+n,k}(x_{k})=0$ for all $l=0,\ldots,n-1$. On the other hand,
\[
\int_{-\infty}^\infty x^nN_{\ux,m+n,k}^{(n)}\ud x=(-1)^nn!\int_{x_k}^{x_{m+n+k}} N_{\ux,m+n,k}\ud x\neq 0,
\]
since $N_{\ux,m+n,k}>0$ for $x_k<x<x_{m+n+k}$ and the assumption that $x_{m+n+k}>x_k$.
\end{proof}

As an immediate consequence of Theorem \ref{Theorm:VMwavelet_arbitraryoKnots}, we have the following result. 

\begin{corollary}\label{Corollary:derivativeofpsimnx}
Let $m,n\geq 1$ be arbitrary integers. Then the derivative of the VM wavelet $\psi_{\ux,m;n,k}$ on an arbitrary knot sequence $\ux$, is given by
\begin{align}\label{formula:psi_uxmnxDerivative}
\psi'_{\ux,m;n,k}(x)=\psi_{\ux,m-1;n+1,k}(x).
\end{align}
In particular, for the interior VM wavelets on half-integer knots $\tus$,
\[
\psi'_{\tus,m;n,k}(x)=2\psi_{\tus,m-1;n+1,k}(x).
\]
\end{corollary}
\begin{proof}
Indeed, it follows from (\ref{definition:psi_uxmnx}) in Theorem \ref{Theorm:VMwavelet_arbitraryoKnots} that 
\[
\psi'_{\ux,m;n,k}(x)=N^{(n+1)}_{\ux,m+n,k}(x)=N_{\ux,(m-1)+(n+1),k}'(x)=\psi_{\ux,m-1;n+1,k}(x).
\]
\end{proof}
The importance of the derivative formula (\ref{formula:psi_uxmnxDerivative}) is its application to the reassignment technique in the time-frequency analysis which we will discuss in Section \ref{section:tvPS}. In view of the need of taking the $n^{th}$ order derivatives of the B-splines $N_{\ux,m+n,k}$ to compute the VM wavelet $\psi_{\ux,m;n,k}(x)$, we end this subsection by recalling the following derivative formula (see, for example, \cite{deBoor:1978})
\[
N'_{\ux,l,k}(x)=\frac{l-1}{x_{k+l-1}-x_k}N_{\ux,l-1,k}(x)-\frac{l-1}{x_{k+l}-x_{k+1}}N_{\ux,l-1,k+1}(x),
\]
for $l=m+n,\ldots,m+1$.

\subsection{Computation of analytic VM wavelets}\label{section:analyticVMwavelet}
In this sub-section, we will focus on the interior wavelets
\[
\psi_{m,n}(t):=\psi_{\ZZ,m;n,0}(t)
\]
in (\ref{definition:psi_uxmnx}) of Theorem \ref{Theorm:VMwavelet_arbitraryoKnots}, with integer knot sequence $\ux=\ZZ$ and $k=0$. In signal processing, if $\psi_{m,n}$ is treated as a signal, then the extension of $\psi_{m,n}$ to
\begin{align}\label{AVMW}
\psi^*_{m,n}(t):=\psi_{m,n}(t)+i(\mathcal{H}\psi_{m,n})(t),
\end{align}
called the analytic signal representation allows us to write $\psi^*_{m,n}(t)$ as an AM-FM signal:
\[
\psi^*_{m,n}(t)=A(t)e^{i\phi(t)},
\]
with $A(t)=\sqrt{\psi_{m,n}^2(t)+(\mathcal{H}\psi_{m,n})^2(t)}$ and $\phi(t)=\tan^{-1}(\mathcal{H}\psi_{m,n}(t)/\psi_{m,n}(t))$. Here and throughout this paper, given a function $f$ in the suitable space, the operator $\mathcal{H}$ denotes the Hilbert transform, defined by
\[
(\mathcal{H}f)(t)=p.v.\left\{\frac{1}{\pi}\int^{\infty}_{-\infty}\frac{f(x)}{t-x}\ud x\right\}.
\] 
However, for our application of the VM wavelet as mother wavelets of the CWT, it is necessary to extend $\psi_{m,n}$ to $\psi_{m,n}^*$ in order to analyze only the (positive) frequency contents of a given signal $g(t)$. Indeed, since
\begin{align*}
\widehat{\psi^*_{m,n}}(\omega)&=\widehat{\psi_{m,n}}(\omega)+i(\widehat{\mathcal{H}\psi_{m,n}})(\omega)\\
&=\widehat{\psi_{m,n}}(\omega)+i(-i\text{sgn}(\omega)\widehat{\psi_{m,n}}(\omega))\\
&=2\widehat{\psi_{m,n}}(\omega)
\end{align*}
for $\omega>0$ and $\widehat{\psi^*_{m,n}}(\omega)=0$ for $\omega<0$, we have, 
\begin{align*}
W_g(a,b)&=\langle g,{\psi_{m,n}^{*(a,b)}}\rangle=\frac{1}{2\pi}\langle \hat{g},\widehat{{\psi^*_{m,n}}^{(a,b)}}\rangle \\
&=\frac{1}{2\pi}\int^{\infty}_{-\infty}\hat{g}(\omega)\overline{\widehat{\psi^*_{m,n}}}(\omega)\ud\omega\\
&=\frac{1}{2\pi}\int^{\infty}_{0}\hat{g}(\omega)\overline{\widehat{\psi^*_{m,n}}}(\omega)\ud\omega \\
&=\frac{1}{\pi}\int^{\infty}_{0}\hat{g}(\omega)\overline{\widehat{\psi_{m,n}}}(\omega)\ud\omega\,,
\end{align*}
where 
\begin{align}\label{definition:CWTofg}
W_{g}(a,b)=\langle g,\psi^{(a,b)}\rangle,
\end{align}
is the CWT of the function $g$ in a suitable linear space of functions, $a>0$ is the scale, $b\in\RR$ denotes the time position, 
\begin{align}\label{definition:psiababbrev}
\psi^{(a,b)}(t):=\frac{1}{a}\psi\left(\frac{t-b}{a}\right),
\end{align}
and $\psi$ is the mother wavelet chosen by the user. To compute $\mathcal{H}\psi_{m,n}$ in (\ref{AVMW}), we recall from Theorem \ref{Theorm:VMwavelet_arbitraryoKnots} that 
\begin{align}\label{UVM}
\psi_{m,n}(t)=N_{m+n}^{(n)}(t)
\end{align}
(for $\ux=\ZZ$ and $k=0$), where $N_r(t)$ denotes the $r^{th}$ order cardinal B-spline, defined by the $r^{th}$-fold convolution of the first order cardinal B-spline (see (\ref{definition:Nm}))
\[
N_1(t)=\chi_{[0,1)}(t),
\]
so that
\[
(\mathcal{H}N_1)(t)=\frac{1}{\pi}\int^1_0\frac{\ud x}{t-x}=\ln\left|\frac{t}{t-1}\right|.
\]
Also recall the following recurrence formula for computing the $m^{th}$ order cardinal B-spline $N_m$ in terms of the characteristic function $N_1:=\chi_{[0,1)}$:
\begin{align}\label{RBC}
N_r(t)=\frac{t}{r-1}N_{r-1}(t)+\frac{r-t}{r-1}N_{r-1}(t-1),
\end{align}
for $r=2,\ldots,m$. One of the important properties of the Hilbert transform for our computation of $\psi_{m,n}^*$ is that it preserves the recurrence relation (\ref{RBC}):
\begin{align}\label{RHBC}
(\mathcal{H}N_r)(t)=\frac{t}{r-1}(\mathcal{H}N_{r-1})(t)+\frac{r-t}{r-1}(\mathcal{H}N_{r-1})(t-1),
\end{align}
as shown in \cite[Theorem 3.2, p.179]{Chen_Huang_Riemenschneider_Xu:2006}. For completeness, we include the derivation of (\ref{RHBC}) as follows. Let $I(t)=t$ denote the identity function. Then for $r\geq 1$,
\begin{align*}
(\mathcal{H}(IN_r))(t)&=\frac{1}{\pi}\int^\pi_{-\pi}\frac{xN_r(x)}{t-x}\ud x \\
&= \frac{t}{\pi}\int^\infty_{-\infty}\frac{N_r(x)}{t-x} \ud x - \frac{1}{\pi}\int^{\infty}_{-\infty} \frac{t-x}{t-x}N_r(x)\ud r\\
&=t(\mathcal{H}N_r)(t)-\frac{1}{\pi}\int^{\infty}_{-\infty}N_r(x)\ud x\\
&=t(\mathcal{H}N_r)(t)-\frac{1}{\pi}.
\end{align*}
Hence, by applying this identity and the translation-invariance property of the Hilbert transform to the recurrence formula (\ref{RBC}), we obtain
\begin{align*}
(\mathcal{H}N_r)(t)&=\frac{1}{r-1}\big[ t(\mathcal{H}N_{r-1})(t)-\frac{1}{\pi}\big]+\frac{r}{r-1}(\mathcal{H}N_{r-1})(t-1)-\frac{1}{r-1}\big[ t(\mathcal{H}N_{r-1})(t-1)-\frac{1}{\pi}\big]\\
&=\frac{t}{r-1}(\mathcal{H}N_{r-1})(t)+\frac{r}{r-1}(\mathcal{H}N_{r-1})(t-1)-\frac{t}{r-1}(\mathcal{H}N_{r-1})(t-1),
\end{align*}
which is (\ref{RHBC}).

We are now ready to derive the following efficient scheme for the computation of $\mathcal{H}\psi_{m,n}$, and hence $\psi_{m,n}^*$. First, from (\ref{UVM}), we observe that 
\[
\psi_{m,n}(t)=N_{m+n}^{(n)}(t)=\sum_{k=0}^n(-1)^k\left(\begin{array}{c}n\\k\end{array}\right)N_m(t-k).
\]
Therefore, by the translation-invariance property of
\begin{align}\label{HVM}
(\mathcal{H}\psi_{m,n})(t)=\sum_{k=0}^m(-1)^k\left(\begin{array}{c}n\\k\end{array}\right)(\mathcal{H}N_m)(t-k),
\end{align}
where $(\mathcal{H}N_m)(t-k)$ can be computed recursively by applying (\ref{RHBC}), namely,
\begin{align*}
(\mathcal{H}N_r)(t-k)=\frac{t}{r-1}(\mathcal{H}N_{r-1})(t-k)+\frac{r-t}{r-1}(\mathcal{H}N_{r-1})(t-k-1),
\end{align*}
for $r=2,\ldots,m$, with initial function
\begin{align*}
(\mathcal{H}N_2)(t-k)&=t(\mathcal{H}N_1)(t-k)+(2-t)(\mathcal{H}N_1)(t-k-1)\nonumber\\
&=t\ln\left|\frac{t-k}{t-k-1}\right|+(2-t)\ln\left|\frac{t-k-1}{t-k-2}\right|.
\end{align*}
We remark that although the VM wavelets are compactly supported, their Hilbert transforms $\mathcal{H}\psi_{m,n}$ do not have finite support. The reason is that
\[
(\mathcal{H}N_1)(t)=\ln\left|\frac{t}{t-1}\right|=\ln\left|1+\frac{1}{t-1}\right|\sim\left| \frac{1}{t-1}\right|
\]
as $|t|\to \infty$, so that
\[
\mathcal{H}N_m(t)=O(|t|^{-1})
\]
for $|t|\to\infty$, and the decay is not fast. However, since the $n^{th}$ difference operator acts like the $n^{th}$ differential operator, it follows from $\psi_{m,n}=N^{(n)}_{m+n}$ that
\[
(\mathcal{H}\psi_{m,n})(t)=O(|t|^{n+1}) 
\]
for $|t|\to \infty$.

We end this section by mentioning that if an VM wavelet $\psi_{m,n}$ is used as the mother wavelet to compute the CWT of an analytic signal
\[
g^*(t)=g(t)+i(\mathcal{H}g)(t),
\]
we can avoid the numerical computation of $\mathcal{H}g$ (which could be expensive and quite inaccurate) by using the analytic VM wavelet $\psi_{m,n}^*$ to compute the CWT of the given signal $g(t)$ instead. Indeed, by the anti-symmetry property of the Hilbert transform, we have
\begin{align*}
\langle g^*,\psi^{(a,b)}_{m,n}\rangle&=\langle g+i\mathcal{H}g,\psi^{(a,b)}_{m,n}\rangle\\
&=\langle g ,\psi^{(a,b)}_{m,n}\rangle+i\langle \mathcal{H}g ,\psi^{(a,b)}_{m,n}\rangle\\
&=\langle g ,\psi^{(a,b)}_{m,n}\rangle-i\langle g ,\mathcal{H}\psi^{(a,b)}_{m,n}\rangle\\
&=\langle g ,\psi^{(a,b)}_{m,n}+i\mathcal{H}\psi^{(a,b)}_{m,n}\rangle=\langle g ,{\psi^{*(a,b)}_{m,n}}\rangle
\end{align*}

\subsection{VM wavelets are asymptotically derivatives of the Gaussian function}
In this subsection, we first prove that the VM wavelets $\psi_{m;n}$, as introduced in (\ref{definition:psimnx}), are asymptotically the same as the $n^{th}$ derivative of the Gaussian, after being properly scaled and centered, as $m+n$ tends to infinity. However, since $\psi_{m,n}$ is compactly supported, the spectra $|\hat{\psi}_{m;n}(\omega)|^2$ of $\psi_{m;n}$ must have side-lobes in nature. The second objective of this subsection is to illustrate the rapid decay of the side-lobes relative to the main-lobes for increasing spline order.  We will plot the spectra of $\psi_m(x):=\psi_{m,m}(x)$ (with $n=m$) for $m=3,\ldots,6$ and demonstrate the rapid decrease of the energy of the side-lobe to main-lobe ratios, in terms of SNR (measured in dB).

Let $g(x)$ be the (normalized) Gaussian function defined by 
\[
g(x):=\frac{1}{\sqrt{2\pi}}e^{-x^2/2},
\]
with Fourier transform given by
\[
\hat{g}(\omega)=e^{-\omega^2/2}.
\] 
In the following, we set $M=m+n$ and assume that $m>1$. From the definition of $\psi_{m;n}(x)$ in (\ref{psimnx}) and following the proof of Theorem \ref{uniform:VMwaveletProperty}, we have 
\begin{align}\label{formula:psimnhatomega}
\hat{\psi}_{m;n}(\omega)=\frac{1}{2}\left(\frac{i\omega}{2}\right)^n\hat{N}_M\left(\frac{\omega}{2}\right).
\end{align}
We have the following result.
\begin{theorem}
For $M=m+n$,
\[
\lim_{M\to\infty}\left[\left(\frac{M}{12}\right)^{n+1}\psi_{m;n}\left(\sqrt{\frac{M}{48}}(x+12)\right)-g^{(n)}(x)\right]=0,
\]
where the convergence is both pointwise for all $x\in\RR$ and in $L^p=L^p(\RR)$ for $1<p<\infty$.
\end{theorem}
\begin{proof}
Taking the Fourier transform of the $n^{th}$ derivative of the Gaussian $g(x)$ and applying the formula (\ref{formula:psimnhatomega}), we obtain
\[
2\left(\sqrt{\frac{M}{12}}\right)^n\hat{\psi}_{m;n}\left(\frac{4\sqrt{3}\omega}{\sqrt{M}}\right)e^{i\sqrt{3M}\omega}-(i\omega)^n\hat{g}(\omega)=
(i\omega)^n\left[ \hat{N}_M\left(\frac{2\sqrt{3}\omega}{\sqrt{M}}\right)e^{i\sqrt{3M}\omega}-\hat{g}(\omega)\right],
\]
and by a direct calculation, for a fixed $\omega$, we have
\[
\hat{N}_M\left(\frac{2\sqrt{3}\omega}{\sqrt{M}}\right)=e^{-i\sqrt{3M}\omega}\left(\frac{\sin(\sqrt{3/M}\omega)}{\sqrt{3/M}\omega}\right)^M.
\]
For the Cardinal B-spline $N_M$, we have the asymptotic property
\begin{align}\label{hatNM_asymptotical}
(i\omega)^n \hat{N}_M\left(\frac{2\sqrt{3}\omega}{\sqrt{M}}\right)e^{i\sqrt{3M}\omega}- (i\omega)^ne^{-\frac{\omega^2}{2}}\to 0
\end{align}
as $m\to \infty$ (see, for example \cite[Corollary 1]{Chui_Wang:1997}). Also observe that $1\leq p<\infty$, for all $M\in\NN$,
\[
(i\omega)^n \hat{N}_M\left(\frac{2\sqrt{3}\omega}{\sqrt{M}}\right)e^{i\sqrt{3M}\omega}-(i\omega)^n\hat{g}(\omega)\in L^p(\RR),
\] 
since $m> 1$. In addition, by the same argument as that of \cite[Lemma 4.4]{Chui_Wang:1997}, for $1<p<\infty$, there exists an $L^1$ function $q(\omega)$ so that
\begin{align}\label{formula:rescaled_psimnhat}
\left|2\left(\sqrt{\frac{M}{12}}\right)^n\hat{\psi}_{m;n}\left(\frac{4\sqrt{3}\omega}{\sqrt{M}}\right)e^{i\sqrt{3M}\omega}-(i\omega)^n\hat{g}(\omega)\right|^p\leq q(\omega).
\end{align}
Indeed, for $|\omega|\leq \frac{\pi}{3}\sqrt{\frac{M}{3}}$, we have $\left(\frac{\sin(\sqrt{3/M}\omega)}{\sqrt{3/M}\omega}\right)^M\leq e^{-\omega^2/4}$; if $|\omega|>\frac{\pi}{3}\sqrt{\frac{M}{3}}$, then $\left(\frac{\sin(\sqrt{3/M}\omega)}{\sqrt{3/M}\omega}\right)^M\leq \frac{1}{1+|\omega|^M}$. Hence, the errors in (\ref{formula:rescaled_psimnhat}) are bounded above by the family of functions $q_M(\omega)$, defined by
\[
q_M(\omega):=\omega^n\left(e^{-\omega^2/4}+\frac{1}{1+|\omega|^M}\right),
\] 
which have uniformly bounded $L^1(\RR)$ norms, since $m>1$.
Thus, by the Lebegue Dominated Convergence theorem, we may conclude that 
\[
\lim_{m\to \infty}\left\|2\left(\sqrt{\frac{M}{12}}\right)^n\hat{\psi}_{m;n}\left(\frac{4\sqrt{3}\omega}{\sqrt{M}}\right)e^{i\sqrt{3M}\omega}-(i\omega)^n\hat{g}(\omega)\right\|_{L^p}=0
\]
for all $1<p<\infty$, and hence,
\[
\lim_{m\to \infty}\left\|\left(\frac{M}{12}\right)^{n+1}\psi_{m;n}\left(\sqrt{\frac{M}{48}}(x+12)\right)-g^{(n)}(x)\right\|_{L^2}=0,
\]
where we have applied the Plancheral identity to 
\begin{align*}
&\mathcal{F}^{-1}\left\{2\left(\sqrt{\frac{M}{12}}\right)^n\hat{\psi}_{m;n}\left(\frac{4\sqrt{3}\omega}{\sqrt{M}}\right)e^{i\sqrt{3M}\omega}-(i\omega)^n\hat{g}(\omega)\right\}(x)\\
&\quad=\left(\frac{M}{12}\right)^{n+1}\psi_{m;n}\left(\sqrt{\frac{M}{48}}(x+12)\right)-g^{(n)}(x),
\end{align*}
where $\mathcal{F}^{-1}$ is the inverse Fourier transform. On the other hand, by a similar argument using the Lebegue Dominated Convergence theorem, we may conclude that \\$\left(\frac{M}{12}\right)^{n+1}\psi_{m;n}\left(\sqrt{\frac{M}{48}}(x+12)\right)-g^{(n)}(x)$ converges to $0$ for all $x$, for $m\to\infty$.   
\end{proof}

\vspace{0.4cm}
Next, since the spectrum of $g^{(n)}(x)$ is given by
\[
|(i\omega)^n\hat{g}(\omega)|^2=\omega^{2n}e^{-\omega^2},
\]
which may be considered as a band-pass filter with only one main lobe and no side lobes, it is important to investigate the analogous filter property of the VM wavelets $\psi_{m;n}(x)$. For convenience, we only consider $n=m$ and provide the plot of $|\hat{\psi}_m(\omega)|$ for $m=3,\ldots,6$ in Figure \ref{fig:SideLobe1}. The reason for plotting $|\hat{\psi}_{m}|$ instead of the spectra $|\hat{\psi}_m|^2$ is that the side-lobes of $|\hat{\psi}_m|^2$ are not visible, even for $m=3$. Observe that the side lobes of $|\hat{\psi}_m(\omega)|$ are visibly much smaller as $m$ increases from $3$ (for quadratic spline) to $6$. In fact, for $m\geq 4$, the side lobes are practically not visible at all. Hence, we also compute the ``side lobe ratio (SLR)'', defined below to compare the energy of the main-lobe (in $0\leq \omega\leq \pi$) versus the total energy of the side lobes in ($\pi\leq \omega<\infty$), measured in dBs:
\begin{align}
\text{SLR}:=20\log_{10}\frac{\|{\hat{\psi}_m}|_{(0,\pi]}\|_{L^2}}{\|{\hat{\psi}_m}|_{(\pi,\infty)}\|_{L^2}}.
\end{align}
In Figure \ref{fig:SideLobe2}, the plot of the SLR curve in dBs of the main lobe to side lobes (for $m=2,\ldots,12$) illustrates that the decrease of the side lobe to main lobes ratios to zero is exponentially fast, as $m$ increases.
\begin{figure}[!t]
\begin{center}
\subfigure[$\psi_3$ and $|\hat{\psi}_3|$]{\includegraphics[width= .23\textwidth]{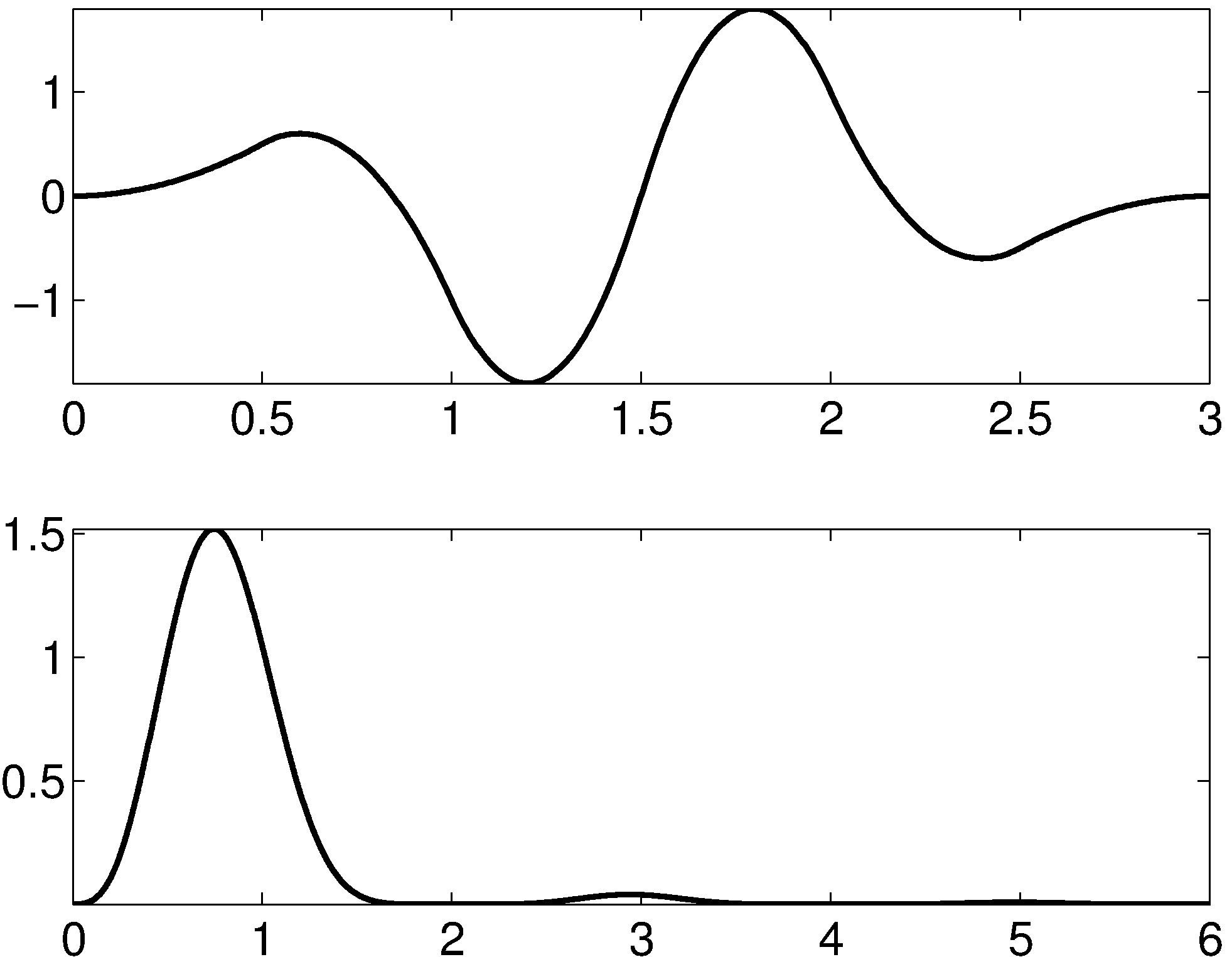}}\hspace{.2cm}
\subfigure[$\psi_4$ and $|\hat{\psi}_4|$]{\includegraphics[width= .23\textwidth]{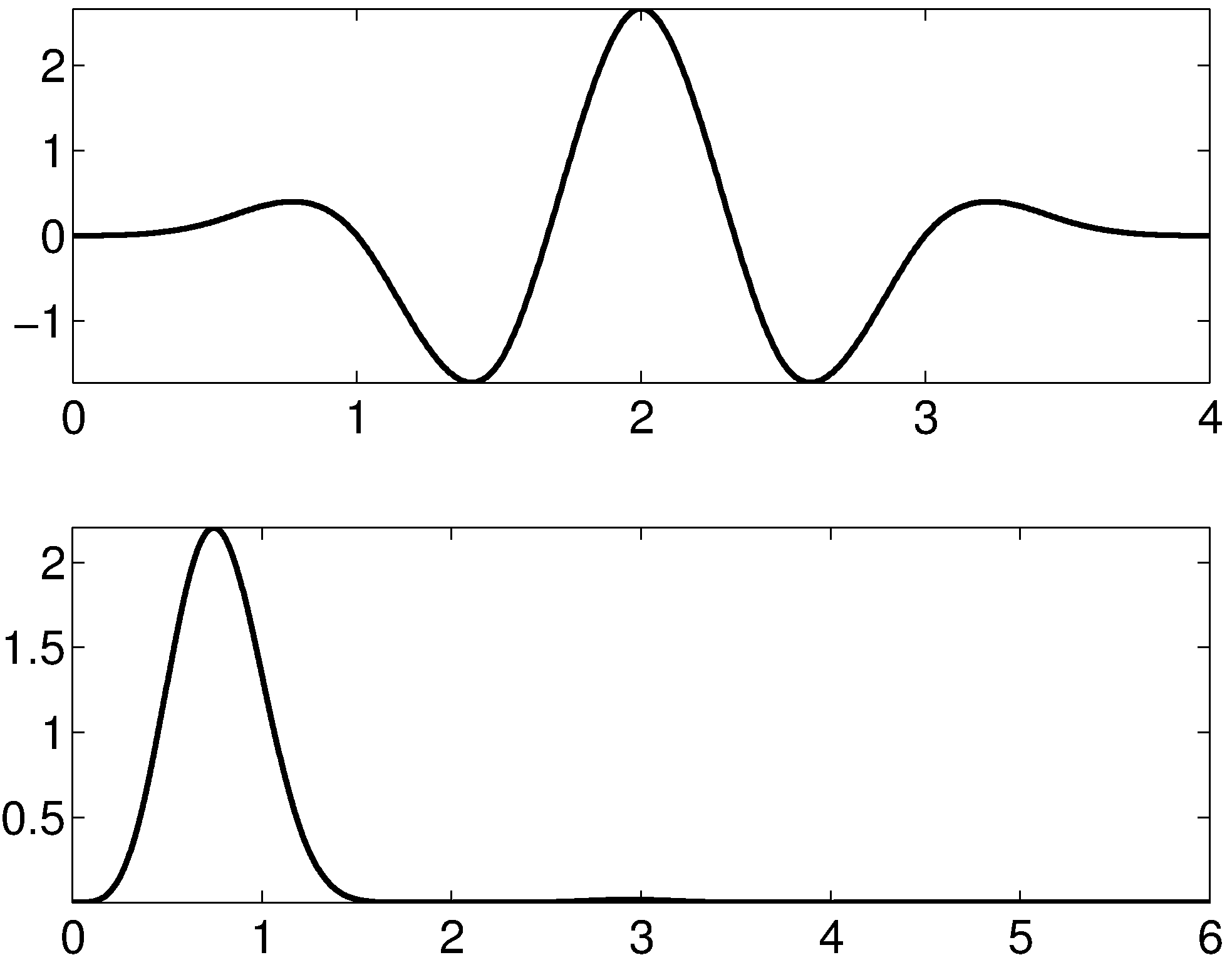}}\hspace{.2cm}
\subfigure[$\psi_5$ and $|\hat{\psi}_5|$]{\includegraphics[width= .23\textwidth]{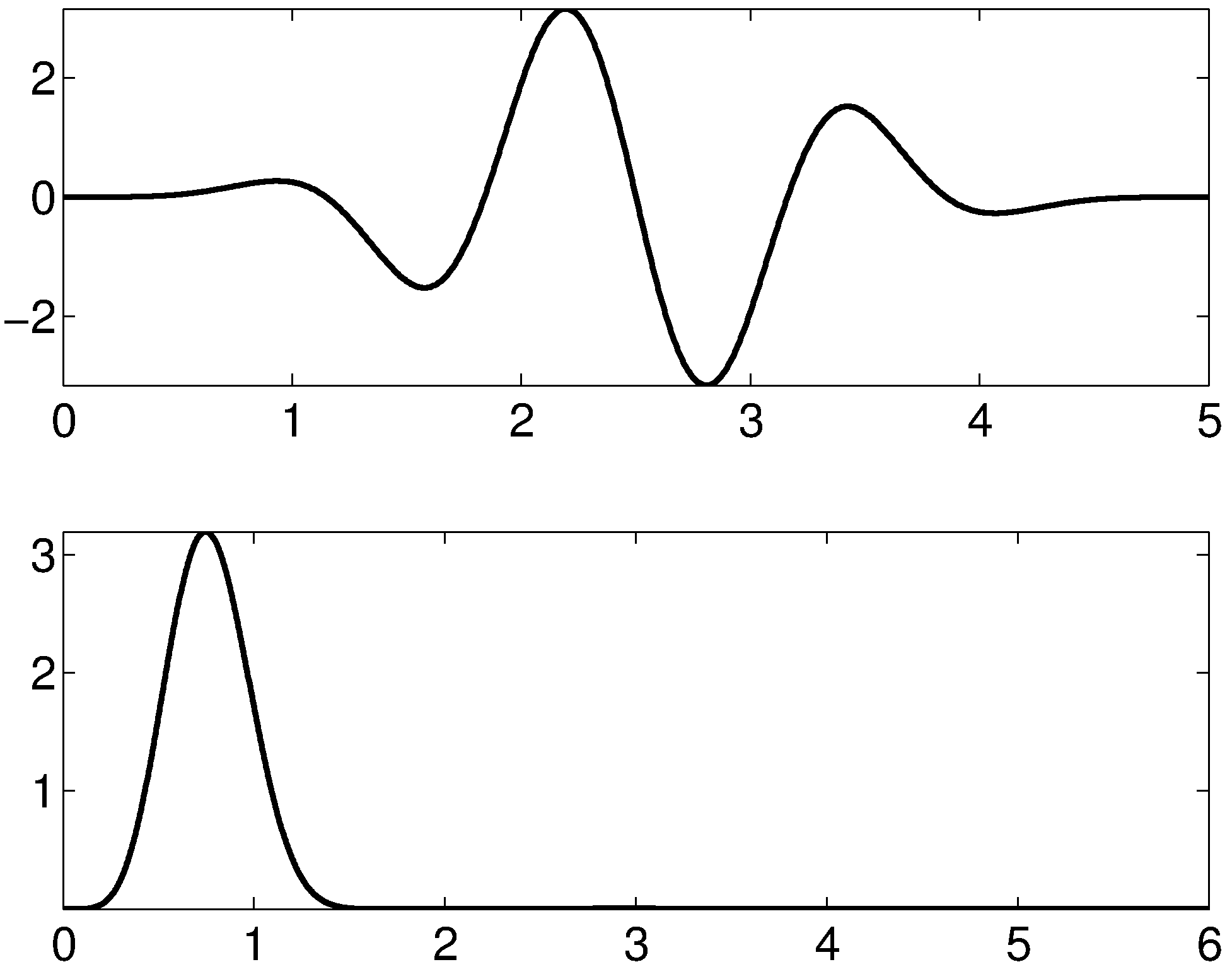}}\hspace{.2cm}
\subfigure[$\psi_6$ and $|\hat{\psi}_6|$]{\includegraphics[width= .23\textwidth]{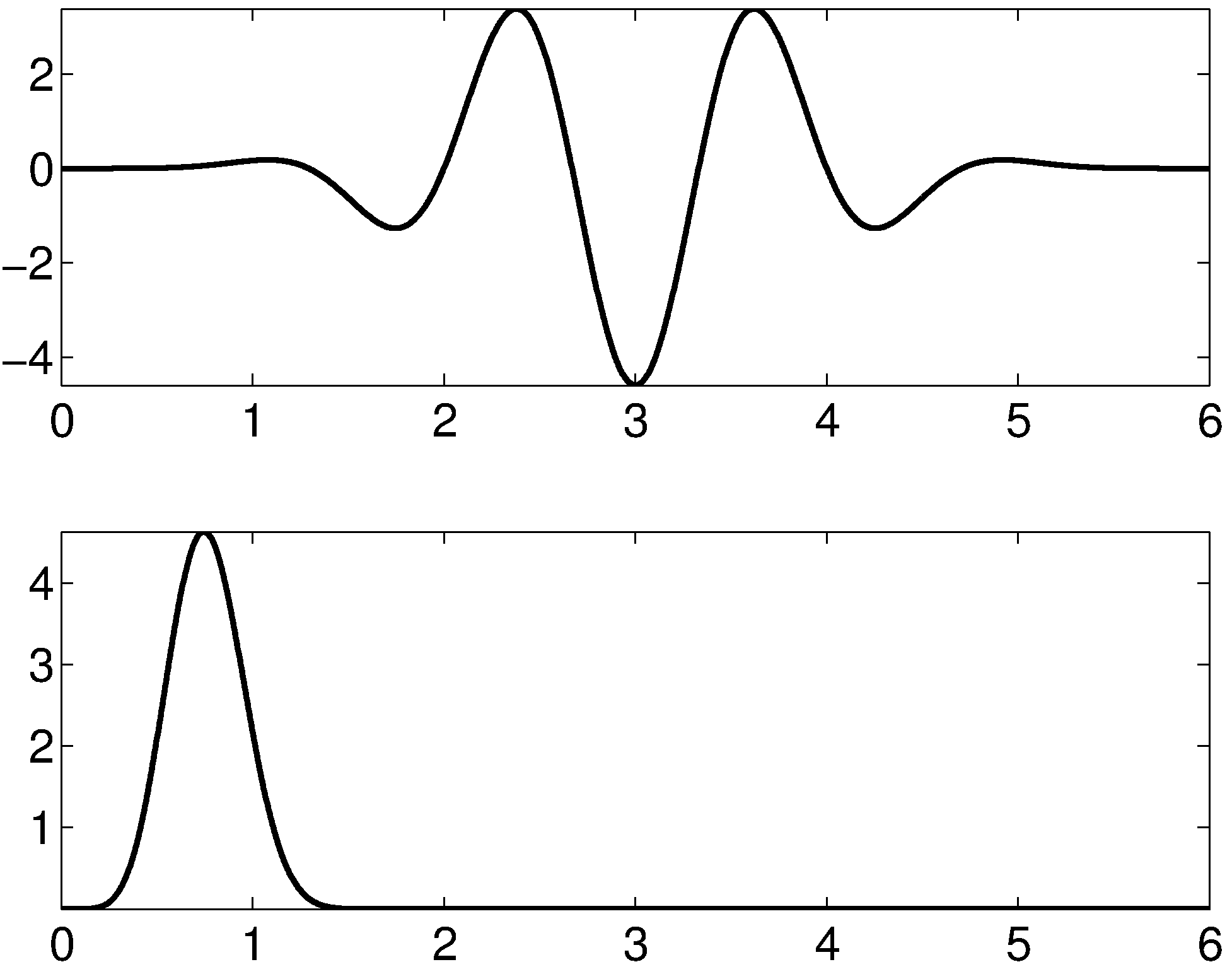}}\\
\subfigure[$\psi_{3,4}$ and $|\hat{\psi}_{3,4}|$]{\includegraphics[width= .23\textwidth]{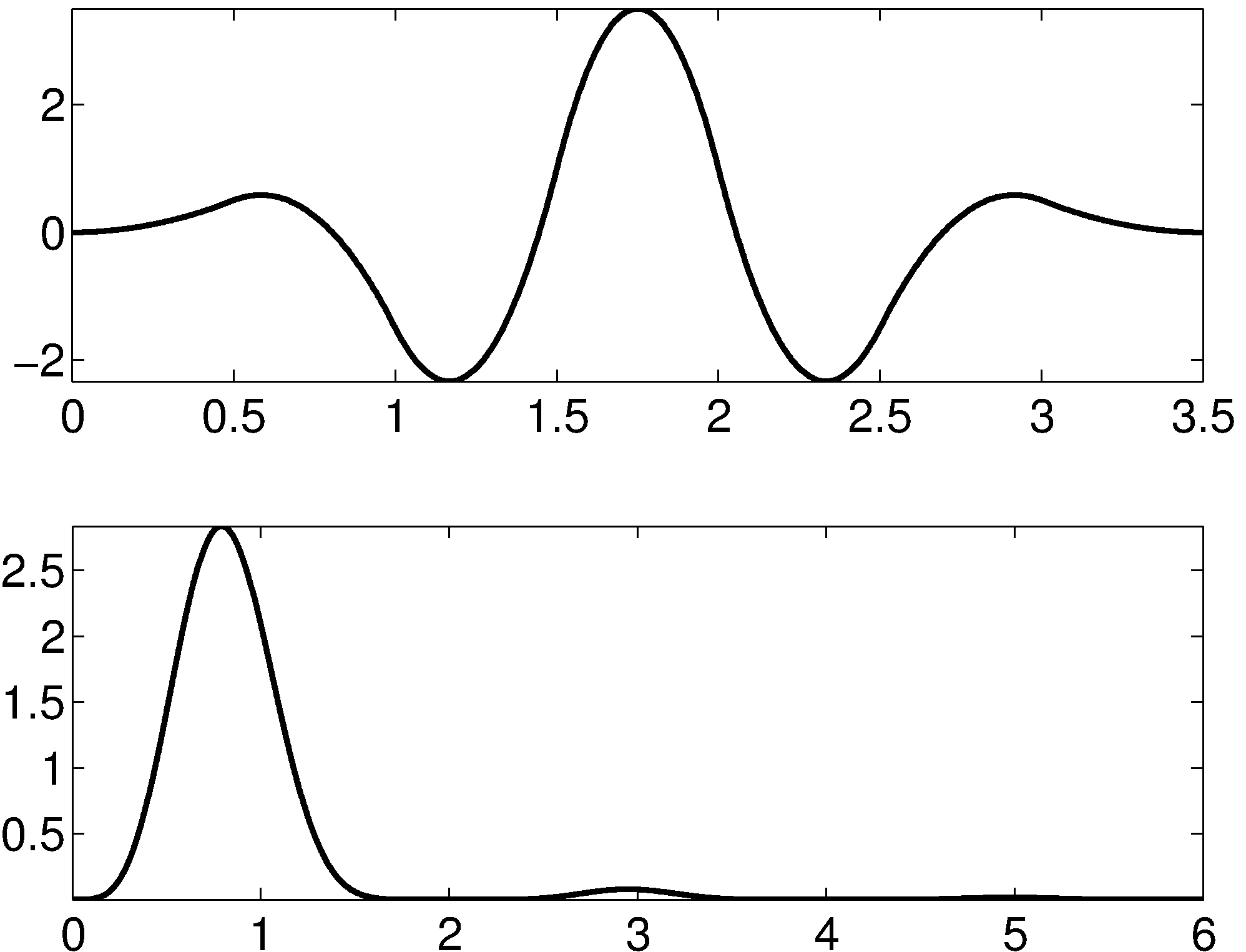}}\hspace{.2cm}
\subfigure[$\psi_{4,5}$ and $|\hat{\psi}_{4,5}|$]{\includegraphics[width= .23\textwidth]{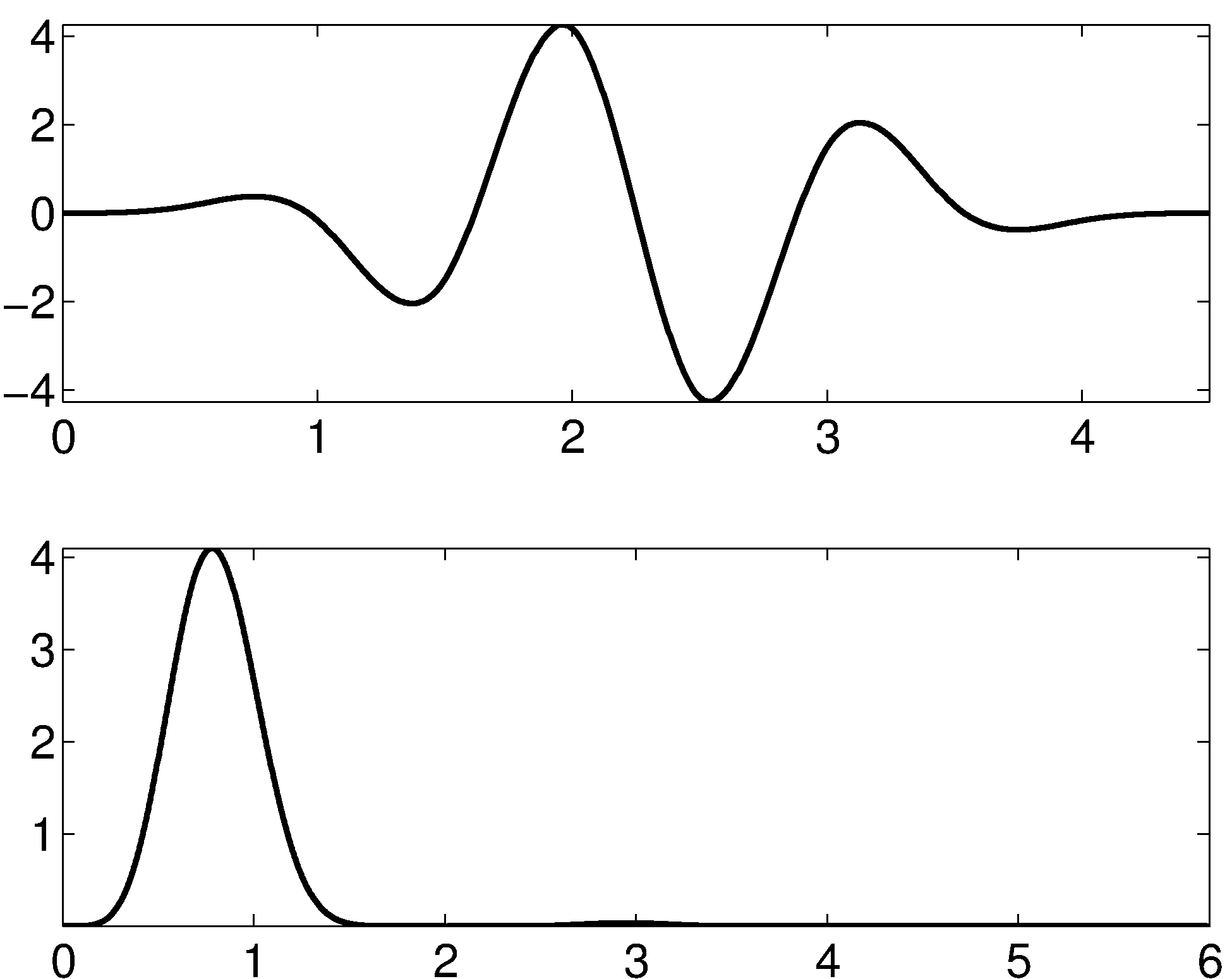}}\hspace{.2cm}
\subfigure[$\psi_{5,6}$ and $|\hat{\psi}_{5,6}|$]{\includegraphics[width= .23\textwidth]{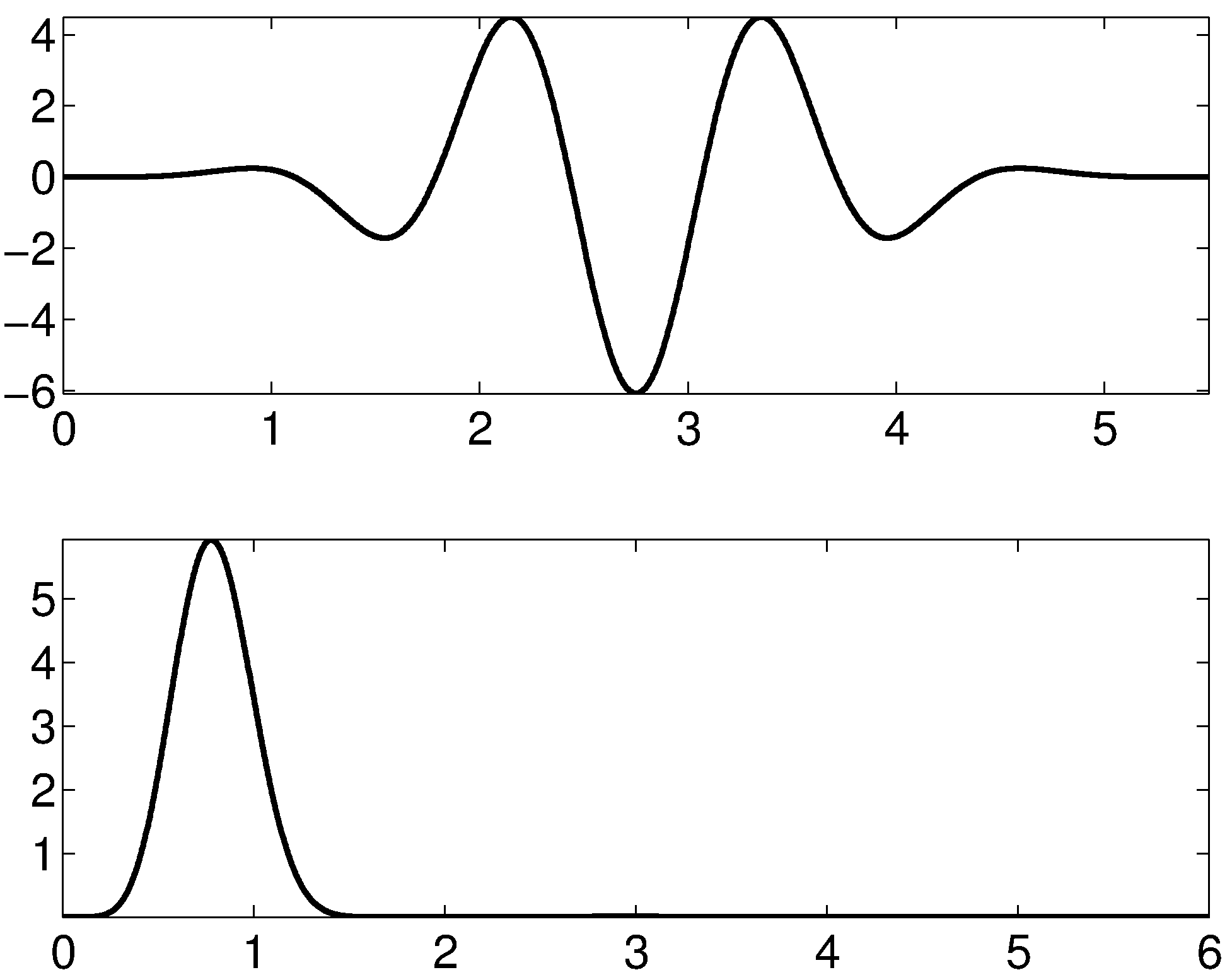}}\hspace{.2cm}
\subfigure[$\psi_{6,7}$ and $|\hat{\psi}_{6,7}|$]{\includegraphics[width= .23\textwidth]{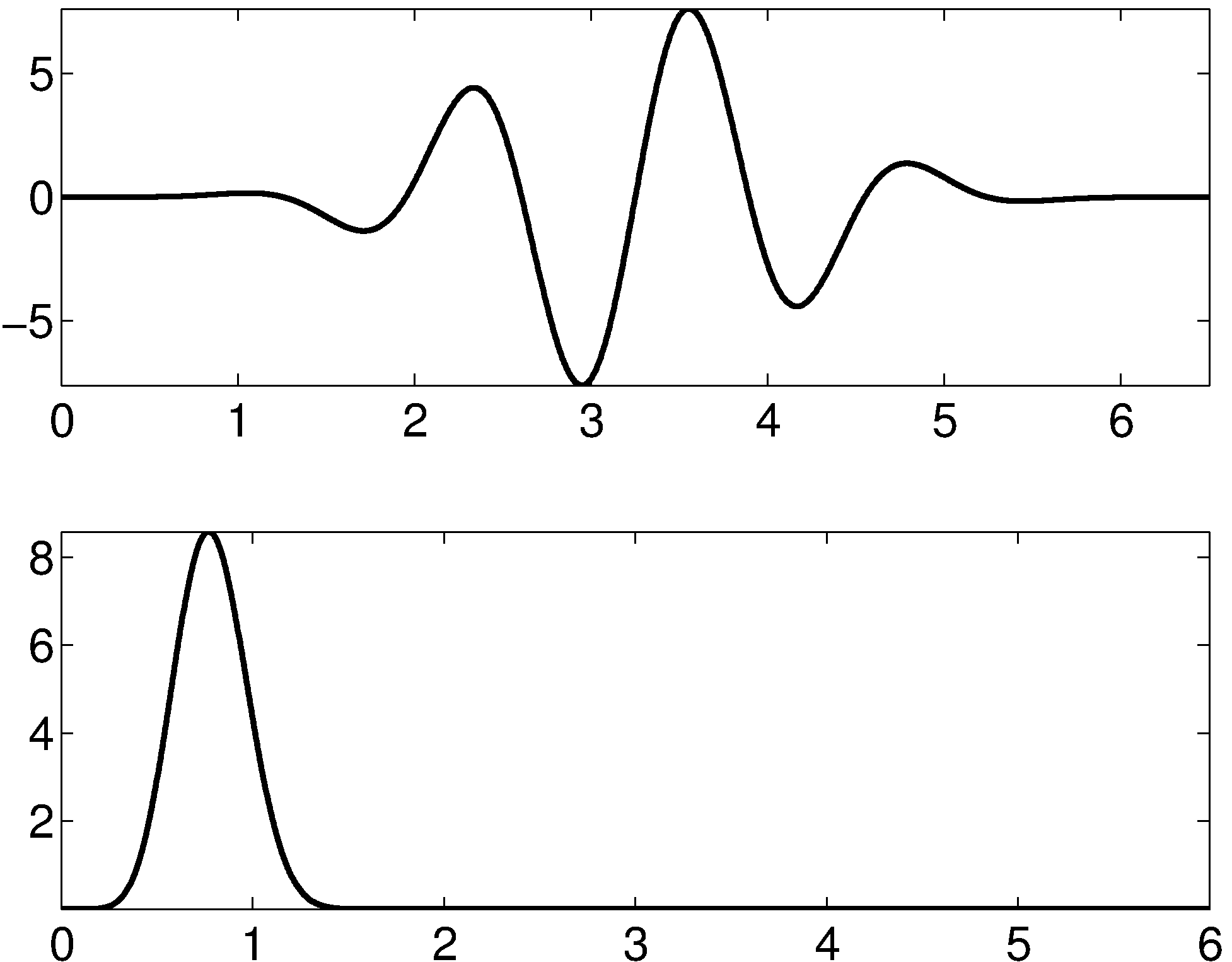}}\end{center}
\caption{From (a) to (d): $\psi_3$, $\psi_4$, $\psi_5$ and $\psi_6$ and the square-root of their power spectra. From (e) to (h): $\psi_{3,4}$, $\psi_{4,5}$, $\psi_{5,6}$ and $\psi_{6,7}$ and the square-root their power spectra. The first side lobe in the square-root of the power spectrum of the quadratic VM wavelets $\psi_3$ and $\psi_{3,4}$ are clearly visible.}
\label{fig:SideLobe1}
\end{figure}

\begin{figure}[!t]
\begin{center}
\includegraphics[width= .35\textwidth]{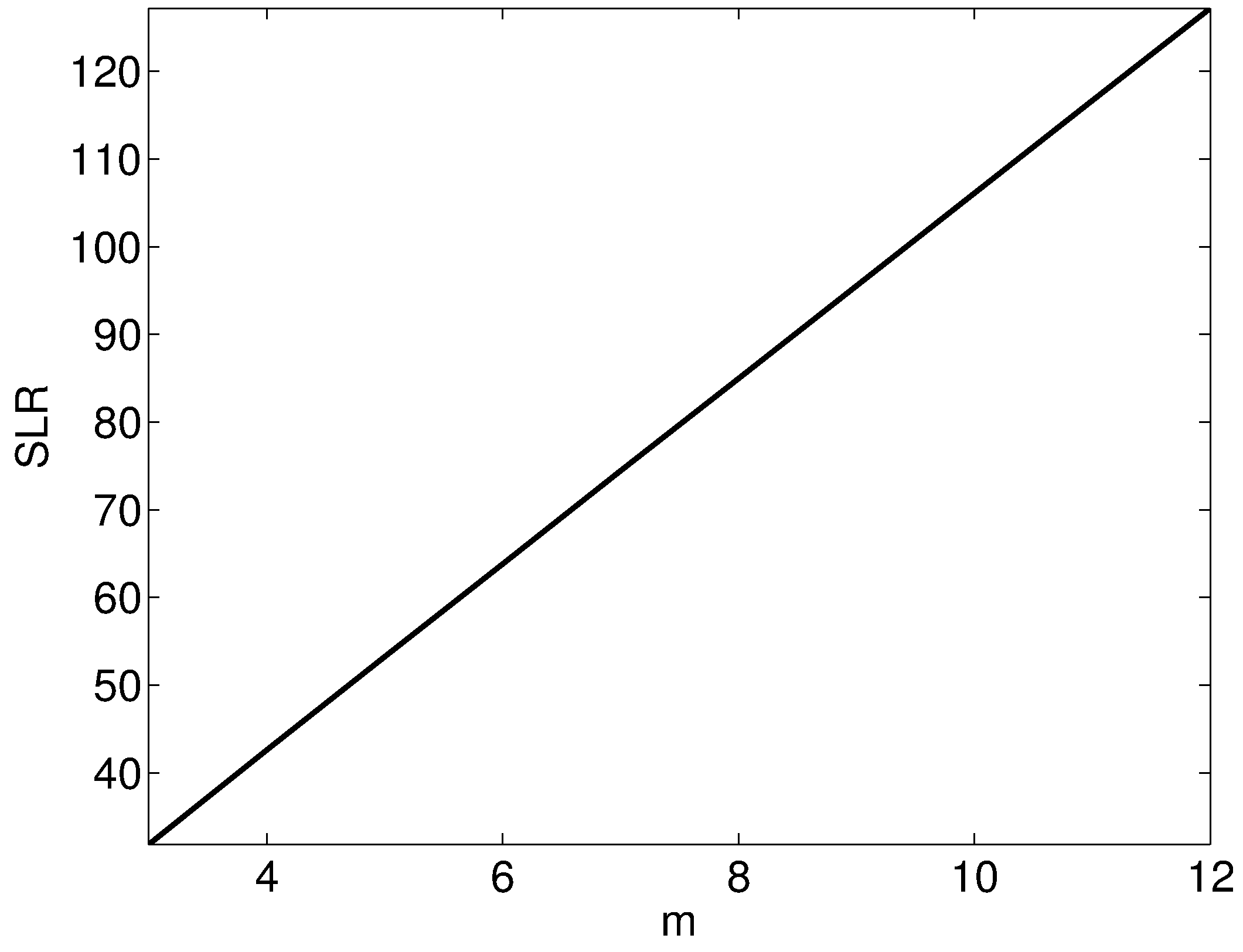}
\end{center}
\caption{The x-axis is used for the orders $m$ of the VM wavelets, and the y-axis is for the SLR, measured in dBs.}
\label{fig:SideLobe2}
\end{figure}

\section{Real-time time-frequency analysis and time-varying power spectrum}\label{section:tvPS}

In many applications, the primary purpose of analyzing a time series is to quantify how the dynamics evolves with respect to time. Of course, the interesting dynamical quantities vary from one field of application areas to another. In any case, for a signal with non-stationary oscillatory behavior, it is important to quantify the term ``frequency'', For instance, consider the respiratory signal as an example of signals with oscillation. It has been well-known that our breathing rate is not constant and the variability of the breathing rate, referred to as breathing rate variability (BRV) (see, for example, \cite{Wysocki2006,Wu_Hseu_Bien_Kou_Daubechies:2013} and the literature inside), contains information about our health condition. To quantify BRV, and hence the information hidden inside, analyzing the power spectrum (PS) of the respiratory signal, denoted by $R(t)$, is the standard yard stick. However, since the notion of frequency is global in nature, the PS is not applicable to momentary oscillatory patterns of $R(t)$. In this section, we focus on applying the proposed spline techniques, including the blending operator and VM wavelets, to acquire this dynamical information in the real-time fashion via time-frequency analysis. The result will be applied to the study of the anesthetic depth problem in the next section.

\subsection{Adaptive harmonic (or non-harmonic) model}
The mathematical model and analysis tools considered have been carefully chosen in the literature, for example, \cite{Daubechies_Lu_Wu:2011,Wu:2011Thesis,Chen_Cheng_Wu:2013}, in order to resolve the limitation of the PS. Let us consider the respiratory signal as an example. As is discussed in \cite{Wu_Hseu_Bien_Kou_Daubechies:2013}, our physiological systems closely interact with each other in a complex way. We treat these kinds of signals by a purely phenomenological model, which quantifies the observable features inside the signal. The main characteristic pattern of an oscillatory physiological signal we focus on is its fairly periodic phenomenon. In general, we consider the following {\it intrinsic-mode typed (IMT) function}:
\begin{equation}\label{decomp1}
g(t) = \sum_{\ell=1}^K A_\ell(t)\cos(2\pi\phi_\ell(t)),
\end{equation}
where $K$ is finite and for each $\ell$ the following conditions are satisfied.
\begin{align}\label{definition:adaptiveHarmonicSingle}
\left\{\begin{array}{l}\vspace{.2cm} 
A_\ell\in C^1(\RR)\cap L^\infty(\RR),\quad\phi_\ell\in C^2(\RR),\\ \vspace{.2cm} 
\inf_{t\in\RR} A_\ell(t)>c_1,\quad \inf_{t\in\RR}\phi'_\ell(t)>c_1,\\ \vspace{.2cm} 
\sup_{t\in\RR} A_\ell(t)\leq c_2,\quad\sup_{t\in\RR}\phi'_\ell(t)\leq c_2,\\ \vspace{.2cm} 
|A'_\ell(t)|\leq \epsilon \phi'_\ell(t),\quad |\phi''_\ell(t)|\leq \epsilon\phi_\ell'(t)\quad\mbox{ for all }t\in\RR,
\end{array}\right.
\end{align}
with $0<\epsilon\ll 1$, $0<c_1\leq c_2<\infty$. In addition, 
\begin{align}\label{definition:adaptiveHarmonicMultiple}
\phi_{\ell+1}'(t)-\phi'_\ell(t)>d(\phi_{\ell+1}'(t)+\phi_{\ell}'(t))
\end{align}
for all $\ell=1,\ldots,K-1$ and $0<d<1$. We call $\epsilon,c_1,c_2,d$ {\it model parameters}.
These seemingly complicated mathematical conditions actually have an intuitive interpretation beyond, leading to their nominations. 
\begin{definition}{[Phase function and Instantaneous frequency]} 
In the IMT function (\ref{decomp1}), for each $\ell$, the monotonically increasing function, $\phi_\ell(t)$, is called the {\it phase function}. The derivative of the phase function, $\phi_\ell'(t)$, is called the {\it instantaneous frequency} (IF) of the $\ell^{th}$ component of the IMT function of $g(t)$. 
\end{definition}
We require the IF to be positive, but usually not constant, allowing it to vary in time, as long as the variations are slight from one ``period'' to the next for all time $t$. 
\begin{definition}{[Amplitude modulation]}
In the IMT function (\ref{decomp1}), the positive function $A_\ell(t)$ is called the {\it amplitude modulation} (AM) of $g(t)$.
\end{definition}
Again, the AM $A_\ell(t)$ is positive, but is allowed to vary. We mention that the motivation and suitability of these definitions have been discussed extensively in \cite{Chen_Cheng_Wu:2013} and we summarize the result here. Note that we might have infinitely many different ways to represent the same cosine function $g(t)=\cos(t)$; indeed, there exist smooth functions $\alpha$ and $\beta$ so that $g(t)=\cos(t)=(1+\alpha(t))\cos(t+\beta(t))$. However, in general, there is no reason to favor $\alpha(t)=\beta(t)=0$, and hence the definition of AM and IF is problematic. In \cite{Chen_Cheng_Wu:2013}, it is shown that if $g(t)=A(t)\cos(\phi(t))=[A(t)+\alpha(t)]\cos(2\pi[\phi(t)+\beta(t)])$ satisfies the IMT condition, then $|\alpha(t)|\leq C\epsilon$ and $|\beta'(t)|\leq C\epsilon$, where $C$ is a constant depending only on the model parameters $c_1,c_2,d$. It is, in this sense, that the IF and AM are well-defined.

\begin{theorem}\label{thm:almost_orthogonal}
For an IMT function $g(t)=\sum_{l=1}^Kf_l(t)$, with $f_l(t)=A_l(t)\cos(2\pi\phi_l(t))$, and for all $t_0\in\RR$ and $L\gg 1$, the following ``almost orthogonal'' relationship holds
\begin{align*}
\left|\int_I \frac{f_k(t)}{\|f_k\|_{L^2(I)}}\frac{f_l(t)}{\|f_l\|_{L^2(I)}}\ud t\right|\leq\epsilon E+F\quad \mbox{for all}\quad k\neq l, 
\end{align*}
where $I:=[t_0,t_0+L]$ and $E$ and $F$ are constants to be precisely given in (\ref{proof:orthogonal:bound}) in the derivation. Moreover, $F$ tends to $0$ as $L\to\infty$. In other words, for $l\neq k$, $f_l$ and $f_k$ are roughly perpendicular to each other for large values of $L$. 
\end{theorem}

\begin{proof}
Suppose that $k>l$. To simplify the notation, we set $\|\cdot\|_I:=\|\cdot\|_{L^2(I)}$. By a direct computation, we observe that
\begin{align*}
 \int_I f_k(t)f_l(t)\ud t= \frac{1}{2}\int_I A_k(t)A_l(t)\big[\cos(2\pi(\phi_k(t)-\phi_l(t)))+\cos(2\pi(\phi_k(t)+\phi_l(t)))\big] \ud t.
\end{align*}
Note that by definition, $\phi'_k(t)-\phi'_l(t)>\left(\left(\frac{1+d}{1-d}\right)^{k-l}-1\right)\phi'_l(t)>0$ for all $t$, allowing us to choose $t_0\leq t_1< \ldots<t_N\leq t_0+L$, so that $\cos(2\pi(\phi_k(t_l)-\phi_l(t_l)))=0$ for all $l=1,\ldots,N$. Thus, we can divide the integration into pieces which are labeled by (\ref{proof:shape:almostOrthogonal:1}) and (\ref{proof:shape:almostOrthogonal:2}) as follows: 
\begin{align}
&\int_I A_k(t)A_l(t)\big[\cos(2\pi(\phi_k(t)-\phi_l(t)))\big]\ud t\nonumber\\
=&\int_{t_0}^{t_1} A_k(t)A_l(t)\cos(2\pi(\phi_k(t)-\phi_l(t)))\ud t\nonumber\\
&\quad+\int_{t_N}^{t_0+L} A_k(t)A_l(t)\cos(2\pi(\phi_k(t)-\phi_l(t)))\ud t\label{proof:shape:almostOrthogonal:1}\\
&\quad+\sum_{l=1}^{N-1} \int_{t_l}^{t_{l+1}} A_k(t)A_l(t)\cos(2\pi(\phi_k(t)-\phi_l(t)))\ud t.\label{proof:shape:almostOrthogonal:2}
\end{align}
The estimate for (\ref{proof:shape:almostOrthogonal:2}) can be achieved simply by applying integration by parts, namely 
\begin{align*}
&\left|\int_{t_l}^{t_{l+1}} A_k(t)A_l(t)\cos(2\pi(\phi_k(t)-\phi_l(t)))\ud t\right|\\
=&\,\left|\int_{t_l}^{t_{l+1}} \partial_t\left(\frac{A_k(t)A_l(t)}{2\pi(\phi'_k(t)-\phi'_l(t))}\right)\sin(2\pi(\phi_k(t)-\phi_l(t)))\ud t\right|\\
\leq\,& \epsilon \left|\int_{t_l}^{t_{l+1}}  \frac{(\phi_k'(t)A_l(t)+\phi_l'(t)A_k(t))(\phi_k'(t)-\phi_l'(t))+A_k(t)A_l(t)(\phi_k'(t)+\phi_l'(t))}{2\pi(\phi'_k(t)-\phi'_l(t))^2}  \ud t\right|\\
\leq\,& \frac{\epsilon}{2\pi (\beta-1)c_1}\int_{t_l}^{t_{l+1}}\big( \phi_k'(t)A_l(t) + \phi_l'(t)A_k(t) +\alpha A_k(t)A_l(t) \big)\ud t,
\end{align*}
where $\beta:= \left(\frac{1+d}{1-d}\right)^{k-l}$ and $\alpha:=\frac{\beta+1}{\beta-1}$. The last inequality holds since $\phi_k'(t)+\phi_l'(t)\leq \alpha (\phi_k'(t)-\phi_l'(t))$.
As a result,
we have the following bound for (\ref{proof:shape:almostOrthogonal:2}):
\begin{align*}
&\sum_{l=1}^{N-1} \int_{t_l}^{t_{l+1}} A_k(t)A_l(t)\cos(2\pi(\phi_k(t)-\phi_l(t)))\ud t\\
\leq&\, \frac{\epsilon}{2\pi (\beta-1)c_1}\int_{[t_1,t_N]}\big( \phi_k'(t)A_l(t)+\phi_l'(t)A_k(t) +\alpha A_k(t)A_l(t) \big)\ud t. 
\end{align*}
Due to the lack of cancellation property, we compute an upper bound of the boundary terms (\ref{proof:shape:almostOrthogonal:1}) simply by $(L+t_1-t_N)c_2^2\leq \frac{2c_2^2}{(\beta-1)c_1}$, where the last inequality holds due to $\int^{t_1}_{t_0}\phi'_k(s)-\phi'_l(s)\ud s = 1\geq (\beta-1)c_1(t_1-t_0)$. 
Hence, by the Holder's inequality, we obtain
\begin{align}
&\left|\int_I A_k(t)A_l(t)\big[\cos(2\pi(\phi_k(t)-\phi_l(t)))\big]\ud t\right|\nonumber\\
\leq\,&\frac{\epsilon}{2\pi (\beta-1)c_1}\left(\|\phi_k'\|_{I}\|A_l\|_{I}+\|\phi_l'\|_{I}\|A_k\|_{I} +\alpha\|A_k\|_{I}\|A_l\|_{I}\right)+\frac{2c_2^2}{(\beta-1)c_1}.\nonumber 
\end{align}   
Similarly, $\phi'_k(t)+\phi'_l(t)> 0$ for all $t$, so we can choose $t_0\leq s_1< \ldots<s_M\leq t_0+L$ so that $\cos(2\pi(\phi_k(s_l)+\phi_l(s_l)))=0$ for all $l=1,\ldots,M$. Again, by a similar argument, we have
\begin{align*}
&\int_I A_k(t)A_l(t)\big[\cos(2\pi(\phi_k(t)+\phi_l(t)))\big]\ud t\\
\leq\,&\frac{\epsilon}{2\pi (\beta+1)c_1}\left(\|\phi_k'\|_{I}\|A_l\|_{I}+\|\phi_l'\|_{I}\|A_k\|_{I} +\|A_k\|_{I}\|A_l\|_{I}\right)+\frac{2c_2^2}{(\beta+1)c_1}. 
\end{align*}
In conclusion, we have $\left|\int_I \frac{f_k(t)}{\|f_k\|_{I}}\frac{f_l(t)}{\|f_l\|_{I}}\ud t\right|\leq E\epsilon+F$, where
\begin{align}\label{proof:orthogonal:bound}
E&:= \left(\frac{ 1}{\beta+1 }+\frac{ 1}{ \beta-1 }\right)\frac{\|\phi_k'\|_{I}\|A_l\|_{I}+\|\phi_l'\|_{I}\|A_k\|_{I} +\alpha\|A_k\|_{I}\|A_l\|_{I}}{4\pi c_1\|f_k\|_{I}\|f_l\|_{I}}\\
F&:= \left(\frac{ 1}{\beta+1 }+\frac{ 1}{ \beta-1 }\right) \frac{2c_2^2}{ c_1\|f_k\|_{I}\|f_l\|_{I}}. \nonumber
\end{align}

\end{proof}
Note that since $\|A_l\|^2_I=\int_I f^2_l(t) \ud t+\int_I  |A_l(t)|\sin^2(2\pi\phi_l(t))\ud t$, we know that $\|f_k\|_{I}\sim \frac{1}{2}\|A_k\|_I$; that is, $\|f_k\|_{I}$ and $\|A_k\|_I$ are of the same order no matter how large $L$ is. Also, by the assumption of $\phi'_k$ and $A_k$, the ratio between $\|A_k\|_I$ and $\|\phi_k'\|_I$ is bounded by $c_2/c_1$ from the above and $c_1/c_2$ from the below. Thus, $F$ in (\ref{proof:orthogonal:bound}), which comes from the boundary, decreases to $0$ when $L$ increases, while $E$ is kept bounded. Also, the $\beta$ showing in (\ref{proof:orthogonal:bound}) indicates that the farther the instantaneous frequencies of two components are separated, the more these two components are perpendicular to each other.

Next, we shall model the commonly encountered ingredient in signal processing -- the trend, of which we consider the following model.
\begin{definition}{[Trend]}
The trend is a continuous function $T(t)\in C^0\cap \mathcal{S}'$ that satisfies 
$$
\left|\int_0^B \langle T(t),\psi^{(a,b)}(t)\rangle\ud a\right|\ll 1,\quad \left|\int_0^B\langle T(t),\partial_t\psi^{(a,b)}(t)\rangle\ud a\right|\ll 1
$$ 
for all $b\in\RR$ and $B\gg 1$, where $\psi\in\mathcal{S}$ is a mother wavelet, $\mathcal{S}$ is the Schwartz space and $\mathcal{S}'$ is the space of tempered distributions. We call $T(t)$ the {\it trend associated with $\psi$ with scale $B$}. 
\end{definition}
The trend characterizes the possibly existing non-oscillatory or slowly oscillatory component inside the signal. Indeed, the condition $|\int_0^B\langle T(t),\psi^{(a,b)}(t)\rangle\ud a|\ll 1$, for $B\gg 1$, suggests that at least locally, the trend behaves like a polynomial. Note that polynomials of any finite order and harmonic functions with very slow frequency are special examples of the trend. Thus, we can view the trend as a ``time varying DC term'' whose local polynomial behavior can be eliminated or minimized by the CWT. 

We mention that polynomials of arbitrary degrees are examples of the trend we consider in this paper. To consider another example, let $\psi(t)$ be any wavelet with compactly supported Fourier transform $\hat{\psi}(\omega)$ that vanishes or $0\leq\omega<1-\Delta$, where $0<\Delta<1$. Then, for any $\xi_0<(1-\Delta)/B$, the harmonic function $\cos(2\pi\xi_0 t)$ can be considered as a trend associated with $\psi$ with scale $B$.

Since all real-world data are contaminated with noise, we consider the following noise model \cite{Chen_Cheng_Wu:2013}. 
\begin{definition}{[Noise]}
Consider a stationary generalized random process $\Phi$ and a Schwartz function $\psi$ such that $\text{var}\Phi(\psi)=1$ with power spectrum $\ud \eta$ of $\Phi$ satisfies $\int (1+|\xi|)^{-2l}\ud \eta<\infty$ for some $l>0$ \cite{Chen_Cheng_Wu:2013}. Then, the noise model we consider will be $\sigma(t)\Phi$, where $\sigma(t)$ is a slowly varying smooth function so that $\|\sigma\|_{L^\infty(\RR)}\ll 1$ and $\max_{\ell=1,\ldots,l}\|\sigma^{(\ell)}\|_{L^{\infty}}\ll 1$. We call $\sigma(t)\Phi$ {\it the noise associated with $\psi$ with non-stationarity $\sigma$}.
\end{definition}
Note that the continuous-time Gaussian white noise commonly considered in the engineering field, is an example of $\Phi$. In addition, the continuous autoregressive and moving average (CARMA) model is also an example of $\Phi$ \cite{Chen_Cheng_Wu:2013}. We mention that when $\sigma$ is a constant, $\sigma\Phi$ is a stationary noise. The purpose of introducing $\sigma$ is to capture the possible non-stationary behavior of the noise. We refer the interested reader to \cite[Section S.1]{Chen_Cheng_Wu:2013} for a quick review of the generalized random process.

In summary, in order to capture an observed time series with oscillatory behavior, we may consider the following {\it adaptive harmonic model associated with a Schwartz function $\psi$}:
\begin{align}\label{decompAdaptive}
Y = \sum_{\ell=1}^K A_\ell(t)\cos(2\pi\phi_\ell(t))+T(t)+\sigma(t)\Phi,
\end{align}
where $\sum_{\ell=1}^K A_\ell(t)\cos(2\pi\phi_\ell(t))$ is the IMT function, $T(t)$ is a trend associated with $\psi$ with scale $B$, and $\sigma(t)\Phi$ is the noise associated with $\psi$ with non-stationarity $\sigma(t)$. Note that in general $Y$ is a generalized random process, since by definition, $\sum_{\ell=1}^K A_\ell(t)\cos(2\pi\phi_\ell(t))+T(t)$ is a tempered distribution.


Lastly, as is discussed in \cite{Wu:2013}, the cosine function is not always suitable for describing oscillations, for example, the respiratory signal. Indeed, inside a respiratory cycle, the time periods of inhalating and exhalating are different. Thus, the adaptive harmonic model can be generalized to the following {\it adaptive non-harmonic model} 
\begin{equation}\label{decompShape}
Y(t) = \sum_{\ell=1}^K A_\ell(t)s_\ell(\phi_\ell(t))+T(t)+\sigma(t)\Phi,
\end{equation}
where (\ref{definition:adaptiveHarmonicSingle}) and (\ref{definition:adaptiveHarmonicMultiple}) are satisfied, and each $s_\ell:[0,1]\to \RR$ is $C^{1,\alpha}$, where $\alpha>1/2$ is a $1$-periodic function with unit $L^2$ norm such that $|\widehat{s}_\ell(k)|\leq \delta |\widehat{s}_\ell(1)|$ for all $k\neq 1$, where $\delta\geq 0$ is a small parameter, and $\sum_{n>D}|n\widehat{s}_\ell(n)|\leq \theta$ for some small parameter $\theta\geq 0$ and $D\in \NN$. The $1$-periodic function $s_\ell(\cdot)$ is called the {\em wave shape function with dominant ratio $\delta$, support $D$ and accuracy $\theta$}. 
We will introduce an algorithm for recovering $s_\ell$ from the observation later. The interested reader is referred to \cite{Wu:2013} for further discussions.

\subsection{Time-frequency analysis}
In general, it is well-known that the PS provided by the Fourier transform is limited to the extraction of the dynamical information of a signal that satisfies the adaptive harmonic (non-harmonic) model. To cope with this limitation, time-frequency analysis was introduced in the literature. Essentially, it is a technique that could be applied to generalize the PS. Intuitively, in order to capture the momentary oscillatory behavior of $R(t)$, we may analyze a small piece of the given $R$ around time $t_0$ to understand how it oscillates. Mathematically, we realize this idea by working with the short-time Fourier transform (STFT) or the continuous wavelet transform (CWT) of $R(t)$ (see, for example, \cite{Flandrin:99} for a detailed discussion and other time-frequency analysis techniques). In this paper, we focus ourselves on the CWT, as defined in (\ref{definition:CWTofg})-(\ref{definition:psiababbrev}). 
We mention that when $Y$ comes from the adaptive harmonic model, $W_Y$ is well defined when the mother wavelet $\psi$ is chosen to be in $\mathcal{S}$.

For example, with CWT, we may define the {\it scalogram} of $g(t)$ by $\big|W_{g}(a, b)\big|^2$ \cite{Flandrin:99}, which reveals the local oscillatory behavior of $g(t)$. We also call $W_{g}(a, b)$ the {\it oscillatory information of the signal associated with scale $a$ at time $b$}. Clearly, with the proposed VM wavelet (or any other compactly supported mother wavelet), we may analyze a given signal with CWT and scalogram in real-time, if we are willing to sacrifice the oscillatory information associated with large scales. However, even if the scalogram could be evaluated in real-time, one big issue about the scalogram that remains unsolved is the Heisenberg uncertainty principle -- for small scale $a>0$, one cannot expect good temporal resolution. To cope with this difficulty and increase the resolution of the time-frequency analysis, the synchrosqueezed CWT transform (SST) was introduced in \cite{daubechies_maes:1996,Daubechies_Lu_Wu:2011}, and will be discussed in the next section. One reason we focus on the synchrosqueezed CWT is the vanishing moment property of the wavelet transform. Indeed, depending on the chosen mother wavelet, the existing polynomial component in the signal is automatically removed.

\subsection{Synchrosqueezed CWT transform and time-varying power spectrum}\label{section:SSTtvPS}

The synchrosqueezed CWT transform (SST) was first introduced in \cite{daubechies_maes:1996} to analyze voice signals. The initial theoretical understanding of SST is provided in \cite{Daubechies_Lu_Wu:2011}, and further studied along with noise and trend in \cite{Thakur_Brevdo_Fuckar_Wu:2013,Chen_Cheng_Wu:2013}. After its introduction, SST has been applied to a wide range of applications. For example, it has been applied to the medical field for the study of sleep depth via the respiratory signal \cite{Chen_Cheng_Wu:2013}, ventilator weaning prediction problem \cite{Wu_Hseu_Bien_Kou_Daubechies:2013}, anesthetic depth estimation via studying the heart rate variability \cite{Lin_Wu_Tsao_Yien_Hseu:2013,Wu_Chan_Lin_Yeh:2014}, cardiopulmonary coupling analysis \cite{Iatsenko_Bernjak_Stankovski_Shiogai_Owen_Clarkson_McClintock_Stefanovska:2013}, etc.
We remark that the SST is a special case of the time-frequency reassignment technique \cite{kodera_gendrin_villedary:1978,Chassande-MottinDaubechiesAuger:97,Chassande-MottinAugerFlandrin:03,Auger_Chassande-Mottin_Flandrin:2012}, and refer the reader to \cite{Auger_Flandrin_Lin_McLaughlin_Meignen_Oberlin_Wu:2013} for a recent review article on their relationship and current progress. In addition, SST can also be defined by using the STFT, which was considered and analyzed in \cite{Wu:2011Thesis,Thakur_Wu:2011}.

We now summarize the SST theory and algorithm for a given observation $Y$ for the adaptive harmonic model. We mention that the definition of the Fourier transform in previous papers on SST, for example \cite{Thakur_Brevdo_Fuckar_Wu:2013,Chen_Cheng_Wu:2013}, is different by a $2\pi$ normalization from what we use in this paper. It is shown that the phase information hidden in CWT contains information explaining the spread-out phenomenon in the time frequency domain. To remedy it, the CWT coefficients are shifted according to certain reassignment rules determined by the phase information. In this paper, we consider the {\it reassignment rule} by
\begin{align*}
\Omega_Y(a,b):= \left\{
\begin{array}{ll}
\displaystyle\frac{-i\partial_b W_{Y}(a,b)}{2\pi W_{Y}(a,b)} & \mbox{when }|W_{Y}(a,b)|>0\\
-\infty & \mbox{when }W_{Y}(a,b)=0.
\end{array}
\right.
\end{align*}
where $\partial_b$ denotes the partial derivative with respect to $b$. 
The synchrosqueezed CWT transform is then achieved by applying this reassignment rule as follows:
\[
S^\Gamma_{Y}(b,\xi):=\int_{\{a:\,|W_Y(a,b)|>\Gamma\}}W_Y(a,b)\frac{1}{\alpha}h\left(\frac{|\xi-\Omega_Y(a,b)|}{\alpha}\right)a^{-1}\ud a,
\]
where $0<\alpha\ll 1$ is chosen by the user, $\Gamma>0$ is the hard thresholding parameter for reducing the numerical error and noise influence, $h$ is a smooth function so that $\frac{1}{\alpha}h(\frac{\cdot}{\alpha})\to \delta$, the Dirac delta distribution, in the weak sense as $\alpha\to 0$, and the scale factor $a^{-1}$ is used for reconstruction. Note that for each fixed time $b$, the CWT coefficient $W_Y(a,b)$ at scale $a$ is moved to a new frequency position according to the reassignment rule $\Omega_Y(a,b)$. We refer to the square of the modulation of the synchrosqueezed CWT transform, denoted as $V_Y(b,\xi)$, as {\it time-varying power spectrum} (tvPS) of $Y$; that is,
\[
V_Y(b,\xi):=|S_{Y}(b,\xi)|^2.
\] 
The efficacy of tvPS will be shown in the following numerical subsection. Now we summarize the advantages of SST relevant to our study below:
\begin{enumerate}
\item[(a)] SST is robust to several different kinds of noise, which might be slightly non-stationary \cite[Theorem 3.1]{Chen_Cheng_Wu:2013}; 
\item[(b)] SST is visually informative \cite[Theorem 3.1 (ii)]{Chen_Cheng_Wu:2013}. When the signal is rhythmic, according to the model (\ref{decomp1}), a dominant curve following the IF can be seen on the tvPS; otherwise the tvPS is blurred ;
\item[(c)] SST is {\it adaptive} to the data in the sense that its dependence on the chosen window function is weak \cite[Theorem 3.1]{Chen_Cheng_Wu:2013}. (This property is important in the sense that the notion ``local'' is automatically determined by the data, and hence the model bias can be reduced.)
\item[(d)] SST analysis is local in nature (see the proof of \cite[Theorem 3.1]{Chen_Cheng_Wu:2013}). Indeed, although the conditions about the adaptive harmonic model is stated in the global sense, the analysis simply depends on the local condition of the signal. In this sense, SST, and hence tvPS, can be used to capture the local dynamic behavior of the signal.
\end{enumerate} 
These properties have been theoretically justified, and we refer the interested reader to \cite{Daubechies_Lu_Wu:2011,Chen_Cheng_Wu:2013} for the proof. The discretized companions of these theorems are discussed in \cite{Chen_Cheng_Wu:2013}, and we refer the reader for the details there. We mention that the condition that the mother wavelet is a Schwartz function is critical in the proof in \cite{Chen_Cheng_Wu:2013} when the adaptive harmonic model is considered. Indeed, the noise, as a generalized random process, and the trend, as a tempered distribution, cannot be analyzed in general if the Schwartz condition is removed. So, if we use a non-Schwartz mother wavelet, such as the VM wavelets, we have to modify the definition of noise and trend. Since the modification of the proof of the theorem based on these modifications is straightforward but requires additional notations, we will skip the discussion here.

\subsection{Synchrosqueezed CWT transform and its numerical limitations}

However, from the numerical viewpoint, we have to carefully discuss its implementation limitations. Firstly, the CWT and SST should be implemented in the time domain to avoid undesirable numerical artifacts introduced by the Fourier transform. For example, suppose we have the constant function $f(t)=c\in\RR$, theoretically $W_f(a,b)=0$ for all $a>0$ and $b\in\RR$. However, this is not the case if we only have observation of the signal on a bounded time interval, and this situation is more serious if the observation time is not long enough or if the constant $c$ is quite large. 

Secondly, the differentiation of the CWT in the reassignment rule for the SST algorithm should be exact in order to avoid any possible numerical inaccuracy. Compared with the other mother wavelets with compact supports in the time domain, like the orthogonal Daubechies wavelets \cite{Daubechies:1992} and semi-orthogonal spline wavelets \cite{Chui:1992}, the proposed VM wavelets enjoy this property. Indeed, Corollary \ref{Corollary:derivativeofpsimnx} assures that the derivative of any VM wavelet is another VM wavelet. This explicit representation allows us to ease the numerical issue of differentiation.

To sum up, as an application of the VM wavelets, we propose a new numerical implementation of the SST which resolves some of these limitations.

\subsection{Causality and real-time implementation of CWT, SST and tvPS}

In the above discussion of CWT and SST, the causality issue in data analysis and real-time computing is not considered. Indeed, the mother wavelet such as the Meyer and Morlet wavelets might have infinite support in the time domain so that the future information might contaminate the present information of interest. In addition, real-time implementation is not feasible due to the need of truncation of the chosen mother wavelet in the time-domain. To alleviate these limitations, we might need to choose a mother wavelet with a compact support in time-domain such as the Daubechies wavelets \cite{Daubechies:1992} or spline wavelets \cite{Chui:1992}. However, due to various technical issues, we may only obtain the precise time-frequency information with a lag, depending on the chosen window. Thus, it would be more beneficial to choose a wavelet with the smallest support in the time-domain. On the other hand, we might also want to have the freedom to control the vanishing moment, since in some applications the trend might be modeled by polynomials of lower order. In addition, for the numerical purpose, it is beneficial to have a precise formula for all the kernels involved in the integration. For example, if we can have a precise formula of the derivative of the mother wavelet instead of numerical differentiation, the information we obtain will be more precise.  

Thus, based on the above discussion, we propose to use the compactly supported VM wavelet $\psi_{\ux,m;n,k}(x)$, where $\ux$ is the uniform knot sequence and $m,n\geq 3$, as the mother wavelet to meet all of the requirements mentioned above. For example, by extending
\[
-\partial_bW_g(a,b)=\langle g, \partial_b\psi^{(a,b)}\rangle
\] 
to
\[
\langle g, \partial_b\psi^{(a,b)}_{\ux,m;n,k}\rangle=\langle g, \psi^{(a,b)}_{\ux,m-1;n+1,k}\rangle 
\]
which follows from Corollary \ref{Corollary:derivativeofpsimnx}, where $\psi_{\ux,m;n,k}$ are VM wavelets on an arbitrary knot sequence $\ux$, the reassignment rule $\Omega(a,b)$ for the SST can be carried out precisely without differentiation.

The most important feature of the proposed solution is actually to settle the lag issue for real-time evaluation. However, near the boundary we may not be able to estimate the CWT coefficient with large scale, or we have to sacrifice the accuracy by either extending the signal or truncating the mother wavelet. With the VM wavelet defined on the nonuniform grid, this issue can be relieved by the extension of the VM wavelet to non-uniform knots, including the derivative of the CWT, by using a boundary wavelet as the mother wavelet. We now introduce the algorithm, as follows.

In this section, the VM wavelets
\begin{align}\label{definition:psij}
\psi_j:=\psi_{\tus,m;m,j}
\end{align}
in (\ref{condition:wavelet:definition}) with $n=m$ and knot sequence $\tilde{\boldsymbol{s}}=\frac{1}{2}\us$ (where $\us$ and $\tilde{\boldsymbol{s}}$ are defined in (\ref{equation:underlines}) and (\ref{equation:tilde_s}) respectively) will be used as the mother wavelets for the wavelet transform of the function $f(t)$, which represents the oscillatory pattern, as defined by the adaptive harmonic model (\ref{decompAdaptive}). 

Here, the spline order $m$ is fixed and the interior VM wavelets $\psi_j$, with $j=0,\ldots, 2(N-m)$, are precisely translations of the first one 
\begin{align}\label{definition:psi0}
\psi_0=\psi_{\tus,m;m,0}
\end{align}
by $j/2$; that is, 
\begin{align}\label{definition:psij2}
\psi_j(x)=\psi(x-j/2),\quad j=0,\ldots,2(N-m).
\end{align}
The omission of various subscripts in the notation (\ref{definition:psij}), (\ref{definition:psi0}) and (\ref{definition:psij2}) allows us to use the same notation defined in (\ref{definition:psiababbrev})
with $\psi$ as defined in (\ref{definition:psi0}), and $b=j/2$, for $j=0,\ldots, 2N-2m$. 
We will denote the ``center'' interior wavelet by
\begin{align}\label{definition:centerpsi}
\psi(x):=\psi_0\left(x-\frac{N-m}{2}\right)=\psi_{\tus,m;m,N-m}(x),
\end{align}
and apply the scaling and translation notation in (\ref{definition:psiababbrev}) to $\psi(x)$ in (\ref{definition:centerpsi}). Observe that for $a>0$, $\psi(x/a)$ is a VM wavelet on the knot sequence.
\[
\tus_a:\, 0<\frac{a}{2}<a<\frac{3a}{2}<\ldots<N,
\]
and has support inside the time-interval $[0,N]$, as long as the scale $a$ is restricted to $0<a<N/m$. Furthermore, the support of $\psi(x/a)$ does not overlap with the supports of the boundary wavelets (at both $t=0$ and $t=N$) as long as the scale $a$ is restricted to 
\begin{align}\label{condition:a:boundary}
0<a\leq \frac{N+1}{m}-2.
\end{align}
Hence, to allow large values of $a$ for analyzing low-frequency oscillations, we need a longer time interval $[0,N]$. 
\newline\newline
\textbf{Case 1.} ($1<a\leq \frac{N+1}{m}-2$ for low frequency oscillation analysis)\newline
\newline
For $a>1$, if $a$ also satisfies (\ref{condition:a:boundary}), then we may apply the scaling and translation operation to the CWT $\langle f,\psi^{(a,b)}\rangle$ on the time interval $[0,N]$ and allow the use of boundary wavelets to take care of the CWT $\langle f,\psi_{\tus,m,j}\rangle$ for $j=-m+1,\ldots,-1$ (for the end-point $t=0$ of the time-interval $[0,N]$) and for $j=2N-2m+1,\ldots,2N-m-1$
 (for the other end-point $t=N$). 
\newline\newline 
\textbf{Case 2.} ($0<a<1$ for high-frequency oscillation)\newline\newline 
Observe that since the VM wavelets have compact support, scaling by $a$, $0<a<1$, can also be applied to the interval
$[\,N-m+\frac{1}{2},N\,]$,
on which the $(m-1)$ boundary wavelets $\psi_{\tus,m;m,j}$, $j=2N-2m+1,\ldots, 2N-m-1$, live. More precisely, the $2m-1$ knots
\begin{align}\label{index:boundary}
N-m+1,\,N-m+3/2,\ldots,N-1/2
\end{align}
of $\psi_{\tus,m;m,j}$, $j=2N-2m+1,\ldots,2N-m+1$, in the open interval $\big(N-m+1/2,N\big)$ can be replaced by $K$ equally spaced knots, with $K>2m-1$, namely,
\begin{align}\label{index:boundaryExtension}
y_k:=N-m+\frac{1}{2}+k\left(\frac{m-1/2}{K+1}\right),\quad k=1,\ldots,K.
\end{align}
(Note that $y_0=N-m+1/2$ and $y_{K+1}=N$ are the two end-points of $\big[(N-m)+1/2,N\big]$) Replacement of the knots in (\ref{index:boundary}) by the knots $\{y_1,\ldots,y_K\}$ in (\ref{index:boundaryExtension}) is equivalent to scaling by 
\begin{align}\label{index:boundary_a}
a:=\frac{m-1/2}{K+1}=\frac{2m-1}{2(K+1)}.
\end{align}
We remark that since there are precisely $m-1$ boundary wavelets at the end-point $N$, the $K-2m+1$ wavelets $\psi_{\tus_a,m;m,j}$, $j=2N-2m+1,\ldots,2N-4m+K+1$, with $a$ given by (\ref{index:boundary_a}) are interior wavelets with the new equally-spaced knots $\{y_k\}$ in (\ref{index:boundaryExtension}), while the $m-1$ wavelets $\psi_{\tus_a,m;m,j}$, $j=2N-4m+K+2,\ldots,2N-3m+K$, are boundary wavelets. Hence, for $j=2N-2m+1,\ldots,2N-4m+K+1$, the wavelets $\psi_{\tus_a,m;m,j}$ are translations of a single wavelet, namely
\[
\psi_{\tus_a,m;m,j}(x)=\psi_{\tus_a,m;m,2N-2m}(x-(j-(2N-2m)a)),
\]
where $a$ is given by (\ref{index:boundary_a}).

\subsection{Numerical Implementation and Simulation}

The numerical real-time implementation of the SST and tvPS are summarized in Algorithm \ref{alg:Tf}. The Matlab code can be downloaded from \url{https://sites.google.com/site/hautiengwu/home/download}.
\begin{algorithm}[h!]
  \begin{algorithmic}
\STATE $\bullet$ Choose $m,n\geq 3$. Take the VM wavelets $\psi_1:=\psi_{\tus,m;n,0}$ and $\psi_2:=\psi_{\tus,m-1;n+1,0}$. 
\STATE $\bullet$ Evaluate the analytic representation of $\psi_1$ and $\psi_2$, denoted by $\tilde{\psi}_1$ and $\tilde{\psi}_2$, by the Hilbert transform (for example, by applying \cite[Section 5, p.634--p.635]{Micchelli_Xu_Yu:2013} for the special case of uniform knots); 
\STATE $\bullet$ Input signal $Y$ is discretized with sampling period $\Delta T$. 
\STATE $\bullet$ Choose the allowed lag time to be $L$ seconds so that $L>(m+n)\Delta T/2$. Denote $M=\lfloor L/\Delta T\rfloor$ and $N=2M$.
\STATE $\bullet$ Choose $1/2\Delta T$ (resp. $1/2L$) to be the largest (resp. lowest) frequency of interest. Discretize $[1/2L,1/2\Delta T]$ by $n_\xi$ uniform grids.  
\STATE $\bullet$ Build up a matrix $\Psi\in \CC^{(N-m-n+1)\times N}$ so that the $l^{th}$ entry of the $i$ row is $\tilde{\psi}_1^*((m+n)(l-i)/(N-i+1))$, where ${}^*$ means the complex conjugate.
\STATE $\bullet$ Build up a matrix $\Lambda\in \CC^{(N-m-n+1)\times N}$ so that the $l^{th}$ entry of the $i$ row is $\tilde{\psi}_2^*((m+n)(l-i)/(N-i+1))$. 
\STATE $\bullet$ Start from time $0$ and wait up to time $N\Delta T$. Set the current time index to be $I$. 
\STATE 
\WHILE{New input}{
\STATE $\bullet$ $\boldsymbol{g}_I\gets[\,g(I-N+1),\ldots,g(I)\,]\in\RR^{N}$;
\STATE $\bullet$ $W_I\gets\Psi \boldsymbol{g}_I\in \CC^{N-m-n+1}$ ; \COMMENT{Evaluate the CWT at time $(I-M)\Delta T$.}
\STATE $\bullet$ $Z_I\gets\Lambda \boldsymbol{g}_I\in \CC^{N-m-n+1}$ ; \COMMENT{Evaluate the partial derivative of the CWT with respect to time $b$ at time $(I-M)\Delta T$.}
\STATE $\bullet$ $\Omega_I\gets \frac{iZ_I}{2\pi W_I}\in \CC^{N-m-n+1}$, where the division is evaluated entry-wisely; \COMMENT{Evaluate the reassignment rule at time $(I-M)\Delta T$.}
\STATE $\bullet$ Initiate a zero vector $S_I \in\CC^{n_\xi}$ and evaluate the SST by
\FORALL{$j=1,\ldots,N-m-n+1$}
\STATE $k\gets \text{ROUND}\left[\frac{\Omega_l(j)-1/2L}{1/2\Delta T-1/2L}n_\xi\right]$
\IF{$1\leq k\leq n_\xi$}
\STATE $S_I(k)\gets S_I(k)+cW_I(j)a^{-1/2}$
\ENDIF
\ENDFOR 
\STATE $\bullet$ $V_I\gets |S_I |^2\in\RR_+^{n_\xi}$, where the absolute value and square operators are evaluated entry-wisely. \COMMENT{Evaluate the tvPS at time $(I-M)\Delta T$.}

\STATE $I\gets I+1$ \COMMENT{$\Delta T$ second passed.}

}\ENDWHILE
\end{algorithmic}
\caption{Summary of the Numerical Implementation}
\label{alg:Tf}
\end{algorithm}

Next we show a numerical simulation of the algorithm. See Figure \ref{figS3} and Figure \ref{figS4} for an example that illustrates how well the tvPS captures the dynamical behavior of the signal. In this example we generate the simulated signal in the following way. Simulate the realization of the following random vector of length $M>1$
\[
\tilde{\phi}_1(l\Delta t)=(W\star K_{\sigma_1})(l\Delta t),
\] 
where $1\leq l\leq M$, $\Delta t=1/32$, $W$ is the standard Brownian motion, $K_{\sigma_1}$ is the Gaussian function of standard deviation $\sigma_1$ and $\star$ denotes the convolution operator. The phase function of the first component is defined as 
\[
\phi_1(K\Delta t)=\Delta t\sum_{l=1}^K 2\frac{\tilde{\phi}_1(l) + 2\max_{k=1}^M |\tilde{\phi}_1(k)|}{\max_{l=1}^M\tilde{\phi}_1(l) + 2\max_{k=1}^M |\tilde{\phi}_1(k)|}.
\]
Then, define the AM of the first component by a similar way. For the realization of the random vector of length $M>1$
\[
\tilde{A}_1(l\Delta t)=(W\star K_{\sigma_2})(l\Delta t),
\]
where $1\leq l\leq M$, set
\[
A_1(K\Delta t)=\Delta t\sum_{l=1}^K 2\frac{\tilde{A}_1(l) + 2\max_{k=1}^M |\tilde{A}_1(k)|}{\max_{l=1}^M\tilde{A}_1(l) + 2\max_{k=1}^M |\tilde{A}_1(k)|}.
\]
We now consider
\[
f_1(l\Delta t):=A_1(l\Delta t)\cos(2\pi \phi_1(l\Delta t))\chi_{1\leq l\leq 18.75/\Delta t}
\]
as the first component. Here, $1\leq l\leq M$ and $\chi$ is the indicator function of $l$.
The second component, denoted by $f_2(t)$, is generated in the same way as that of $f_1(t)$, and the trend, denoted by $T(t)$, is generated in the same way as that of $A_1(t)$, but with non-zero mean removed; that is, the mean of $T(t)$ is zero. As a result, the final observed signal is $f(t)=f_1(t)+f_2(t)+T(t)$.
Note that in general, there is no close form expression of $A_1(t)$ and $\phi_1(t)$, and the dynamic of both components can be visually seen from the signal. The signal is shown in Figure \ref{figS2}. To model the noise, we define the signal to noise ratio (SNR) as
\begin{align}
\text{SNR}:=20\log_{10}\frac{\text{std}(f_1+f_2)}{\text{std}(\Phi)},
\end{align}
where $\text{std}$ means the standard deviation and $\Phi$ is the noise process. In the following simulations, we choose the VM wavelet $\psi_{9,9}$ as the mother wavelet with the lag time of $2$ seconds.
In Figure \ref{figS3}, we show the result of SST when the signal is not contaminated by the noise, and in Figure \ref{figS4}, we show the result of SST when the signal is contaminated by $5$dB noise. 
Note that unlike the PS of $f(t)$, it is somewhat obvious that the time-varying dynamics is faithfully reflected in the tvPS. For more robustness results of SST against different noises, we refer the reader of interest to \cite{Chen_Cheng_Wu:2013}.

\begin{figure}[ht]
\includegraphics[width=.65\textwidth]{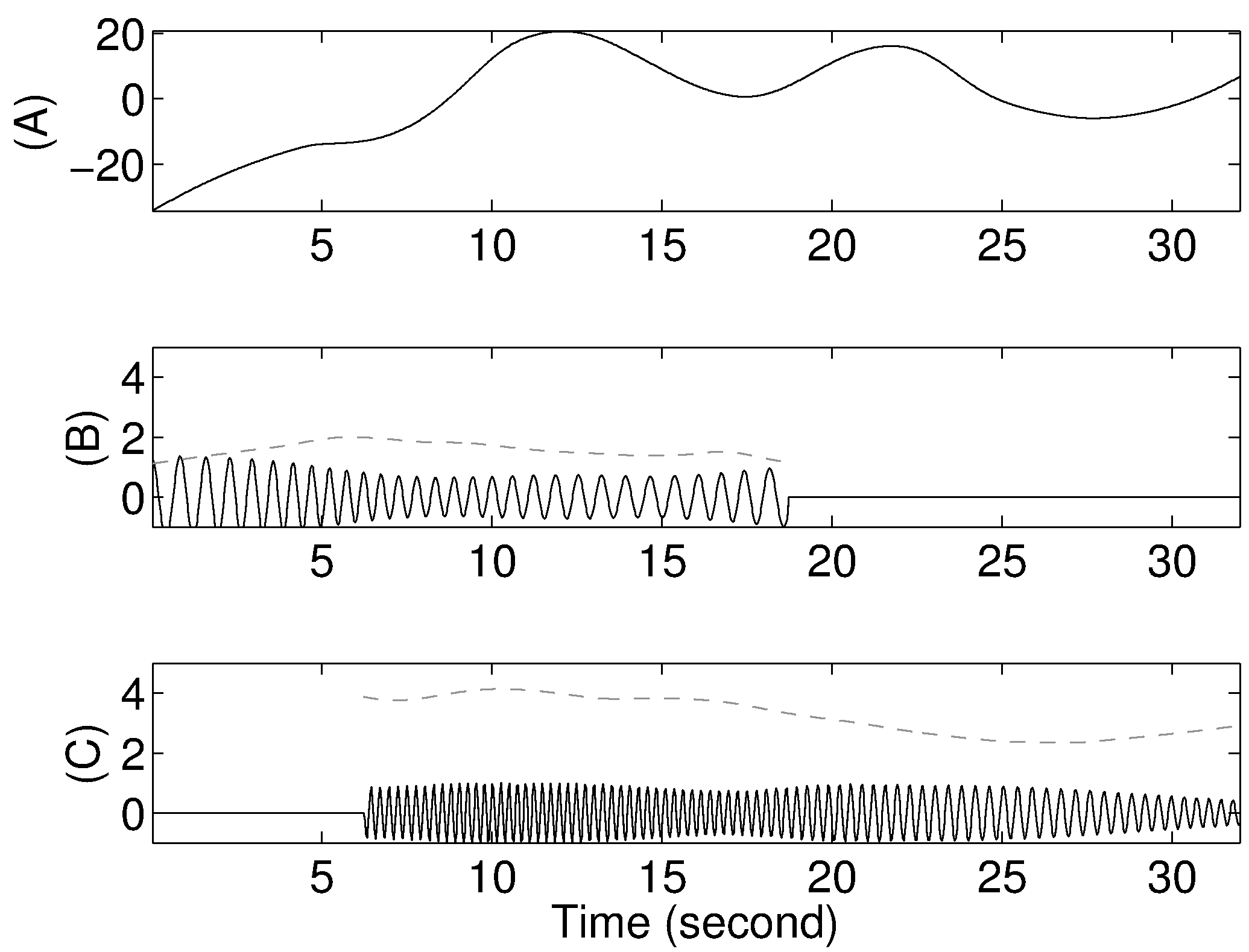}
\caption{The trend $T(t)$ is shown in (A). Note that the mean of trend is $0$, and its value is $-20$ in the beginning and about $10$ in the end. The first component $f_1(t)$ and its instantaneous frequency are shown in (B) as the black curve and the gray dashed curve respectively, and the second component $f_2(t)$ and its instantaneous frequency are shown in (C) as the black curve and the gray dashed curve respectively. Note that $f_1(t)$ exists up to $18.75^{th}$ second and $f_2(t)$ does not exist before the $6.25^{th}$ second. }   
\label{figS2}
\end{figure}

\begin{figure}[ht]
\includegraphics[width=1\textwidth]{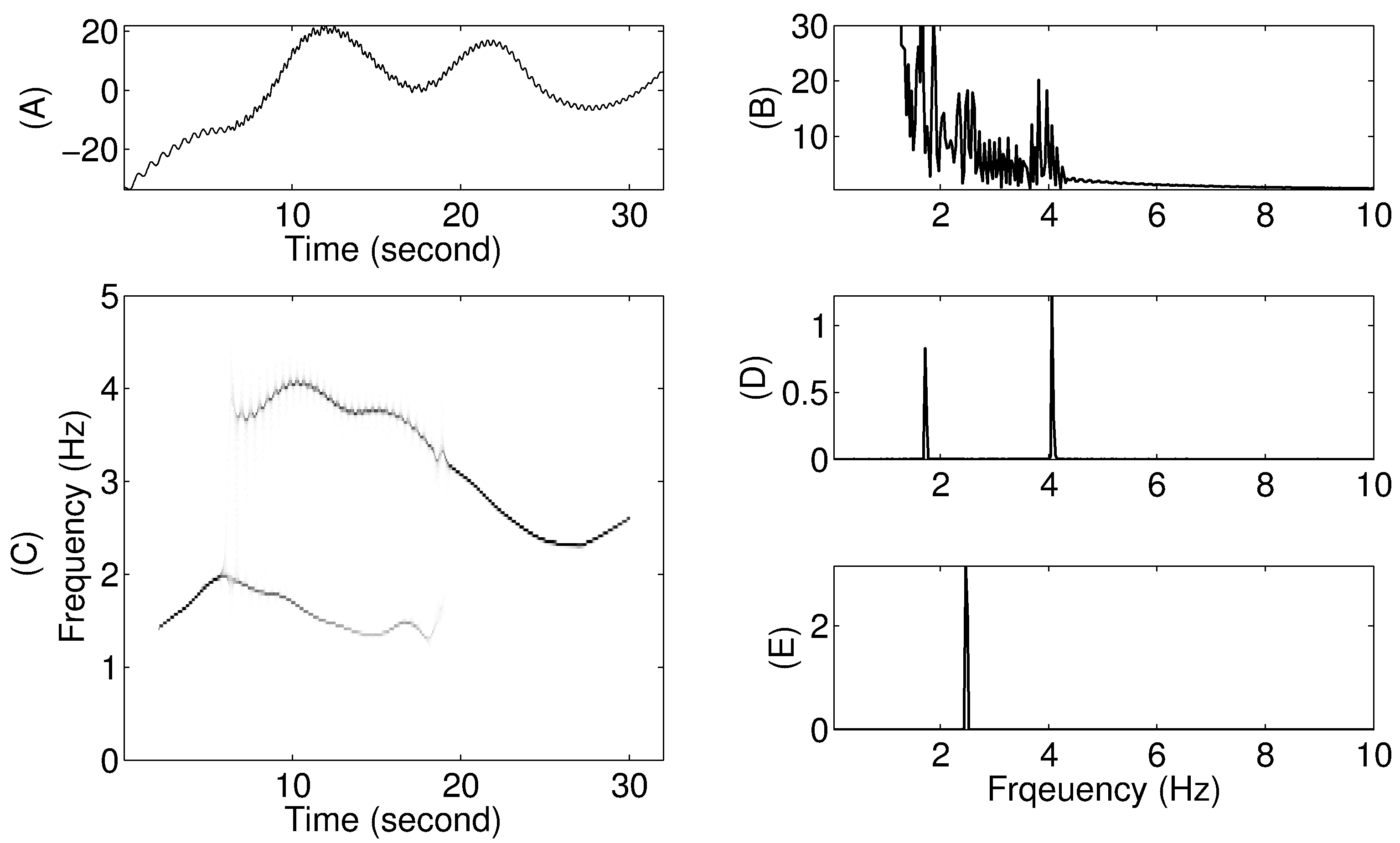}
\caption{The clean signal $f(t):=f_1(t)+f_2(t)+T(t)$ is shown in (A). Compared with the oscillatory signal, it is clear that the trend dominates the signal. In (B), we show the power spectrum of $f(t)$, where the huge amount of energy in the frequency range $[0,1]$ contributed by the trend is discarded to enhance the visualization. Note that the power spectrum is dominated by the trend, and the momentary behavior is difficult to see. The tvPS of $f(t)$ is shown in (C). Note that the dynamics of $f_1(t)$ and $f_2(t)$ can be clearly seen. In (D) and (E), we show the cross section of the tvPS at the $10^{th}$ second and the $24^{th}$ second, respectively. Note that the tvPS is concentrated at the frequency describing the momentary oscillatory behavior of each component, which capture the intuition of ``momentary power spectrum''. Also note that the influence of the trend $T(t)$ in the time-frequency analysis is minimized.}
\label{figS3}
\end{figure}

\begin{figure}[ht]
\includegraphics[width=1\textwidth]{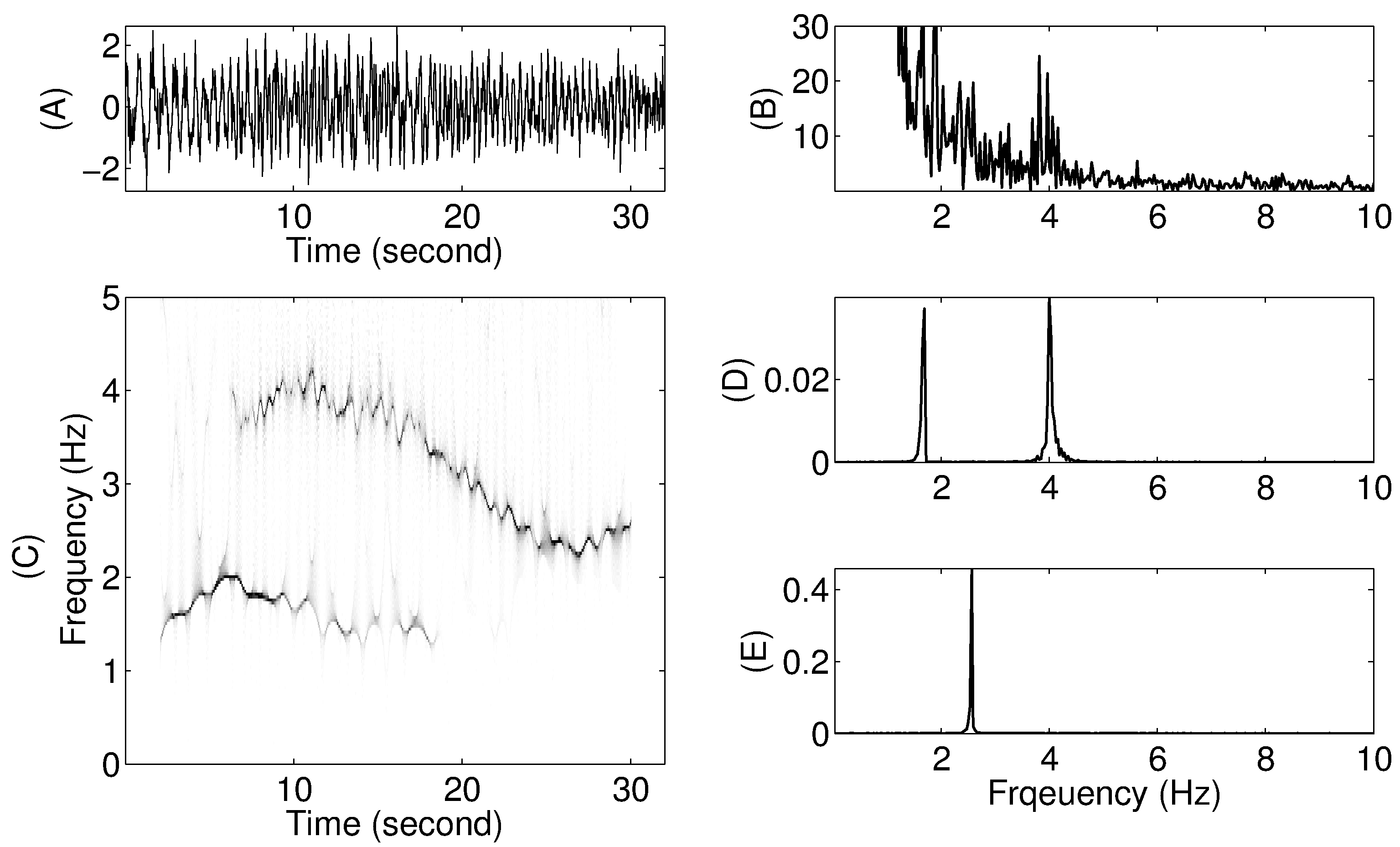}
\caption{The noisy signal $f_1(t)+f_2(t)+\Phi$, where $\Phi$ is the Gaussian white noise, is shown in (A). For the sake of emphasizing the amount of noise, the trend is not shown. In (B), we show the power spectrum of $Y=f_1(t)+f_2(t)+T(t)+\Phi$, and in (C), the tvPS of $Y$. (D) and (E) reveal the cross section of the tvPS at the $10^{th}$ second and the $24^{th}$ second, respectively. Note that although the observation is noisy and not too many features can be told from the observation, the dynamics of $f_1(t)$ and $f_2(t)$, like the appearance and disappearance of each component, and the instantaneous frequency, can still be observed. Again, the influence of the trend $T(t)$ in the time-frequency analysis is minimized.}   
\label{figS4}
\end{figure}

\subsection{Estimation of the wave shape functions}
Before closing this section, let us discuss a naive approach to estimate the wave shape functions as introduced in (\ref{decompShape}) in the adaptive non-harmonic model under the functional regression setup.
Consider the wave shape functions $s_j$, $j=1,\ldots,K$, with dominant ratio $\delta>0$, support $D\in\NN$ and accuracy $\theta>0$ in (\ref{decompShape}). To simplify our discussion, we assume that $\theta=0$, $\sigma=1$ and $T=0$. Thus, for a function $f$ that satisfies the adaptive non-harmonic model (\ref{decompShape}) can be represented by 
\begin{align}
f(t)=&\,\sum_{j=1}^KA_j(t)s_{j}(\phi_j(t))=\sum_{j=1}^KA_j(t)\sum_{\ell=1}^{D} \alpha_{j,\ell}\cos(2\pi \ell\phi_j(t))+\beta_{j,\ell}\sin(2\pi\ell\phi_j(t))\nonumber,
\end{align}
where $\alpha_{j,\ell}\in\RR$ and $\beta_{j,\ell}\in\RR$ are the Fourier coefficients of the shape function $s_j$. 
We will assume that the observation vector $\vY\in\RR^N$ satisfies
\begin{align}\label{model:with_shape_and_trend}
\vY(l)=f(l\Delta t)+\Phi_l,
\end{align}
where $l=1,\ldots, N$, $\Delta t>0$ is the sampling period and $\Phi$ is a random vector satisfying $\textup{var}(\Phi_l)=1$ for all $l$, which might not be Gaussian and the covariant matrix might not be the identity. 

The assumption $\phi'_1(t)<\phi'_k(t)$, for all $t\in\RR$ and $k>1$, enables us to estimate $A_1$ and $\phi_1$ at the sampling points $l\Delta t$, $l=1,\ldots,N$, from $\vY$ via SST with high accuracy; that is, the error is of order $\epsilon$ (the proof is the same as that in \cite{Daubechies_Lu_Wu:2011,Chen_Cheng_Wu:2013} and omitted here). Denote the estimation of $A_1(t)$ and $\phi_1(t)$ at the sampling points $l\Delta t$, $l=1,\ldots,N$, by $\widetilde{A}_1\in\RR^{1\times N}$ and $\widetilde{\phi}_1\in\RR^{1\times N}$. 
We construct the estimators $c_{1,\ell}\in\RR^{1\times N}$ and $d_{1,\ell}\in\RR^{1\times N}$ by
\begin{align}
c_{1,\ell}(l):=\widetilde{A}_1(l)\cos(2\pi\ell\widetilde{\phi}_j(l)),\quad d_{1,\ell}(l):=\widetilde{A}_1(l)\sin(2\pi\ell\widetilde{\phi}_j(l)),\nonumber
\end{align} 
where $l=1,\ldots, N$ and $\ell=1,\ldots,D$.
Next, consider the following ``functional vectors'' with components given by $c_{1,\ell}$ and $d_{1,\ell}$, namely 
\begin{align*}
&\boldsymbol{c}_{1}=[c_{1,1}^T,\ldots,c_{1,D}^T,d_{1,1}^T,\ldots,d_{1,D}^T]^T\in \RR^{2D\times N}
\end{align*}
The problem is to evaluate the parameters $\alpha_{1,\ell}$ and $\beta_{1,\ell}$, which are the Fourier coefficients of $s_1$, from the functional vectors $\boldsymbol{c}_{1}$. To facilitate our discussion, let us assume that the estimates $\widetilde{A}_1$ and $\widetilde{\phi}_1$ are precise without error; that is, $\widetilde{A}_1(l)=A_1(l\Delta t)$ and $\widetilde{\phi}_1(l)=\phi_1(l\Delta t)$ for all $l=1,\ldots,N$. Then (\ref{decompShape}) becomes
\begin{align}\label{model:with_shape_and_trend:regression_model}
\vY(l)=[\boldsymbol{\gamma}_1\boldsymbol{c}_1](l) +\sum_{j=2}^KA_j(l\Delta t)s_{j}(\phi_j(l\Delta t))+\Phi_l
\end{align}
for all $l=1,\ldots,N$, where $\boldsymbol{\gamma}_1=[\alpha_{1,1},\ldots,\alpha_{1,D},\beta_{1,1},\ldots,\beta_{1,D}]\in\RR^{1\times 2D}$. By Theorem \ref{thm:almost_orthogonal} which remains valid for the discretization setting, we have 
\[
\Delta t\vY\boldsymbol{c}^T_1  =\Delta t\boldsymbol{\gamma}_1\boldsymbol{c}_1 \boldsymbol{c}_1^T+\Delta t\Phi\boldsymbol{c}_1^T+O(\epsilon,(N\Delta t)^{-1}),
\]
where $\boldsymbol{c}^T_1$ is the transpose of $\boldsymbol{c}_1$. 
Note that since $\|f_k(t)\|_{L^2([0,N\Delta t])}=\|A_k\|_{L^2([0,N\Delta t])}\gg \epsilon$, we know the $2D\times 2D$ matrix $\boldsymbol{c}_1\boldsymbol{c}^T_1 $ is diagonal dominant. Also, by a direct evaluation, we have $\mathbb{E}(\Delta t\Phi\boldsymbol{c}_1^T)=0$ and $\text{var}(\Delta t\Phi\boldsymbol{c}_1^T)=O(L\Delta t)$. Thus we can now estimate the first shape function by using the estimator:
\begin{align}\label{estimator_shapefunction}
\widehat{\boldsymbol{\gamma}}_1:=(Y\boldsymbol{c}_1^T)(\boldsymbol{c}_1\boldsymbol{c}_1^T)^{-1},
\end{align}
where $\widehat{\boldsymbol{\gamma}}_1=[\widehat{\alpha}_{1,1},\ldots,\widehat{\alpha}_{1,D},\widehat{\beta}_{1,1},\ldots,\widehat{\beta}_{1,D}]^T\in\RR^{2D}$.
Indeed, by applying (\ref{estimator_shapefunction}), we are able to reconstruct the shape function by
\[
\widetilde{s}_1(t):=\sum_{\ell=1}^{D} \widehat{\alpha}_{1,\ell}\cos(2\pi \ell t)+\widehat{\beta}_{1,\ell}\sin(2\pi\ell t).
\]
Then we may iteratively estimate $A_j(t)$ and $\phi_j(t)$ for $j=2,\ldots,K$ and achieve the final result. Note that the success of the regression hinges on Theorem  \ref{thm:almost_orthogonal} and the estimation of $A_j(t)$ and $\phi_j(t)$.

The proposed approach to deal with the adaptive non-harmonic model is based on the functional least-square idea. Since handling the adaptive non-harmonic model is beyond the scope of this paper, a systematic study, particularly the detailed statistical analysis, will be reported in the near future. Here, we only point out a few possible directions. Firstly, in this subsection we assume the knowledge of $D$ and $J$. However, it might not be the case in practice. Furthermore, for different wave shape functions, the number of non-zero Fourier modes might be less than $D$. In this case, we may count on the hypothesis testing to determine the number of significant modes and the number of components. Secondly, the wave shape function model we consider is not the most general possible mode. In some cases we might have a more general, even discontinuous, periodic pattern. In this more general setting, we may apply another approach to recover the shape (see, for example, the diffeomorphism based spectral analysis \cite{Yang:2013} and singular value decomposition \cite{Hou_Shi_Tavallali:2013}).

\begin{figure}[h]
    \begin{center}
        \includegraphics[width=.6\textwidth]{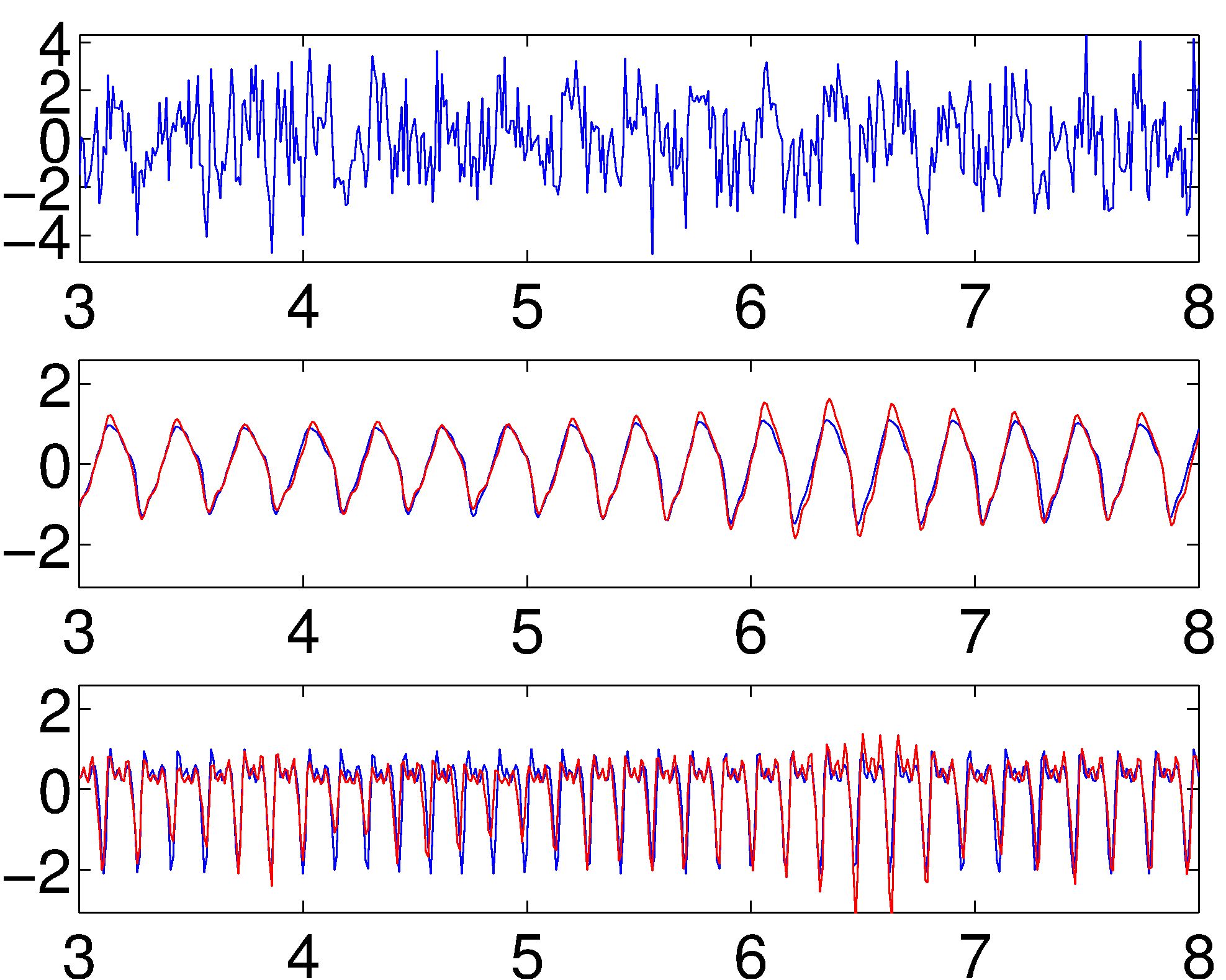}              
    \end{center}
    \caption{The result of estimating components with non-trivial shape. 
Top: The signal with two components $f=f_1+f_2$, where $f_1$ (resp. $f_2$) is non-harmonic oscillatory with shape $s_1$ (resp. $s_2$), contaminated by the $0$-dB Gaussian white noise. Middle: the curve in blue color is the clean signal $f_1$, while the red curve is the reconstructed $f_1$. Bottom: the curve in blue is the clean signal $f_2$, while the curve in red is the reconstructed $f_2$.}\label{simulation:fig1}
  \end{figure}

\section{Testbed: Anesthestic Depth Estimation}\label{section:testbed2}

Respiratory signals contain a wealth of information. A specific information that attracts more and more attention in recent years is the respiratory rate variability (BRV) \cite{Wysocki2006,Wu_Hseu_Bien_Kou_Daubechies:2013}.   
In addition to BRV, the analysis of other respiratory patterns shall facilitate the prediction of the patient's prognosis and the choice of the appropriate treatment in some cases \cite{somers_etc:2008}. We call these information extracted from the respiratory signal {\it respiratory dynamical features}.    
Physiologically, the neural respiratory control is located in the brain and comprises two system: the voluntary respiration and the involuntary respiration \cite{cite36}. In the brainstem, or more specifically the preB\"otzinger complex \cite{cite38}, the involuntary control center generates more regular oscillatory activity, whereas the voluntary respiration mediated by the control system in the forebrain presents an irregular activity. While competing with each other, the neural controls from these two systems are integrated to control the respiratory motor neuron. 

Anesthetics exerting differential effects on the central nerve system, particularly the human brain. It is known that during anesthesia, human respiration is more regular in deeper level of anesthesia, and more irregular in lighter level of anesthesia \cite{cite11}. This suggests that as the anesthesia depth increases, the anesthetics influences first the forebrain before the brain stem. Despite these early descriptions of the clinical finding, there is no available clinical instrument for quantifying this phenomenon in respiration. 

We may model these physiological findings by using the adaptive harmonic model introduced above -- when the anesthetic depth is sufficiently deep, the respiratory signal follows the model (\ref{decomp1}), but the model (\ref{decomp1}) is not suitable otherwise. In any case, this mathematical model allows us to apply the synchrosqueezing transform to evaluate its tvPS, with which we might be able to evaluate the anesthetic depth evolution during the anesthesia. 

In this section, we propose the {\it blending ECG derived respiration (EDR) algorithm} and apply the proposed real-time SST for estimating respiratory dynamics from the ECG signal in clinical anesthesia. The estimated dynamic is further shown to be closely related to the anesthetic depth.

\subsection{The blending EDR algorithm}\label{section:blending_EDR_algorithm}

We begin with applying the optimal real-time spline interpolation to obtain the respiratory signal from analyzing the morphology of the ECG signal, which is a commonly encountered non-uniform sampling dataset in the clinics.

The respiratory signal obtained from the ECG measurement is referred to as the {\it ECG-derived respiration (EDR)} signal. 
Physiologically, respiration induced ECG distortions occur for two different reasons. Firstly, the respiration-related mechanical changes affect the thoracic electrical impedance, 
and consequently the cardiac axis rotation occurring during the respiratory cycle has been shown to be the largest factor contributing to the distortion of ECG signals \cite{moody_mark:1986}. Secondly, respiration affects the heart rate variability (HRV), thereby causing respiratory sinus arrhythmia \cite{bailon_sornmo:2006}. Many algorithms have been developed to estimate the respiration on the basis of the above-mentioned physiological factors. Here we focus on analyzing the ECG morphology variation, which can be viewed as a surrogate of the non-uniform sampling of the respiratory signal. In particular, we consider the one-lead ECG signal, since it is a common setup in several clinical situations. 

Now we introduce the EDR algorithm. Given a one-lead ECG signal, for example, lead II or lead III, we first determine the amplitude of the observed R peaks. An R peak is a spiky peak in one heart beat observed in the ECG signal, as illustrated in the middle row of Figure \ref{fig:33} (marked by black crosses).
Based on the above mentioned physiological facts, the sampling times of the R peaks and their amplitudes form a non-uniform sampling of the respiratory signal. We comment that the well-known HRV behavior renders these samples irregularly. We than apply the blending operator to interpolate the R peak amplitude data to obtain the respiratory signal in real time. We refer this algorithm as the {\it blending EDR algorithm}, and we call the result the {\it blending EDR waveform}. 

\begin{enumerate}
\item[(P1)] Preprocessing the ECG signal: The median filter is applied to remove the wandering baseline artifact in the ECG signal, for example the lead II ECG signal, which may come from patient movement, dirty lead electrodes and a variety of other things. We choose the moving window of length $100$ ms, so that it is longer than the average length of the QRS complex. 
\item[(P2)] Time determination of the R peaks (or S peaks when the cardiac axis is deviated): In particular, if the cardiac axis is within the normal range of $-30^\circ$ to $90^\circ$ and the ECG signal shows the Rs pattern, we detect the timing of the R peaks; otherwise, we detect the timing of the S peaks with the condition of the rS pattern. If a beat is determined to be premature ventricular complex (PVC), it is deleted.
\item[(P3)] Construction of EDR signal: Build up the EDR signal by applying the blending operator on the amplitudes of the detected R peaks.
\end{enumerate}

The above proposed blending EDR algorithm is illustrated in Figure \ref{fig:33}. In summary, we mention that the median filter and the R peak detection can be carried out in real-time, so that by applying the blending operator, the blending EDR waveform can be obtained in real-time. Finally, we remark that traditionally researchers directly apply the cubic spline interpolation of the amplitude of the observed R peaks, which is not real-time.

\begin{figure}[!t]
\begin{center}
\subfigure{
\includegraphics[width= .6\textwidth]{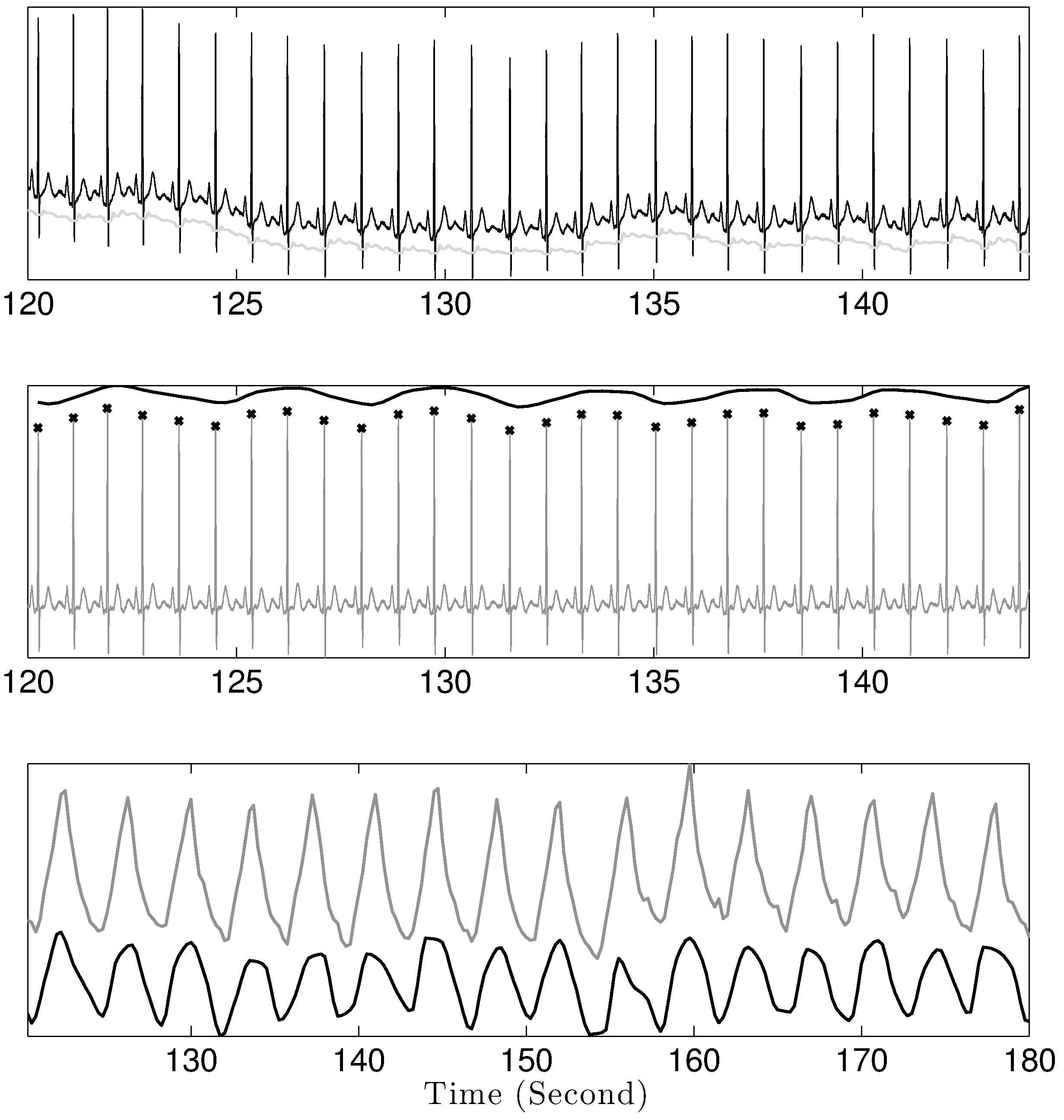}}
\end{center}
\caption{The blending EDR algorithm. In the upper subfigure, the lead II ECG signal is shown as a black curve and the wandering baseline of the lead II ECG signal is plotted as a light gray curve shifted below to enhance the visualization. In the middle subfigure, the median-filtered lead II ECG signal is plotted as a dark gray curve, superimposed with the R peaks marked as black crosses, and the blending EDR waveform is plotted as a black curve shifted up to increase the visualization. In the lower subfigure, the respiratory signal recorded from the chest band is plotted as the dark gray curve while the blending EDR waveform is plotted as the black curve. Note that we can find the oscillatory pattern in the blending EDR waveform, which oscillates similar to the true respiratory signal.} 
\label{fig:33}
\end{figure}

\subsection{Nonrhythmic to Rhythmic Ratio}
Let us first introduce the necessary terminologies for our purpose. If the signal oscillates in a way that can be quantified by the adaptive harmonic model (\ref{decompShape}), we say that the signal is {\it rhythmic}. To be more precise, we consider the following model for the rhythmic blending EDR waveform, namely, 
\begin{equation}\label{decompResp}
R(t) =  A(t)s(\phi(t))+T(t)+\sigma(t)\Phi,
\end{equation}
that is, with $K=1$ in (\ref{decompShape}). Note that the trend is important in the EDR algorithm. 
We call the signal which is not rhythmic {\it non-rhythmic}. A given time series could have a rhythmic behavior at one moment and non-rhythmic at another moment, and this is the dynamics we observe in clinics that we want to quantify. 
Precisely, as the qualitative description in the literature discussed above, the deeper the anesthetic depth, the more rhythmic the respiration should be, and the proposed index should capture this difference. The proposed index is aimed to quantify this qualitative observation.
We now discuss the index to quantify the respiratory dynamic features.


First, discretize the continuous blending EDR waveform $R(t)$ from the $0^{th}$ second to the $L^{th}$ second at $\eta$Hz, and evaluate the tvPS of $R(t)$ with the VM wavelet $\psi_{m,n}$ and the lag time $Q$ seconds, which is discretized as a $n_\xi\times L\eta$ matrix, denoted as $\boldsymbol{V}$, where $n_\xi$ is the number of discretization of the frequency axis of the tvPS. Here the resolution in the frequency axis is $\Delta\xi=\frac{\eta}{2n_\xi}$. With the tvPS, we can define the index called {\it Nonrhythmic to rhythmic ratio} (NRR), which was first proposed in \cite{Lin_Wu_Tsao_Yien_Hseu:2013} to study the dynamic hidden inside HRV. In a nutshell, NRR is defined as a ratio of nonrhythmic component to rhythmic component based on tvPS to differentiate the ``nonrhythmic-to-rhythmic'' pattern transition. We first define the {\it dominant frequency} $f_r$ as the time varying frequency associated with the maximal power on the tvPS, namely 
\begin{align}
f_r&=\argmax_{f_r\in Z_{n_\xi}^{L\eta}}\Big[ \sum_{m=1}^{L\eta} \log\left(\frac{|\boldsymbol{V}(f_r(m),m)|}{\sum_{i=1}^{n_\xi}\sum_{j=1}^{L\eta}|\boldsymbol{V}(j,i)|}\right)-\lambda\sum_{m=2}^{L\eta} |f_r(m)-f_r(m-1)|^2\Big], \label{algorithm:sst}
\end{align}
where $Z_{n_\xi}=\{1,\ldots,n_\xi\}$ and $\lambda$ is the penalty coefficient determined by the user. Note that when $R$ satisfies (\ref{decompResp}), then $\Delta\xi f_r(l)$, where $1\leq l\leq L\eta$ is simply the instantaneous frequency at time $l/\eta$ defined in the model \cite[Theorem 3.1 (ii)]{Chen_Cheng_Wu:2013}; otherwise we would obtain the region with highest energy distribution. 

Next we define the {\it rhythmic component power at time $l/\eta$}, where $1\leq l\leq L\eta$, denoted by $P_r(l)$, as the sum of the power inside the bands around $f_r(l)$ and its multiples on the tvPS, namely 
\[
P_r(l) := \sum_{k= \lfloor f_r(l) - 0.02/\Delta \xi\rfloor}^{\lceil f_r(l) + 0.02/\Delta\xi\rceil} \boldsymbol{V}(l, k),
\]
where the width of the band is chosen to be $0.02$Hz in an ad hoc way. 
The {\it nonrhythmic power at time $l/\eta$} is defined as the rhythmic power subtracted from high frequency power of $R(t)$; that is,
\[
P_{nr}(l):=\sum_{k= \lceil 0.1/\Delta \xi\rceil}^{n_\xi} \boldsymbol{V}(l, k) -P_r(l).
\] 
Finally, NRR at time $b$ is defined as the ratio of the nonrhythmic component power to the rhythmic component power at time $l/\eta$, namely 
\[
\NRR(l) = \log_{10} \left(\frac{P_{nr}(l)}{P_r(l)} \right).
\]
Clearly, when the respiration is rhythmic, $\NRR$ is small; while the respiration is nonrhythmic, $\NRR$ is large. In this study, we choose $\eta=4$, $Q=45$, $m=n=11$, $n_\xi=2000$ and $\lambda=0.5$. 

From the data analysis viewpoint, $\NRR$ is a time-varying feature extracted from a given oscillatory signal via tvPS, and plays a role of dimensionality reduction. We also mention that although NRR is useful for this anesthetic study, for a different study, we may need a different index derived from the tvPS.

\subsection{Correlations with Sevoflurane Concentration}
To quantify the clinical observation that human respiration is more regular in deeper level of anesthesia, we study the correlation between the NRR index and the inhaled anesthetics (sevoflurane) concentration. 


After approval from the Institutional Reveiw Board and obtaining individual written informed consent, we enrolled 31 patients in this study. Physiologic data, EEG data, ECG were recorded continuously and synchronously from standard anesthetic monitoring (HP agilent patient monitor system),  Bispectral Index (BIS) monitor (Aspect A-2000 BIS monitor version XP, Host Rev:3.21, smoothing window $15$ seconds; Aspect Medical Systems, Nattick, CA, USA) and the ECG recorder (MyECG E3-80; Micro-Star Intl Co., New Taipei City, Taiwan) respectively. The surgery and anesthesia were performed as usual. The inhaled and end-tidal anesthetic gas concentration detected by the gas analyzer on a Datex-Ohmeda S/5 anesthesia machine (GE Health Care, Helsinki, Finland) were also recorded. We focus on the data interval in wakening period from the start of adequate spontaneous breath to return of consciousness when sevoflurane concentration monotonically and continuously decreased. 

The NRR is evaluated from the blending EDR waveform derived from the offline lead II ECG data via its tvPS.
The inhaled sevoflurane concentration is evaluated by the estimated effect-site (brain) anesthetic gas concentration ($C_{\textup{eff}}$), which is derived from the end-tidal anesthetic gas concentration ($C_{\textup{et}}$) by the the following pharmacokinetic-pharmacodynamic modeling \cite{cite22}:
$$
\frac{\ud C_{\textup{eff}}}{\ud t} = K_{\textup{e0}} ( C_{\textup{et}}-C_{\textup{eff}} ),
$$
where the constant $K_{\textup{e0}}$ was defined as $0.20$/min for all data \cite{cite22}.

To evaluate the correlation between the NRR index and the inhaled anesthetics (sevoflurane) concentration, we employed the prediction probability ($\PK$ analysis). $\PK$ analysis is a standard statistic tool evaluating the performance of anesthetic depth index \cite{cite25}. Here we briefly summarize the $\PK$ analysis. Suppose $x$ is the indicator under analysis, for example, NRR, and $y$ is the outcome, for example, the concentration of the sevoflurane. Their relationship is described by the rank ordering of pairs of $(x,y)$, and there are five possible relationships:
\begin{enumerate}
\item $x$ and $y$ are {\it concordant} if pairs of $(x,y)$ are rank ordered in the same direction; 
\item $x$ and $y$ are {\it disconcordant} if pairs of $(x,y)$ are rank ordered in the reverse direction
\end{enumerate}
Three different ties in a pair of $(x,y)$ should be considered:
\begin{enumerate}
\item the tie in $x$. We call pairs of $(x,y)$ tie in $x$ if $x$'s are the same while $y$'s are different; 
\item the tie in $y$. We call pairs of $(x,y)$ tie in $y$ if $y$'s are the same while $x$'s are different;
\item the tie in both $x$ and $y$. We call pairs of $(x,y)$ tie in both $x$ and $y$ if $x$'s and $y$' are the same.
\end{enumerate} 
Among these five relationships, like the usual notion of correlation, in the $\PK$ analysis the concordance is desirable but not the disconcordance; unlike the usual notion of correlation, in the $\PK$ analysis the tie in $x$ is undesirable but we tolerate the tie in $y$ and the tie in both $x$ and $y$. To realize this notion, we define $\PK$ in the following way. Denote $P_c$, $P_d$ and $P_{tx}$ the respective probabilities that two pairs of $(x,y)$ independently drawn from the population with replacement are concordant, disconcordant and tie in $x$. Define 
$$
\PK=\frac{P_c+\frac{1}{2}P_{tx}}{P_c+P_d+P_{tx}}.
$$
We interpret $\PK$ value in the following way. A value of one means that the indicator always correctly predicts the observed depth of anesthesia, a value of $0.5$ means that the indicator predicts no better than $50/50$ chance, and a $\PK$ value less than $0.5$ means that the indicator predicts inversely. We mention that the difference between PK analysis and the commonly used ranking statistics, Spearman correlation, denoted as $R$, is that the tie in $y$ is overlooked in $\PK$ analysis; it is overlooked since the outcome $y$ (ex. awake vs. asleep) is usually coarse. 

The results were presented as weighted averages according to each patient's data length. 
As a result, NRR is well correlated with effect-site sevoflurane concentration. Indeed, 
during spontaneous breathing, the weighted  $\PK$ is $0.711\pm 0.021$. 
The $p$ value of this result is $<0.0001$ (A $p$ value less than $0.05$ is considered to be significant in this study). 
Note that this correlation indicates that the respiratory signal contains the dynamical information regarding the anesthesia.

The above preliminary results suggest that the nonrhythmic term and the rhythmic term of the proposed model  extract a hidden physiologic information: the relative strength of involuntary and voluntary respiratory controls varies under the influence the dynamic anesthetic effect to the brain. In conclusion, the real-time information of respiratory dynamics revealed in ECG is potential to be a real-time monitoring for clinician to administrate anesthetics. Hence the proposed algorithm could provide an immediate and continuing benefit to the patient.

\section{Discussion}


In addition to the preliminary result about the anesthetic depth evaluation, we mention some potential applications of the proposed algorithms. First, the popular algorithm empirical mode decomposition (EMD) \cite{Huang_Shen_Long_Wu_Shih_Zheng_Yen_Tung_Liu:1998} may benefit from it. Indeed, we may replace the cubic spline operator in the sifting process (the key ingredient of the EMD algorithm) by the blending operator so that the real-time implementation of EMD is possible. Second, SST can be applied to analyze the decomposed components (called the intrinsic mode function) to estimate the instantaneous frequency and other information in real-time. Note that if the intrinsic mode functions satisfy the adaptive harmonic model, the result is rigorously justified. A theoretical work toward understanding the sifting process, its real-time implementation, and a survey of its feasibility to data analysis will be reported in the future paper.
In particular, an exact and efficient computational scheme for the analytic VM wavelets
\[
\psi^*_{\ux,m;n,j}=\psi_{\ux,m;n,j}+i\mathcal{H}\psi_{\ux,m;n,j},
\]
with arbitrary knot sequence $\ux$ on bounded intervals will be derived (based on the Hilbert B-spline in \cite{Micchelli_Xu_Yu:2013}), to facilitate the wavelet analysis of the sifting process
\[
g(t)=\sum_{k=1}^Kg_k(t)+r(t)
\] 
for an arbitrarily given signal $g$ on a bounded interval, without the need of numerical integration (in computing the analytic intrinsic mode signal components $g_k^*=g_k+i\mathcal{H}g_k$). The reason, as already discussed at the end of Section \ref{section:analyticVMwavelet}, is that
\[
\langle g_k^*,\psi_{\ux,m;n,j}\rangle = \langle g_k,\psi^*_{\ux,m;n,j}\rangle.
\]

In Section \ref{section:SSTtvPS}, we already mentioned some applications of SST in the medical field. Here we briefly mention that SST itself can be applied to other fields as well. Examples include: adaptive seasonality estimation in the high frequency financial data \cite{Vatter_Wu_Chavez-Demoulin_Yu:2013}, non-stationary dynamics analysis in the financial system \cite{Guharay_Thakur_Goodman_Rosen_Houser:2013}, gearbox fault diagnosis \cite{Li_Liang:2012}, paleoclimiatic data analysis \cite{Thakur_Brevdo_Fuckar_Wu:2013}, seasonal behavior of diseases in epidemiology \cite{Chen_Cheng_Wu:2013}, etc. It is worth to noting that while combined with the blending operator and VM wavelets, SST may be applied to the problems where online monitoring of the system dynamics is important.

Note that in Algorithm \ref{alg:Tf}, we need the analytic representation of the VM wavelet $\psi_1$ by applying the Hilbert transform. We show the necessity of the Hilbert transform by illustrating with the following example. Suppose $f(t)=\cos(2\pi t)$; that is, we have a pure-tone harmonic function with frequency $1$, or an IMT function with a constant IF $1$ and constant AM $1$. By a direct calculation, we have, for any mother wavelet $\psi_1$, 
\[
W_f(a,b)=\langle f,\psi^{(a,b)}_1\rangle = \frac{1}{2}\sqrt{a}\hat{\psi}_1(a)e^{i2\pi b}+\frac{1}{2}\sqrt{a}\hat{\psi}_1(-a)e^{-i2\pi b},
\] 
and hence
\[
-i\partial_b W_f(a,b)=\pi\sqrt{a}\hat{\psi}_1(a)e^{i2\pi b}-\pi\sqrt{a}\hat{\psi}_1(-a)e^{-i2\pi b}.
\] 
It is clear that $\Omega_f(a,b)$ does not provide any information about the frequency $1$. However, if we consider $\psi^*_1$ as the analytic representation of $\psi_1$, we see that
\[
W_f(a,b)=\langle f,\psi^{*(a,b)}_1\rangle = \frac{1}{2}\sqrt{a}\hat{\psi^*}_1(a)e^{i2\pi b},\quad -i\partial_b W_f(a,b)=\pi\sqrt{a}\hat{\psi}^*_1(a)e^{i2\pi b},
\] 
which indicates that $\Omega_f(a,b)=1$ when $|W_f|>0$. 


There are several open problems left unanswered from the viewpoint of data analysis. For example, how is the number of components determined from a given observation, under the assumption that the clean signal satisfies the adaptive harmonic (non-harmonic) model? A possible solution is proposed in \cite{Chen_Cheng_Wu:2013} by using the hypothesis test. However, a systematic study is lacking. Even if we know the oracle about the number of the components, what is the proper notion of statistical accuracy of extracting the information when noise exist, and how to optimize the extraction accuracy? Also, when the noise is non-stationary, the above problems become more challenging. One possible approach to the above problems is  to model the AM and IF functions in terms of splines so that the adaptive harmonic (non-harmonic) model becomes parametric. Under this situation, the above problems could become more trackable and several traditional statistical techniques may be applied.


\section*{Acknowledgment}

The research of Charles K. Chui was supported by the U.S. Army Research Office under Grant \#W911NF-11-1-0426. Hau-Tieng Wu acknowledges support by AFOSR grant FA9550-09-1-0643. The authors are also indebted to Maryke van der Walt for carrying out the computations of the cubic VM wavelet example on a bounded interval in Section \ref{subsection:VM}.

\bibliographystyle{plain}
\bibliography{EDR,NRR}

\end{document}